\documentclass[leqno,11pt]{article}
\usepackage[utf8]{inputenc}
\usepackage[T1]{fontenc}
\usepackage{microtype}

\usepackage[letterpaper]{geometry}

\ifdefined\screenview
  \edef\mtht{\the\textheight}
  \edef\mtwd{\the\textwidth}
\fi

\usepackage[dvipsnames]{xcolor}

\usepackage{amsmath}
\usepackage{amsthm}
\usepackage{tikz-cd}
\usetikzlibrary{arrows} 
\tikzset{
  commutative diagrams/.cd, 
  arrow style=tikz, 
  diagrams={>=stealth}
}
\usetikzlibrary{matrix,decorations.pathreplacing,calc}

\usepackage{textcomp}
\usepackage[sb]{libertine}
\usepackage[varqu,varl]{zi4}%
\usepackage[libertine,bigdelims,vvarbb]{newtxmath}
\usepackage[supstfm=libertinesups,supscaled=1.2,raised=-.13em]{superiors}
\useosf

\usepackage[scr=boondox,cal=euler]{mathalfa}
\usepackage{marvosym}

\usepackage{slashed}
\usepackage{esint} % SLASHED INTEGRAL
\usepackage[english]{babel}

\usepackage{imakeidx}
\indexsetup{noclearpage}
\makeindex[intoc]

\usepackage{csquotes}
\usepackage[
  backend=biber,
  hyperref=true,
  backref=true,
  isbn=false,
  doi=true,
  natbib=true,
  eprint=true,
  useprefix=true,
  maxcitenames=99,
  maxbibnames=99,  
  maxalphanames=99, 
  minalphanames=99,
  safeinputenc,
  style=alphabetic,
  citestyle=alphabetic,
  block=space,
  datamodel=preamble/ext-eprint
]{biblatex}
\usepackage[
  bookmarksnumbered = true,
  hypertexnames = false,
  colorlinks    = true,
  citecolor     = gray,
  linkcolor     = gray,
  urlcolor      = gray,
  breaklinks
]{hyperref}

\DeclareFieldFormat{url}{%
  \href{#1}{\ComputerMouse}
}
\DeclareFieldFormat{doi}{%
  \mkbibacro{DOI}\addcolon\space
  \href{https://doi.org/#1}{#1}
}
\makeatletter
\DeclareFieldFormat{arxiv}{%
  arXiv\addcolon\space
  \href{http://arxiv.org/\abx@arxivpath/#1}{#1}
}
\makeatother
\DeclareFieldFormat{mr}{%
  MR\addcolon\space
  \href{http://www.ams.org/mathscinet-getitem?mr=MR#1}{#1}
}
\DeclareFieldFormat{zbl}{%
  Zbl\addcolon\space
  \href{http://zbmath.org/?q=an:#1}{#1}
}
\renewbibmacro*{eprint}{%
  \printfield{arxiv}%
  \newunit\newblock
  \printfield{mr}%
  \newunit\newblock
  \printfield{zbl}%
  \newunit\newblock
  \iffieldundef{eprinttype}
  {\printfield{eprint}}
  {\printfield[eprint:\strfield{eprinttype}]{eprint}}
}

\AtEveryBibitem{%
  \clearlist{publisher}%
}
\AtEveryBibitem{%
  \clearlist{address}%
}
\DeclareFieldFormat[article,inproceedings,inbook,incollection,thesis]{title}{\textit{#1}}
\renewbibmacro{in:}{}
\newcommand{\printreferences}{\printbibliography[heading=bibintoc]}

\usepackage[inline,shortlabels]{enumitem}

\usepackage{subcaption}

\usepackage[yyyymmdd]{datetime}

\usepackage{etoolbox}
\ifundef{\abstract}{}{\patchcmd{\abstract}%
    {\quotation}{\quotation\noindent\ignorespaces}{}{}}

\usepackage[super]{nth}

\usepackage{thmtools}

\numberwithin{equation}{section}

\renewcommand{\qedsymbol}{$\blacksquare$}

 \ifdefined\coloring
  \declaretheorem[numberlike=equation,shaded={bgcolor=Cyan!10}]{theorem}
  \declaretheorem[numbered=no,name=Theorem,shaded={bgcolor=Cyan!10}]{theorem*}
  \declaretheorem[numberlike=equation,name=Lemma,shaded={bgcolor=Cyan!10}]{lemma}
  \declaretheorem[numberlike=equation,name=Proposition,shaded={bgcolor=Cyan!10}]{prop}
  \declaretheorem[numberlike=equation,name=Corollary,,qed=\qedsymbol,shaded={bgcolor=Cyan!10}]{cor}
  \declaretheorem[numberlike=equation,name=Conjecture,shaded={bgcolor=Cyan!10}]{conjecture}
  \declaretheorem[numberlike=equation,name=Hypothesis,shaded={bgcolor=green!10}]{hypothesis}
  \declaretheorem[numberlike=equation,name=Situation,qed=$\times$,shaded={bgcolor=green!10},style=definition]{situation}
  \declaretheorem[numberlike=equation,name=Definition,style=definition,qed=$\diamondsuit$,shaded={bgcolor=green!10}]{definition}
  \declaretheorem[numberlike=equation,name=Notation,style=definition,shaded={bgcolor=green!10}]{notation}
  \declaretheorem[numberlike=equation,name=Data,style=definition,shaded={bgcolor=yellow!10}]{data}
  \declaretheorem[numberlike=equation,style=definition,qed=$\spadesuit$,shaded={bgcolor=yellow!10}]{example}
  \declaretheorem[numberlike=equation,style=definition,shaded={bgcolor=yellow!10}]{exercise}
  \declaretheorem[numberlike=equation,style=remark,qed=$\clubsuit$,shaded={bgcolor=yellow!10}]{remark}
  \declaretheorem[numberlike=equation,style=remark,name=Warning,shaded={bgcolor=red!10}]{warn}
  \declaretheorem[numberlike=equation,style=remark,shaded={bgcolor=yellow!10}]{convention}
  \declaretheorem[numberlike=equation,style=definition,shaded={bgcolor=yellow!10}]{question}
\else
  \declaretheorem[numberlike=equation,]{theorem}
  \declaretheorem[numbered=no,name=Theorem]{theorem*}
  \declaretheorem[numberlike=equation,name=Lemma]{lemma}
  \declaretheorem[numberlike=equation,name=Proposition]{prop}

  \declaretheorem[numberlike=equation,name=Situation,qed=$\times$,style=definition]{situation}
  \declaretheorem[numberlike=equation,name=Definition,style=definition,qed=$\bullet$]{definition}

  \declaretheorem[numberlike=equation,qed=$\spadesuit$,style=definition]{example}
  
  \declaretheorem[numberlike=equation,style=remark,qed=$\clubsuit$]{remark}

\fi

\def\makeautorefname#1#2{\AtBeginDocument{\expandafter\def\csname#1autorefname\endcsname{#2}}}
\makeautorefname{table}{Table}        
\makeautorefname{chapter}{Chapter}
\makeautorefname{section}{Section}
\makeautorefname{subsection}{Section}
\makeautorefname{subsubsection}{Section}
\makeautorefname{footnote}{Footnote}
\AtBeginDocument{\def\itemautorefname~#1\null{(#1)\null}}
\AtBeginDocument{\def\equationautorefname~#1\null{(#1)\null}}

\numberwithin{substep}{step}
\makeautorefname{step}{Step}
\makeautorefname{substep}{Step}

\makeautorefname{case}{Case}
\makeautorefname{substep}{Step}
\let\C\undefined

\usepackage{bm}
\usepackage{mathtools} % FOR PAIREDDELIMITERS
\usepackage{stmaryrd} % FOR SUPER LIE BRACKET

\DeclareFontFamily{U}{mathx}{\hyphenchar\font45}
\DeclareFontShape{U}{mathx}{m}{n}{
      <5> <6> <7> <8> <9> <10>
      <10.95> <12> <14.4> <17.28> <20.74> <24.88>
      mathx10
      }{}
\DeclareSymbolFont{mathx}{U}{mathx}{m}{n}
\DeclareFontSubstitution{U}{mathx}{m}{n}
\DeclareMathAccent{\widecheck}{0}{mathx}{"71}
\DeclareMathAccent{\wideparen}{0}{mathx}{"75}

\DeclareMathOperator{\Aut}{Aut}
\DeclareMathOperator{\aut}{\mathfrak{aut}}

\DeclareMathOperator{\Bij}{Bij}

\DeclareMathOperator{\Diff}{Diff}
\DeclareMathOperator{\End}{End}

\DeclareMathOperator{\GL}{GL}

\DeclareMathOperator{\gr}{gr}
\DeclareMathOperator{\HF}{\HF}

\DeclareMathOperator{\Hom}{Hom}

\DeclareMathOperator{\Map}{Map}

\DeclareMathOperator{\Pic}{Pic}

\DeclareMathOperator{\Res}{Res}

\DeclareMathOperator{\Sym}{Sym}
\DeclareMathOperator{\Tor}{Tor}

\DeclareMathOperator{\codim}{codim}
\DeclareMathOperator{\coim}{coim}
\DeclareMathOperator{\coker}{coker}

\DeclareMathOperator{\im}{im}
\DeclareMathOperator{\ind}{index}

\DeclareMathOperator{\rk}{rk}

\DeclareMathOperator{\tr}{tr}

\DeclarePairedDelimiter\paren{\lparen}{\rparen}
\DeclarePairedDelimiter\sqparen{[}{]}
\DeclarePairedDelimiter{\Abs}{\|}{\|}

\DeclarePairedDelimiter{\Inner}{\langle}{\rangle}
\DeclarePairedDelimiter{\abs}{\lvert}{\rvert}
\DeclarePairedDelimiter{\bracket}{\langle}{\rangle}

\DeclarePairedDelimiter{\set}{\lbrace}{\rbrace}
\def\({\left(}
\def\){\right)}
\def\<{\left\langle}
\def\>{\right\rangle}

\newcommand{\C}{{\mathbf{C}}}

\newcommand{\N}{{\mathbf{N}}}

\newcommand{\R}{\mathbf{R}}

\newcommand{\Span}[1]{\bracket{#1}}

\newcommand{\Vect}{\mathrm{Vect}}

\newcommand{\Z}{\mathbf{Z}}

\newcommand{\co}{\mskip0.5mu\colon\thinspace}

\newcommand{\defined}[2][\key]{\def\key{#2}\textbf{#2}\index{#1}}
\newcommand{\delbar}{\bar{\del}}

\newcommand{\del}{\partial}
\newcommand{\ev}{\mathrm{ev}}

\newcommand{\id}{\mathrm{id}}
\newcommand{\incl}{\hookrightarrow}
\newcommand{\inner}[2]{\bracket{#1, #2}}
\newcommand{\into}{\hookrightarrow}
\newcommand{\iso}{\cong}

\newcommand{\one}{\mathbf{1}}
\newcommand{\onto}{\twoheadrightarrow}

\newcommand{\pr}{\mathrm{pr}}
\newcommand{\qandq}{\quad\text{and}\quad}

\newcommand{\qand}{\quad\text{and}}

\newcommand{\qwithq}{\quad\text{with}\quad}

\renewcommand{\H}{\mathbf{H}}

\renewcommand{\O}{\mathrm{O}}

\renewcommand{\emptyset}{\varnothing}
\renewcommand{\epsilon}{\varepsilon}
\renewcommand{\setminus}{{\backslash}}

\renewcommand{\leq}{\leqslant}
\renewcommand{\geq}{\geqslant}

\newcommand{\superrigid}{\blacklozenge}

\newcommand{\Wedge}{\Lambda}

\makeatletter
\renewcommand*\env@matrix[1][*\c@MaxMatrixCols c]{%
  \hskip -\arraycolsep
  \let\@ifnextchar\new@ifnextchar
  \array{#1}}

\renewcommand\xleftrightarrow[2][]{%
  \ext@arrow 9999{\longleftrightarrowfill@}{#1}{#2}}
\newcommand\longleftrightarrowfill@{%
  \arrowfill@\leftarrow\relbar\rightarrow}
\makeatother

\newcommand{\rd}{{\rm d}}

\newcommand{\uV}{{\underline V}}

\newcommand{\bD}{{\mathbf{D}}}
\newcommand{\bE}{{\mathbf{E}}}
\newcommand{\bF}{{\mathbf{F}}}

\newcommand{\bJ}{{\mathbf{J}}}
\newcommand{\bK}{{\mathbf{K}}}
\newcommand{\bL}{{\mathbf{L}}}

\newcommand{\bQ}{{\mathbf{Q}}}

\newcommand{\sA}{\mathscr{A}}

\newcommand{\sC}{\mathscr{C}}

\newcommand{\sE}{\mathscr{E}}
\newcommand{\sF}{\mathscr{F}}
\newcommand{\sG}{\mathscr{G}}
\newcommand{\sH}{\mathscr{H}}

\newcommand{\sJ}{\mathscr{J}}
\newcommand{\sK}{\mathscr{K}}
\newcommand{\sL}{\mathscr{L}}
\newcommand{\sM}{\mathscr{M}}
\newcommand{\sN}{\mathscr{N}}
\newcommand{\sO}{\mathscr{O}}
\newcommand{\sP}{\mathscr{P}}
\newcommand{\sQ}{\mathscr{Q}}
\newcommand{\sR}{\mathscr{R}}
\newcommand{\sS}{\mathscr{S}}
\newcommand{\sT}{\mathscr{T}}
\newcommand{\sU}{\mathscr{U}}
\newcommand{\sV}{\mathscr{V}}
\newcommand{\sW}{\mathscr{W}}

\newcommand{\fd}{{\mathfrak d}}

\newcommand{\fn}{{\mathfrak n}}

\newcommand{\fV}{{\mathfrak V}}

\newcommand{\df}{{\rd f}}

\newcommand{\ubR}{{\underline \R}}

\newcommand{\bsJ}{\bm{\sJ}}
\newcommand{\bsM}{\bm{\sM}}

\DeclareMathOperator{\ord}{ord}
\newcommand{\op}{\mathrm{op}}
\DeclareMathOperator{\lcm}{lcm}

\usepackage{appendix}

\makeatletter
\newcommand{\subalign}[1]{%
  \vcenter{%
    \Let@ \restore@math@cr \default@tag
    \baselineskip\fontdimen10 \scriptfont\tw@
    \advance\baselineskip\fontdimen12 \scriptfont\tw@
    \lineskip\thr@@\fontdimen8 \scriptfont\thr@@
    \lineskiplimit\lineskip
    \ialign{\hfil$\m@th\scriptstyle##$&$\m@th\scriptstyle{}##$\hfil\crcr
      #1\crcr
    }%
  }%
}
\makeatother

\usepackage{chngcntr} 
\counterwithin{section}{part}

\usepackage[tocindentauto]{tocstyle}
\usetocstyle{standard}

\author{
  Aleksander Doan
  \and
  Thomas Walpuski
}
\title{  
  Equivariant Brill--Noether theory for elliptic operators and super-rigidity of $J$--holomorphic maps
}

\date{2020-06-01}

\begin{document}
\maketitle

\begin{abstract}  
  The space of Fredholm operators of fixed index is stratified by submanifolds according to the dimension of the kernel.
  Geometric considerations often lead to questions about the intersections of concrete families of elliptic operators with these submanifolds:
  are the intersections non-empty? are they smooth? what are their codimensions?
  The purpose of this article is to develop tools to address these questions in equivariant situations.  
  An important motivation for this work are transversality questions for multiple covers of $J$--holomorphic maps.
  As an application,
  we use our framework to give a concise exposition of Wendl’s proof of the super-rigidity conjecture.
\end{abstract}

\section*{Introduction}
\label{Sec_Introduction}

Let $X$ and $Y$ be two finite dimensional vector spaces.
The space $\Hom(X,Y)$ is stratified by the submanifolds
\begin{align*}
  \sH_r
  \coloneq
  \set{ L \in \Hom(X,Y) : \rk L = r }
\end{align*}
of codimension
\begin{equation*}
  \codim \sH_r = (\dim X-r)(\dim Y-r).
\end{equation*}
This generalizes to infinite dimensions as follows.
Let $X$ and $Y$ be two Banach spaces.
The space of Fredholm operators from $X$ to $Y$,
denoted by $\sF(X,Y)$,
is stratified by the submanifolds
\begin{equation*}
  \sF_{d,e} \coloneq \set{ L \in \sF(X,Y) : \dim\ker L = d \text{ and } \dim\coker L = e }
\end{equation*}
of codimension
\begin{equation*}
  \codim \sF_{d,e} = de.
\end{equation*}
In many geometric problems,
especially in the study of moduli spaces in algebraic geometry, gauge theory, and symplectic topology,
one is led to consider families of Fredholm operators $D\co \sP \to \sF(X,Y)$ parametrized by a Banach manifold $\sP$, and
to analyze the subsets $D^{-1}(\sF_{d,e})$.

The archetypal example is Brill--Noether theory in algebraic geometry.
Let $\Sigma$ be a closed, connected Riemann surface of genus $g$.
Denote by $\Pic(\Sigma)$ the Picard group of isomorphism classes of holomorphic line bundles $\sL \to \Sigma$.
Brill--Noether theory is concerned with the study of the subsets $G_d^r \subset \Pic(\Sigma)$,
called the Brill--Noether loci,
defined by
\begin{equation*}
  G_d^r
  \coloneq
  \set*{ [\sL] \in \Pic(\Sigma) : \deg(\sL) = d \text{ and } \dim H^0(\Sigma,\sL) = r+1 }.
\end{equation*}
The fundamental results of this theory deal with the questions of whether $G_d^r$ is non-empty, smooth, and of the expected codimension.

This connects to the previous discussion as follows.
Let $L$ be a Hermitian line bundle of degree $d$ over $\Sigma$.
Denote by $\sA(L)$ the space of unitary connections on $L$.
The complex gauge group $\sG^\C(L)$ acts on $\sA(L)$ and the quotient $\sA(L)/\sG^\C(L)$ is  biholomorphic to $\Pic^d(\Sigma)$, the component of $\Pic(\Sigma)$ parametrizing holomorphic line bundles of degree $d$.
Define the family of Fredholm operators
\begin{equation*}
  \delbar\co \sA(L) \to \sF(\Gamma(L),\Omega^{0,1}(\Sigma,L))
\end{equation*}
by assigning to every connection $A$ the Dolbeault operator $\delbar_A = \nabla_A^{0,1}$.
Set
\begin{equation*}
  \tilde G_d^r \coloneq \delbar^{-1}(\sF_{r+1,g-d+r}).
\end{equation*}
It follows from the Riemann--Roch Theorem and Hodge theory that the Brill--Noether loci can be described as the quotients
\begin{equation*}
  G_d^r = \tilde G_d^r/\sG^\C(L).
\end{equation*}
If $G_d^r$ is non-empty, then
\begin{equation*}
  \codim G_d^r = \codim \tilde G_d^r \leq (r+1)(g-d+r).
\end{equation*}
This is an immediate consequence of the definition of $\tilde G_d^r$ and $\codim \sF_{d,e} = de$.
Ideally, every $G_d^r$ is smooth of codimension $(r+1)(g-d+r)$.
This is not always true,
but \citet{Gieseker1982} proved that it holds for generic $\Sigma$;
see also \cite{Eisenbud1983,Lazarsfeld1986}.
Furthermore,
\citet{Kempf1971,Kleiman1972,Kleiman1974} proved that if $(r+1)(g-d+r) \leq g$,
then $G_d^r$ is non-empty.
For an extensive discussion of Brill--Noether theory in algebraic geometry we refer the reader to \cite{Arbarello1985}.

By analogy,
for a family of Fredholm operators $D\co \sP \to \sF(X,Y)$ one might ask:
\begin{enumerate}
\item
  \label{Q_Nonempty}
  When are the subsets $D^{-1}(\sF_{d,e})$ non-empty?  
\item
  \label{Q_Smooth}
  When are they smooth submanifolds of $\sP$? 
\item  
  \label{Q_Codimension}
  What are their codimensions?   
\end{enumerate}
Index theory and theory of spectral flow sometimes give partial results regarding \autoref{Q_Nonempty}.
A simple answer to \autoref{Q_Smooth} and \autoref{Q_Codimension} is that $D^{-1}(\sF_{d,e})$ is smooth and of codimension $de$ if the map $D$ is transverse to $\sF_{d,e}$.
However,
for many naturally occurring families of elliptic operators this condition does not hold.
For example,
if $D$ is a family of elliptic operators over a manifold $M$ and $\uV$ is a local system,
then the family $D^\uV$ of the elliptic operators $D$ twisted by $\uV$ often is not transverse to $\sF_{d,e}$ even if $D$ is.
Related issues arise for families of elliptic operators pulled back by a covering map $\pi \co \tilde M \to M$.
The purpose of this article is to give useful tools for answering \autoref{Q_Smooth} and \autoref{Q_Codimension} which apply to these equivariant situations.
This theory is developed in \autoref{Part_Theory}.

The issues discussed above are well-known to arise from multiple covers in the theory of $J$--holomorphic maps in symplectic topology.
In fact,
our motivation for writing this article came from trying to understand \citeauthor{Wendl2016}'s proof of \citeauthor{Bryan2001}'s super-rigidity conjecture for $J$--holomorphic maps \cite{Wendl2016}.
The theory developed in \autoref{Part_Theory} is essentially an abstraction of \citeauthor{Wendl2016}'s ideas,
some of which can themselves be traced back to \citet{Taubes1996b,Eftekhary2016}.
In \autoref{Part_Application} we use this theory to give a concise exposition of the proof of the super-rigidity conjecture.
The main results of \autoref{Part_Application} are contained in \cite{Wendl2016} and most of the proofs closely follow Wendl's approach.
There are, however, two key differences:
\begin{enumerate}
\item
  Our discussion consistently uses the language of local systems.
  This appears to us to be more natural for the problem at hand.
  It also avoids the use of representation theory and covering theory.
  In particular,
  there is no need to take special care of non-normal covering maps.
\item
  Our approach to dealing with branched covering maps is geometric:
  branched covering map between Riemann surfaces are reinterpreted as unbranched covering maps between orbifold Riemann surfaces.
  This is to be compared with \citeauthor{Wendl2016}'s analytic approach which uses suitable weighted Sobolev spaces on punctured Riemann surfaces.
  One feature of our approach is that it leads to a simple proof of the crucial index theorem;
  cf. \autoref{Sec_OrbifoldRiemannRoch} and \cite[Theorem 4.1]{Wendl2016}.
\end{enumerate}

We expect the theory developed in \autoref{Part_Theory} to have many applications outside of the theory of $J$--holomorphic maps.
In future work we plan to study transversality for multiple covers of calibrated submanifolds
in manifolds with special holonomy, such as associative submanifolds in $G_2$--manifolds and special Lagrangians in Calabi--Yau $3$--folds.

%%% Local Variables:
%%% mode: latex
%%% TeX-master: "EquivariantBrillNoetherSuperRigidity"
%%% End:

\paragraph{Acknowledgements}
This material is based upon work supported by  \href{https://www.nsf.gov/awardsearch/showAward?AWD_ID=1754967&HistoricalAwards=false}{the National Science Foundation under Grant No.~1754967},
\href{https://sloan.org/grant-detail/8651}{an Alfred P. Sloan Research Fellowship},
\href{https://sites.duke.edu/scshgap/}{the Simons Collaboration ``Special Holonomy in Geometry, Analysis, and Physics''}, and \href{https://www.simonsfoundation.org/simons-society-of-fellows/}{the Simons Society of Fellows}.

\newpage
\tableofcontents

\part{Equivariant Brill--Noether theory}
\label{Part_Theory}

Throughout this part,
let $(M,g)$ be a closed, connected, oriented Riemannian orbifold of dimension $\dim M = n$,
and
let $E$ and $F$ be Euclidean vector bundles of rank $\rk E = \rk F = r$ over $M$ equipped with orthogonal connections.%
\footnote{%
  \autoref{Rmk_OrbifoldsBranchedCovers} explains why we allow orbifolds.
  For the purposes of this article,
  the category of orbifolds is the one constructed by \citet{Moerdijk2002} via groupoids;
  see also \cite{Adem2007}.
  \cite[Section 5]{Lupercio2004} compares \citeauthor{Moerdijk2002}'s approach with the original approach via orbifold charts developed by \citet{Satake1956} and \citet{Thurston2002}.
  \cite[Section 3]{Shen2019} discusses differential operators and Sobolev spaces on orbifolds.
}
For $k \in \N_0$
denote by $W^{k,2}\Gamma(E)$ and $W^{k,2}\Gamma(F)$ the Sobolev completions of $\Gamma(E)$ and $\Gamma(F)$ with respect to the $W^{k,2}$--norm induced by the Euclidean metric and the connection on $E$ and $F$,
respectively.
Set $L^2\Gamma(E) \coloneq W^{0,2}\Gamma(E)$ and $L^2\Gamma(F) \coloneq W^{0,2}\Gamma(F)$.

\section{Brill--Noether loci}
\label{Sec_BrillNoetherLoci}

Let us begin by discussing the non-equivariant theory.

\begin{definition}
  \label{Def_FamilyOfEllipticOperators}
  Let $k \in \N_0$.
  A \defined{family of linear elliptic differential operators of order $k$} consists of
  a Banach manifold $\sP$ and
  a smooth map
  \begin{equation*}
    D \co \sP \to \sF(W^{k,2}\Gamma(E),L^2\Gamma(F))
  \end{equation*}
  such that for every $p \in \sP$ the operator $D_p \coloneq D(p)$ is the extension of a linear elliptic differential operator $\Gamma(E) \to \Gamma(F)$ of order $k$.%
  \footnote{%
    \label{Footnote_BanachManifolds}%
    Banach manifolds are assumed to be Hausdorff, paracompact, and separable.
    This is required in \autoref{Sec_CodimensionInBanachManifolds} where the Sard--Smale Theorem is used.
  }
\end{definition}

\begin{definition}
  \label{Def_BrillNoetherLoci}
  Let $(D_p)_{p \in \sP}$ be a family of linear elliptic differential operators.
  For
  $d,e \in \N_0$
  define the \defined{Brill--Noether locus} $\sP_{d,e}$ by
  \begin{equation*}
    \sP_{d,e}
    \coloneq
    \set[\big]{ p \in \sP : \dim \ker D_p = d \textnormal{ and } \dim\coker D_p = e }.
    \qedhere
  \end{equation*}
\end{definition}

\begin{remark}
  \label{Rmk_Index}
  Let $(D_p)_{p \in \sP}$ be a family of linear elliptic operators of index $i \in \Z$.
  If $\sP_{d,e} \neq \emptyset$, then $d-e = i$;
  in particular: $d \geq i$ and $e \geq -i$.
\end{remark}

The following elementary fact from the theory of Fredholm operators reduces the discussion to the finite-dimensional case. 

\begin{lemma}
  \label{Lem_LyapunovSchmidtReduction}
  Let $X$ and $Y$ be Banach spaces.
  For every $L \in \sF(X,Y)$ there is an open neighborhood $\sU \subset \sF(X,Y)$ and a smooth map $\sS \co \sU \to \Hom(\ker L,\coker L)$ such that
  for every $T \in \sU$ there are isomorphisms
  \begin{equation*}
    \ker T \iso \ker \sS(T) \qandq
    \coker T \iso \coker \sS(T);
  \end{equation*}
  furthermore, $\rd_L\sS\co T_L\sF(X,Y) \to \Hom(\ker L,\coker L)$ satisfies
  \begin{equation*}
    \rd_L\sS(\hat L)s
    =
    \hat L s \mod \im L.
  \end{equation*}
\end{lemma}

\begin{proof}
  Pick a complement $\coim L$ of $\ker L$ in $X$ and a lift of $\coker L$ to $Y$.
  With respect to the splittings
  $X = \coim L \oplus \ker L$
  and
  $Y = \im L \oplus \coker L$
  every $T \in \sF(X,Y)$ can be written as
  \begin{equation*}
    T
    =
    \begin{pmatrix}
      T_{11} & T_{12} \\
      T_{21} & T_{22}
    \end{pmatrix}.
  \end{equation*}  
  Choose an open neighborhood $\sU$ of $L$ in $\sF(X,Y)$ such that for every $T \in \sU$ the operator $T_{11}$ is invertible.
  Define $\sS\co \sU \to \Hom(\ker L,\coker L)$ by
  \begin{equation*}
    \sS(T)
    \coloneq
    T_{22} - T_{21}T_{11}^{-1}T_{12}.
  \end{equation*}
  A brief computation shows that for every $T \in \sU$
  \begin{equation*}
    \Phi
    T
    \Psi
    =
    \begin{pmatrix}
      \one & 0 \\
      0 & \sS(T)
    \end{pmatrix}
    \qwithq
    \Phi
    \coloneq
    \begin{pmatrix}
      T_{11}^{-1} & 0 \\
      -T_{21}T_{11}^{-1} & \one
    \end{pmatrix}
    \qandq
    \Psi
    \coloneq
    \begin{pmatrix}
      \one & -T_{11}^{-1}T_{12} \\
      0 & \one
    \end{pmatrix};
  \end{equation*}
  hence,
  $\ker T \iso \ker \sS(T)$
  and
  $\coker T \iso \coker \sS(T)$.
\end{proof}

\autoref{Lem_LyapunovSchmidtReduction} together with the Regular Value Theorem immediately imply the following.

\begin{theorem}
  \label{Thm_BrillNoetherLoci}
  Let $(D_p)_{p \in \sP}$ be a family of linear elliptic differential operators.
  Let $d,e \in \N_0$.
  If for every $p \in \sP_{d,e}$ the map $\Lambda_p \co T_p\sP \to \Hom(\ker D_p,\coker D_p)$ defined by
  \begin{equation*}
    \Lambda_p(\hat p)s \coloneq \rd_pD(\hat p) s \mod \im D_p
  \end{equation*}
  is surjective,
  then the following hold:
  \begin{enumerate}
  \item
    $\sP_{d,e}$ is a submanifold of codimension
    \begin{equation*}
      \codim \sP_{d,e} = de.
    \end{equation*}
  \item
    If $\sP_{d,e} \neq \emptyset$,
    then $\sP_{\tilde d,\tilde e} \neq \emptyset$ for every $\tilde d,\tilde e \in \N_0$ with $\tilde d \leq d, \tilde e \leq e$, and $\tilde d-\tilde e = d-e$.
    \qed
  \end{enumerate}
\end{theorem}

\begin{remark}
  \label{Rmk_BrillNoetherLoci_Complex}
  If $E$ and $F$ are Hermitian vector bundles and $(D_p)_{p \in \sP}$ is a family of complex linear elliptic differential operators,
  then the map $\Lambda_p$ factors through $\Hom_\C(\ker D_p,\coker D_p)$.
  Therefore,
  the hypothesis of \autoref{Thm_BrillNoetherLoci} cannot be satisfied (unless it holds trivially).
  Of course,
  this issue is rectified by replacing $\R$ with $\C$ throughout the above discussion.
\end{remark}

\begin{example}[Brill--Noether theory for holomorphic line bundles over a Riemann surface]
  \label{Ex_BrillNoetherHolomorphicLineBundlesRiemannSurface}
  Let $\Sigma$ be a closed, connected Riemann surface of genus $g$.
  Let $L$ be a Hermitian line bundle of degree $d$ over $\Sigma$.
  Denote by $\sA(L)$ the space of unitary connections on $L$.
  Define the family of complex linear elliptic differential operators
  \begin{equation*}
    \delbar\co \sA(L) \to \sF(W^{1,2}\Gamma(L),L^2\Omega^{0,1}(\Sigma,L))
  \end{equation*}
  by assigning to every connection $A$ the Dolbeault operator $\delbar_A \coloneq \nabla_A^{0,1}$.
  Let $A \in \sA(L)$.
  Denote by $\sL$ the holomorphic line bundle associated with $\delbar_A$.
  By Serre duality,
  \begin{equation*}
    \coker \delbar_A = H^1(\Sigma,\sL) \iso H^0(\Sigma,K_\Sigma\otimes_\C \sL^*)^*.
  \end{equation*}
  Since
  $\rd_A\delbar(a) = a^{0,1}$,
  the map
  $\Lambda_A$
  factors through the isomorphism
  $T_A\sA(L) = \Omega^1(\Sigma,i\R) \iso \Omega^{0,1}(\Sigma)$.
  Therefore, the adjoint of $\Lambda_A\co T_A\sA(L) \to \Hom_\C(\ker \delbar_A, \coker \delbar_A)$ is the composition of the Petri map
  \begin{equation}
    \label{Eq_PetriL}
    \varpi_\sL
    \co
    H^0(\Sigma,\sL) \otimes_\C H^0(\Sigma,K_\Sigma\otimes \sL^*) \to H^0(\Sigma,K_\Sigma)
  \end{equation}
  with the inclusion $H^0(\Sigma,K_\Sigma) \incl \Omega^{0,1}(\Sigma)$.
  In particular, $\Lambda_A$ is surjective if and only if $\varpi_\sL$ is injective.
  If $\varpi_\sL$ is injective for every $[\sL] \in \Pic^d(\Sigma)$,
  then
  \begin{equation*}
    \tilde G_d^r
    \coloneq
    \delbar^{-1}(\sF_{r+1,g-d+r})
  \end{equation*}
  is a complex submanifold of codimension $(r+1)(g-d+r)$; 
  therefore, so is the classical Brill--Noether locus
  \begin{equation*}
    G_d^r
    \coloneq
    \tilde G_d^r/\sG^\C(L)
    \iso
    \set*{
      \sL \in \Pic^d(\Sigma)
      :
      \begin{array}{l}
        \dim H^0(\Sigma,\sL) = r+1 \textnormal{ and} \\
        \dim H^1(\Sigma,\sL) = g-d+r
      \end{array}
    };
  \end{equation*}
  cf.~\cite[Lemma 1.6, Chapter IV]{Arbarello1985}.
\end{example}

This example motivates the following definitions,
which are particularly appropriate for first order operators appearing in geometric applications.

\begin{definition}
  \label{Def_Flexible}
  Let $U \subset M$ be an open subset.
  A family of linear elliptic differential operators $(D_p)_{p \in \sP}$ is \defined{flexible in $U$} if for every $p \in \sP$ and $A \in \Gamma(\Hom(E,F))$ supported in $U$
  there is a $\hat p \in T_p\sP$ such that
  \begin{equation*}
    \rd_p D(\hat p)s = As \mod \im D_p
  \end{equation*}
  for every $s \in \ker D_p$.
\end{definition}

\begin{definition}
  \label{Def_FormalAdjoint}
  Let $D\co \Gamma(E) \to \Gamma(F)$ be a linear differential operator.
  Set
  \begin{equation*}
    E^\dagger \coloneq E^*\otimes \Wedge^n T^* M \qandq
    F^\dagger \coloneq F^*\otimes \Wedge^n T^* M.
  \end{equation*}
  The \defined{formal adjoint} of $D$ is the linear differential operator $D^\dagger \co \Gamma(F^\dagger) \to \Gamma(E^\dagger)$ characterized by
  \begin{equation*}
    \int_M \Inner{s,D^\dagger t} = \int_M \Inner{Ds,t}.
  \end{equation*}
  Here $\Inner{\cdot,\cdot}$ denotes the canonical pairings $E\otimes E^\dagger \to \Wedge^nT^* M$ and $F\otimes F^\dagger \to \Wedge^nT^* M$.
\end{definition}

\begin{definition}
  \label{Def_PetrisCondition}
  The \defined{Petri map} $\varpi \co \Gamma(E)\otimes\Gamma(F^\dagger) \to \Gamma(E\otimes F^\dagger)$ is defined by
  \begin{equation*}
    \varpi(s\otimes t)(x) \coloneq s(x)\otimes t(x).
  \end{equation*}
  Let $U \subset M$ be an open subset.
  A linear elliptic differential operator $D\co \Gamma(E) \to \Gamma(F)$ satisfies \defined{Petri's condition in $U$} if the map
  \begin{equation*}
    \varpi_{D,U} \co \ker D \otimes \ker D^\dagger \to \Gamma(U,E\otimes F^\dagger) 
  \end{equation*}
  induced by the Petri map is injective.  
\end{definition}

\begin{prop}
  \label{Prop_Flexible+Petri=>Surjective}
  Let $(D_p)_{p \in \sP}$ be a family of linear elliptic differential operators and let $p \in \sP$.
  Let $U \subset M$ be an open subset.
  If $(D_p)_{p \in \sP}$ is flexible in $U$ and
  $D_p$ satisfies Petri's condition in $U$,
  then the map $\Lambda_p$ defined in \autoref{Thm_BrillNoetherLoci} is surjective.
\end{prop}

\begin{proof}
  Define the map $\ev_p \co \Gamma_c(U,\Hom(E,F)) \to \Hom(\ker D_p,\coker D_p)$ by
  \begin{equation*}
    \ev_p(A)s \coloneq As \mod \im D_p.
  \end{equation*}  
  $(D_p)_{p \in \sP}$ is flexible in $U$ if and only $\im \ev_p \subset \im \Lambda_p$.
  $D_p$ satisfies Petri's condition in $U$ if and only if $\ev_p$ is surjective.
  To see this, observe the following.
  The isomorphism $\ker D_p^\dagger \iso (\coker D_p)^*$ induces an isomorphism $\Hom(\ker D_p,\coker D_p)^* \iso \ker D_p \otimes \ker D_p^\dagger$.
  An element $B \in \ker D_p \otimes \ker D_p^\dagger$ annihilates $\im \ev_p$ if and only if
  \begin{equation*}
    \int_M \inner{\varpi(B)}{A} = 0
  \end{equation*}
  for every $A \in \Gamma(\Hom(E,F))$ supported in $U$;
  that is: $\varpi(B) = 0$ in $U$.
  Therefore,
  $(\im \ev_p)^\perp = \ker \varpi$.
\end{proof}

\begin{remark}  
  In \autoref{Ex_BrillNoetherHolomorphicLineBundlesRiemannSurface},
  $\im \Lambda_p = \im \ev_p$ (with $U=\Sigma$, and $\R$ replaced with $\C$).
  Therefore, $\Lambda_p$ being surjective is equivalent to Petri's condition.
  Furthermore, tracing through the isomorphisms identifies the restriction of Petri map $\varpi$ to $\ker \delbar_A\otimes_\C \ker \delbar_A^*$ with the Petri map $\varpi_\sL$.  
\end{remark}

Flexibility is not a rare condition.
In fact, typically the following condition holds,
which evidently implies flexibility.

\begin{definition}
  \label{Def_StronglyFlexible}
  Let $U \subset M$ be an open subset.
  A family of linear elliptic differential operators $(D_p)_{p \in \sP}$ is \defined{strongly flexible in $U$} if for every $p \in \sP$ there is a $C^0$--dense subset $\sH_p \subset \Gamma_c(U,\Hom(E,F))$ such that for every $A \in \sH_p$ there is a $\hat p \in T_p\sP$ such that
  \begin{equation*}
    \rd_p D(\hat p)s = As
  \end{equation*}
  for every $s \in \Gamma(E)$.
\end{definition}

Petri's condition up appears to be more subtle.
(By the uniqueness theorem for ordinary differential equations,
it holds for first operators linear elliptic differential operators on $1$--manifolds.
This is somewhat useful;
see, e.g., \cite{Eftekhary2019}.)
In \autoref{Sec_PetrisConditionRevisited},
we revisit Petri's condition and discuss an algebraic criterion due to \citeauthor{Wendl2016} for Petri's condition to hold away from a subset of infinite codimension.

\begin{remark}
  \label{Rmk_BrillNoetherLoci_EstimateCodimension}
  Assume the situation of \autoref{Thm_BrillNoetherLoci}.
  The following observations are useful in situations where the primary objective is to estimate the codimension of $\sP_{d,e}$.
  \begin{enumerate}
  \item
    \label{Rmk_BrillNoetherLoci_EstimateCodimension_Codimension>=Rank}
    Every $p \in \sP_{d,e}$ has an open neighborhood $\sU$ in $\sP$ such that $\sP_{d,e}\cap \sU$ is contained in a submanifold of codimension $\rk \Lambda_p$.
    
  \item
    \label{Rmk_BrillNoetherLoci_EstimateCodimension_PetriUpToRankRho}
    Let $\rho \in \N$ and let $U \subset M$ be an open subset.
    A linear elliptic differential operator $D\co \Gamma(E) \to \Gamma(F)$ satisfies \defined{Petri's condition up to rank $\rho$ in $U$}
    if for every non-zero $B \in \ker D\otimes\ker D^\dagger$ of rank at most $\rho$ the section $\varpi(B)$ does not vanish on $U$.
    (A \defined{simple tensor} is tensor of the form $v\otimes w$.
    Every tensor $B$ is a sum of simple tensors.
    The \defined{rank} of $B$ is the minimal number of simple tensors that sum to $B$.)
    If $D_p$ satisfies this condition and $(D_p)_{p\in \sP}$ is flexible,
    then
    \begin{equation*}
      \rk \Lambda_p \geq \min\set{\rho,d,e}\max\set{d,e}.
    \end{equation*}
    \begin{proof}
      Set $\sigma \coloneq \min\set{\rho,d,e}$.
      If $d \leq e$, then choose an injection $\R^\sigma \into \ker D_p$ and set $H \coloneq \Hom(\R^\sigma,\coker D_p)$;
      otherwise, choose a surjection $\coker D_p \onto \R^\sigma$ and set $H \coloneq \Hom(\ker D_p,\R^\sigma)$.
      In either case, composition defines a surjection
      \begin{equation*}
        \pi_p\co \Hom(\ker D_p,\coker D_p) \to H.
      \end{equation*}
      The subspace $\im \pi_p^* \subset \Hom(\ker D_p,\coker D_p)^* \iso \ker D_p \otimes \ker D_p^\dagger$ consists of elements of rank at most $\sigma \leq \rho$.
      The argument of the proof of \autoref{Prop_Flexible+Petri=>Surjective} thus shows that $\pi_p \circ \Lambda_p$ is surjective.
      Therefore, $\rk \Lambda_p \geq \dim H = \min\set{\rho,d,e}\max\set{d,e}$.      
      \pushQED{\hfill$\blacksquare$}
    \end{proof}
  \item
    \label{Rmk_BrillNoetherLoci_EstimateCodimension_PetriUpToRankThree}
    (This is due to \citet[Proof of Lemma 4.4]{Eftekhary2016}.)
    Let $U \subset M$ be a non-empty open subset.
    Every first order linear elliptic differential operator $D \co \Gamma(E) \to \Gamma(F)$ satisfies Petri's condition up to rank three in $U$.
    \begin{proof}
      Every $B \in \ker D\otimes \ker D^\dagger$ can be written as $B = s_1\otimes t_1 + \cdots + s_\rho\otimes t_\rho$ with $\rho \coloneq \rk B$, and $s_1,\ldots,s_\rho \in \ker D$ and $t_1,\ldots,t_\rho \in \ker D^\dagger$ linearly independent.
      If $\rho = 1$ and $\varpi(B) = 0$,
      then $s_1$ or $t_1$ vanishes on an open subset;
      hence, by unique continuation, $s_1 = 0$ or $t_1 = 0$: a contradiction.
      
      Henceforth, assume that $\rho \geq 2$.
      Define $\delta,\epsilon\co U \to \N_0$ by $\delta(x) \coloneq \dim\Span{s_1(x),\ldots,s_\rho(x)}$ and $\epsilon(x) \coloneq \dim\Span{t_1(x),\ldots,t_\rho(x)}$.
      By unique continuation,
      $\delta$ and $\epsilon$ are positive on a dense open subset.
      In fact, $\delta,\epsilon \geq 2$ on a dense open subset.
      To see this,
      observe that if $\delta = 1$ on a non-empty open subset,
      then there is a non-empty open subset $V \subset U$ and a function $f \in C^\infty(V)$ such that $s_1(x) = f(x)s_2(x)$ for every $x \in V$.
      Therefore, $\sigma(\rd f)s_2 = 0$ with $\sigma$ denoting the symbol of $D$.
      Since $D$ is elliptic,      
      $f$ must be constant:
      a contradiction to $s_1$ and $s_2$ being linearly independent.
      The same argument applies to $\epsilon$. 

      If $\rho = 2$,
      then there exists an $x \in U$ such that $\delta(x) = \epsilon(x) = 2$;
      therefore: $\varpi(B)$ does not vanish at $x$.
      If $\rho = 3$,
      then there is an $x \in U$ such that $\min\set{\delta(x),\epsilon(x)} \geq 2$.
      If $\delta(x) = \epsilon(x) = 3$,
      then $\varpi(B)$ evidently does not vanishing at $x$;
      otherwise,
      without loss of generality,
      $s_1(x)$ and $s_2(x)$ are linearly independent,
      and $s_3(x) = \lambda_1s_1(x) + \lambda_2s_2(x)$.
      In the latter case,
      \begin{equation*}
        \varpi(B)(x) = s_1(x) \otimes \paren*{t_1(x)+\lambda_1t_3(x)} + s_2(x)\otimes \paren*{t_2(x) + \lambda_2t_3(x)}
      \end{equation*}
      which cannot vanish because $\epsilon(x) \geq 2$.
      \pushQED{\hfill$\blacksquare$}
    \end{proof}
    There are examples of first order linear elliptic operators which fail to satisfy Petri's condition up to rank four;
    see \cite[Example 5.5]{Wendl2016} or \autoref{Prop_BInKerHatVarpi}.
    Finally, a brief warning:
    the preceding observation is false when $\R$ is replaced with $\C$ or $\H$.
    The analogue of Petri's condition only holds up to rank one in this case.
    (The issue is that $\sigma(\df)s_2 =0$ does not imply $\rd f = 0$ if $f$ takes values in $\C$ or $\H$.)
    \qedhere
  \end{enumerate}
\end{remark}

%%% Local Variables:
%%% mode: latex
%%% TeX-master: "EquivariantBrillNoetherSuperRigidity"
%%% End:

\section{Pulling back and twisting}
\label{Sec_PullingBackTwisting}

This section introduces two constructions which produce new linear elliptic operators from old ones:
pulling back by a covering map and twisting by a Euclidean local system.

\begin{definition}
  \label{Def_DPullback}
  Let $\pi\co \tilde M \to M$ be a covering map with $\tilde M$ connected.%
  \footnote{%
    \label{Footnote_CoveringMap}%
    An orbifold map $\pi\co \tilde M \to M$ is a covering map if
    every point $x$ in the topological space underlying $M$ has a neighborhood of the form $U/G$ with $U$ a $G$--manifold,
    $\pi^{-1}(U/G)$ also is of the form $\tilde U/G$ with $\tilde U$ a $G$--manifold,
    and $\pi$ induces a $G$--equivariant covering map $\tilde U \to U$;
    see \cites[Section 5.3]{Moerdijk2002}[Section 2.2]{Adem2007}.
  }
  Let $D\co \Gamma(E) \to \Gamma(F)$ be a linear differential operator.
  The \defined{pullback} of $D$ by $\pi$ is the linear differential operator
  \begin{equation*}
    \pi^*D \co \Gamma(\pi^*E) \to \Gamma(\pi^*F)
  \end{equation*}
  characterized by
  \begin{equation*}
    (\pi^*D)(\pi^*s) = \pi^*(Ds).
    \qedhere
  \end{equation*}
\end{definition}

\begin{remark}
  \label{Rmk_OrbifoldsBranchedCovers}
  If $\pi\co \tilde M \to M$ is a branched covering map of manifolds whose ramification locus is a closed submanifold of codimension two,
  then $\tilde M$ and $M$ can be equipped with orbifold structures and $\pi$ can be lifted to an unbranched covering map of orbifolds.
  \autoref{Sec_BranchedCoveringMapsAsOrbifoldCoveringMaps} discusses this construction in the case of Riemann surfaces.
\end{remark}

\begin{definition}
  \label{Def_EuclideanLocalSystem}
  A \defined{Euclidean local system on $M$} is a Euclidean vector bundle $\uV$ over $M$ together with a flat orthogonal connection.
\end{definition}

\begin{remark}
  Let $x_0 \in M$.
  Parallel transport induces a \defined{monodromy representation}
  \begin{equation*}
    \mu\co \pi_1(M,x_0) \to \O(V)
  \end{equation*}
  with $V$ denoting the fiber of $\uV$ over $x_0$.
  $\uV$ can be recovered from $\mu$ as follows.
  Denote by $\pi\co \tilde M \to M$ the universal cover.
  A choice of $\tilde x_0 \in \pi^{-1}(x_0)$ induces an anti-isomorphism between $\pi_1(M,x_0)$ and the deck transformation group $\Aut(\pi)$.
  Therefore, $\tilde M$ is a principal $\pi_1(M,x_0)$--bundle.
  A moment's thought shows that
  \begin{equation*}
    \uV \iso \tilde M \times_\mu V.
  \end{equation*}
  This sets up a bijection between gauge equivalence classes $[\uV]$ of Euclidean local systems of rank $r$ and equivalence classes $[\mu]$ of representations $\pi_1(M,x_0) \to \O(r)$ up to conjugation by $\O(r)$.
  For a more detailed discussion---in particular, of how the to interpret the above in the category of orbifolds---we refer the reader to \cite[Sections 2.4 and 2.5]{Shen2019}.
\end{remark}

\begin{definition}
  \label{Def_DTwist}
  Let $D \co \Gamma(E) \to \Gamma(F)$ be a linear differential operator.
  Let $\uV$ be a Euclidean local system on $M$.
  The \defined{twist} of $D$ by $\uV$ is the linear differential operator
  \begin{equation*}
    D^\uV \co \Gamma(E\otimes\uV) \to \Gamma(F\otimes\uV)
  \end{equation*}
  characterized as follows:
  if $U$ is a open subset $M$, $s \in \Gamma(U,E)$, and $f \in \Gamma(U,\uV)$ is constant,
  then
  \begin{equation*}
    D^\uV (s\otimes f) = (Ds)\otimes f.
    \qedhere
  \end{equation*}
\end{definition}

The following shows that the pullback $\pi^*D$ is equivalent to the twist $D^\uV$ for a suitable choice of $\uV$.

\begin{definition}
  Let $G$ be a group and let $H < G$ be a subgroup.
  The \defined{normal core} of a $H$ is the normal subgroup
  \begin{equation*}
    N \coloneq \bigcap_{g \in G} gHg^{-1}.
    \qedhere
  \end{equation*}
\end{definition}

\begin{prop}
  \label{Prop_Pullback=Twist}
  Let $\pi\co \tilde M \to M$ be a finite covering map with $\tilde M$ connected.
  Let $x_0 \in M$ and $\tilde x_0 \in \pi^{-1}(x_0)$.
  Denote by
  \begin{equation*}
    C \coloneq \pi_*\pi_1(\tilde M,\tilde x_0) < \pi_1(M,x_0)
  \end{equation*}
  the characteristic subgroup of $\pi$ and by $N$ its normal core.
  Set $S \coloneq \pi_1(M,x_0)/C$.
  Denote by $\ubR$ the trivial rank one local system on $\tilde M$.
  Set
  \begin{equation*}
    \uV \coloneq \pi_*\ubR.
  \end{equation*}
  Let $D \co \Gamma(E) \to \Gamma(F)$ be a linear differential operator.
  The following hold:
  \begin{enumerate}
  \item
    \label{Prop_Pullback=Twist_Monodromy}
    The monodromy representation of $\uV$ factors through $G \coloneq \pi_1(M,x_0)/N$;
    indeed, it is induced by the representation of $G$ on $\Map(S,\R)$ defined by
    \begin{equation*}
      (\mu_gf)(s) \coloneq f(g^{-1}s).
    \end{equation*}
  \item
    \label{Prop_Pullback=Twist_Isomorphism}
    There are isomorphisms $\pi_*\co \Gamma(\pi^*E) \iso \Gamma(E\otimes\uV)$ and $\pi_*\co \Gamma(\pi^*F) \iso \Gamma(F\otimes\uV)$ such that
    \begin{equation*}
      D^{\uV} = \pi_* \circ \pi^*D \circ \pi_*^{-1}.
    \end{equation*}
  \end{enumerate}
\end{prop}

\begin{remark}
  \label{Rmk_PoincareTheorem}
  If $\pi$ is a normal covering, then $C=N$ and $G = \pi_1(M,x_0)/N$ is anti-isomorphic to its deck transformation group.    
  If $\pi$ has $k$ sheets, then $C$ has index $k$.
  Its normal core has index at most $k!$ by an elementary result known as Poincaré's Theorem.
  This theorem follows from the observation that the kernel of the canonical homomorphism $\pi_1(M,x_0) \to \Bij(G/C)$ is precisely $N$ and $\Bij(G/C) \iso S_k$.
\end{remark}

\begin{proof}[Proof of \autoref{Prop_Pullback=Twist}]
  The monodromy representation $\mu\co \pi_1(M,x_0) \to \O(V)$ of $\uV$ is trivial on $C$;
  hence, it factors through $G$.
  Denote by $\rho\co (\hat M,\hat x_0) \to (M,x_0)$ the pointed covering map with characteristic subgroup $N$.
  $\hat M$ is a principal $G$--bundle and $\tilde M \iso \hat M \times_G S$.
  This implies the assertion about the monodromy representation.

  For every vector bundle $\tilde E$ over $\tilde M$ there is a canonical isomorphism $\Gamma(\tilde E) \iso \Gamma(\pi_*\tilde E)$.
  For every vector bundle $E$ over $M$ there is a canonical isomorphism
  \begin{equation*}
    \pi_*\pi^* E \iso \pi_*(\pi^* E\otimes\ubR) \iso E \otimes\pi_*\ubR = E\otimes\uV.
  \end{equation*}
  Denote the resulting isomorphism $\Gamma(\pi^*E) \iso \Gamma(E\otimes\uV)$ by $\pi_*$.
  For $s \in \Gamma(E)$ and $f \in C^\infty(\tilde M)$
  \begin{equation*}
    \pi_*((\pi^*s)f) = s \otimes \pi_*f.
  \end{equation*}
  Let $U$ be an open subset of $M$, $s \in \Gamma(U,E)$, and $f \in \Gamma(U,\uV)$.
  Suppose that $f$ is constant.
  This is equivalent to the corresponding function $\tilde f \coloneq (\pi_*)^{-1}f$ on $\tilde U \coloneq \pi^{-1}(U)$ being locally constant.
  By the characterizing properties of $D^\uV$ and $\pi^*D$ and since $\pi^*D$ is a differential operator,
  \begin{equation*}
    D^\uV (s\otimes f) = (Ds)\otimes f
  \end{equation*}
  and
  \begin{equation*}
    (\pi^*D)(\pi_*)^{-1} (s\otimes f)
    =
    (\pi^*D)(\pi^*s \cdot \tilde f) 
    =
    \pi^*(Ds)\cdot \tilde f 
    =
    (\pi_*)^{-1}\paren*{Ds\otimes f}.
  \end{equation*}
  This proves that $D^{\uV} = \pi_* \circ \pi^*D \circ \pi_*^{-1}$.
\end{proof}

%%% Local Variables:
%%% mode: latex
%%% TeX-master: "EquivariantBrillNoetherSuperRigidity"
%%% End:

\section{Equivariant Brill--Noether loci, I}
\label{Sec_EquivariantBrillNoetherLoci_Twists}

Pulling back and twisting lead to families of linear elliptic differential operators which fail to satisfy the hypotheses of \autoref{Thm_BrillNoetherLoci} (except for a few corner cases).
In this section we formulate a variant of this result which applies to families of twisted linear elliptic differential operators.
Throughout this section, assume the following.

\begin{situation}
  \label{Sit_V}
  Let $\fV = (\uV_\alpha)_{\alpha=1}^m$ be a finite collection of irreducible Euclidean local systems which are pairwise non-isomorphic.
  For every $\alpha=1,\ldots,m$ denote by $\bK_\alpha$ the algebra of parallel endomorphisms of $\uV_\alpha$ and set $k_\alpha \coloneq \dim_\R \bK_\alpha$.
\end{situation}

\begin{remark}
  Since $\uV_\alpha$ is irreducible,
  $\bK_\alpha$ is a division algebra;
  hence, by Frobenius' Theorem it is (isomorphic to) either $\R$, $\C$, or $\H$
  and $k_\alpha \in \set{1,2,4}$.
\end{remark}

If $D$ is a linear elliptic differential operator,
then the twists $D^{\uV_\alpha}$ commute with the action of $\bK_\alpha$.
Therefore,
$\ker D^{\uV_\alpha}$ and $\coker D^{\uV_\alpha}$ are left $\bK_\alpha$--modules and,
hence, right $\bK_\alpha^\op$--modules.
Here $\bK_\alpha^\op$ denotes the opposite algebra of $\bK_\alpha$.

\begin{definition}
  \label{Def_EquivariantBrillNoetherLoci_Twists}
  Let $(D_p)_{p \in \sP}$ be a family of linear elliptic differential operators.
  For $d,e \in \N_0^m$
  define the \defined{$\fV$--equivariant Brill--Noether locus} $\sP_{d,e}^\fV$ by
  \begin{equation*}
    \sP_{d,e}^\fV
    \coloneq
    \set[\big]{ p \in \sP : \dim_{\bK_\alpha}\ker D_p^{\uV_\alpha} = d_\alpha \text{ and } \dim_{\bK_\alpha}\coker D_p^{\uV_\alpha} = e_\alpha \text{ for every } \alpha = 1,\ldots,m }.
    \qedhere
  \end{equation*}  
\end{definition}

\begin{remark}
  \label{Rmk_Index_Twist}
  Let $i \in \Z^m$.
  Let $(D_p)_{p \in \sP}$ be a family of linear elliptic operators such that $\ind_{\bK_\alpha} D_p^{\uV_\alpha} = i_\alpha$ for every $p \in \sP$ and $\alpha = 1,\ldots,m$.
  If $\sP_{d,e}^\fV \neq \emptyset$, then $d-e = i$;
  in particular: $d_\alpha \geq i_\alpha$ and $e_\alpha \geq -i_\alpha$.
    
  If $M$ is a manifold,  
  then
  \begin{equation*}
    \ind_{\bK_\alpha} D_p^{\uV_\alpha} = \rk_{\bK_\alpha} \uV_\alpha \cdot \ind D_p
  \end{equation*}
  by the Atiyah--Singer index theorem;
  therefore, the $i_\alpha$ all have the same sign.
  If $M$ is an orbifold,
  there are corrections terms in the index formula which spoil this relation between the indices;
  see, e.g., \autoref{Prop_IndexOfNormalCROperatorOfBranchedCover}.
\end{remark}

\autoref{Lem_LyapunovSchmidtReduction} immediately implies the following.

\begin{theorem}
  \label{Thm_EquivariantBrillNoetherLoci_Twists}
  Let $(D_p)_{p \in \sP}$ be a family of linear elliptic differential operators.
  Let $d,e \in \N_0^m$.
  If for every $p \in \sP_{d,e}^\fV$ the map $\Lambda_p^\fV \co T_p\sP \to \bigoplus_{\alpha=1}^m \Hom_{\bK_\alpha}(\ker D_p^{\uV_\alpha},\coker D_p^{\uV_\alpha})$ defined by
  \begin{equation*}
    \Lambda_p^\fV(\hat p) \coloneq \bigoplus_{\alpha=1}^m \Lambda_p^\alpha(\hat p) \qandq
    \Lambda_p^\alpha(\hat p)s \coloneq \rd_pD^{\uV_\alpha}(\hat p) s \mod \im D_p^{\uV_\alpha}
  \end{equation*}
  is surjective,
  then the following hold:
  \begin{enumerate}
  \item
    $\sP_{d,e}^\fV$ is a submanifold of codimension
    \begin{equation*}
      \codim \sP_{d,e}^\fV = \sum_{\alpha=1}^m k_\alpha d_\alpha e_\alpha.
    \end{equation*}
  \item
    If $\sP_{d,e}^\fV \neq \emptyset$,
    then $\sP_{\tilde d,\tilde e}^\fV \neq \emptyset$ for every $\tilde d,\tilde e \in \N_0^m$ with $\tilde d \leq d, \tilde e \leq e$, and $\tilde d-\tilde e = d-e$.
    \qed
  \end{enumerate}
\end{theorem}

\begin{remark}
  If $E$ and $F$ are Hermitian vector bundles and $(D_p)_{p \in \sP}$ is a family of complex linear elliptic differential operators,
  then \autoref{Thm_EquivariantBrillNoetherLoci_Twists} does not apply; cf.~\autoref{Rmk_BrillNoetherLoci_Complex}.
  Again, this issue is rectified by replacing $\R$ with $\C$ throughout.
  In fact, this somewhat simplifies the discussion since $\C$ is the unique complex division algebra;
  hence, there is no need to introduce $\bK_\alpha$.
\end{remark}

\begin{remark}
  \label{Rmk_EquivariantBrillNoetherLoci_LocalSystem}
  Every Euclidean local system $\uV$ decomposes into irreducible local systems
  \begin{equation*}
    \uV \iso \bigoplus_{\alpha=1}^m \uV_\alpha^{\oplus \ell_\alpha}
  \end{equation*}
  with $\ell_1,\ldots,\ell_m \in \N_0$ for a suitable choice of $\fV$.
  For every $\bar d,\bar e \in \N_0$
  the Brill--Noether locus
  \begin{align*}
    \sP_{\bar d,\bar e}^\uV
    \coloneq
    \set*{ p \in \sP : \dim\ker D_p^\uV = \bar d \text{ and } \dim \coker D_p^\uV = \bar e }
  \end{align*}
  is the finite disjoint union of the subsets $\sP_{d,e}^\fV$ with $(d,e) \in \N_0^m\times \N_0^m$ satisfying
  \begin{equation*}
    \sum_{\alpha=1}^m \ell_\alpha k_\alpha d_\alpha = \bar d
    \qandq
    \sum_{\alpha=1}^m \ell_\alpha k_\alpha e_\alpha = \bar e.
  \end{equation*}
  Through this observation \autoref{Thm_EquivariantBrillNoetherLoci_Twists} can be brought to bear on families of linear elliptic differential operators twisted by $\uV$.
\end{remark}

\autoref{Def_Flexible},
\autoref{Def_PetrisCondition},
and \autoref{Prop_Flexible+Petri=>Surjective}
have the following analogues in the present situation.

\begin{definition}
  \label{Def_VEquivariantlyFlexible}
  A family of linear elliptic differential operators $(D_p)_{p \in \sP}$ is \defined{$\fV$--equivariantly flexible in $U$} if for every $p \in \sP$ and $A \in \Gamma(\Hom(E,F))$ supported in $U$
  there is a $\hat p \in T_p\sP$ such that
  \begin{equation*}
    \rd_p D^{\uV_\alpha}(\hat p)s = (A\otimes \id_{\uV_\alpha})s \mod \im D_p^{\uV_\alpha}
  \end{equation*}
  for every $\alpha = 1,\ldots,m$ and $s \in \ker D_p^{\uV_\alpha}$.
\end{definition}

\begin{prop}
  \label{Prop_StronglyFlexible=>VEquivariantlyFlexible}
  If $(D_p)_{p \in \sP}$ is strongly flexible in $U$,
  then it is $\fV$--equivariantly flexible in $U$.
  \qed
\end{prop}

\begin{definition}
  \label{Def_VEquivariantPetriCondition}
  The \defined{$\fV$--equivariant Petri map}
  \begin{equation*}
    \varpi^\fV \co \bigoplus_{\alpha=1}^m \Gamma(E\otimes \uV_\alpha)\otimes_{\,\bK_\alpha^\op} \Gamma(F^\dagger \otimes \uV_\alpha^*) \to \Gamma(E\otimes F^\dagger)
  \end{equation*}
  is defined by $\varpi^\fV \coloneq \sum_{\alpha=1}^m \varpi_\alpha$ with $\varpi_\alpha$ denoting the composition of the Petri map
  \begin{equation*}
    \varpi_\alpha\co \Gamma(E\otimes \uV_\alpha)\otimes_{\,\bK_\alpha^\op} \Gamma(F^\dagger \otimes \uV_\alpha^*)  \to \Gamma(E\otimes F^\dagger \otimes \uV_\alpha \otimes_{\,\bK_\alpha^\op} \uV_\alpha^*)
  \end{equation*}
  and the map induced by
  \begin{equation*}
    \tr\co \uV_\alpha \otimes_{\,\bK_\alpha^\op} \uV_\alpha^* \to \ubR.
  \end{equation*}
  Let $U \subset M$ be an open subset.
  A linear elliptic differential operator $D\co \Gamma(E) \to \Gamma(F)$ satisfies the \defined{$\fV$--equivariant Petri condition in $U$} if the map
  \begin{equation*}
    \varpi_{D,U}^\fV
    \co
    \bigoplus_{\alpha=1}^m \ker D_p^{\uV_\alpha} \otimes_{\,\bK_\alpha^\op} \ker D_p^{\dagger,\uV_\alpha^*}
    \to
    \Gamma(U,E\otimes F^\dagger)
  \end{equation*}
  induced by the $\fV$--equivariant Petri map is injective.
\end{definition}

\begin{prop}
  \label{Prop_VEquivarianFlexible+Petri=>Surjective}
  Let $(D_p)_{p \in \sP}$ be a family of linear elliptic differential operators and let $p \in \sP$.
  Let $U \subset M$ be an open subset.
  If $(D_p)_{p \in \sP}$ is $\fV$--equivariantly flexible in $U$ and
  $D_p$ satisfies the $\fV$--equivariant Petri condition in $U$,
  then the map $\Lambda_p^\fV$ defined in \autoref{Thm_EquivariantBrillNoetherLoci_Twists} is surjective.
  \qed
\end{prop}

\begin{remark}
  \label{Rmk_EquivariantBrillNoetherLoci_EstimateCodimension}
  There are analogues of the observations from \autoref{Rmk_BrillNoetherLoci_EstimateCodimension} in the equivariant setting.
  \begin{enumerate}
  \item
    Every $p \in \sP_{d,e}^\fV$ has an open neighborhood $\sU$ in $\sP$ such that $\sP_{d,e}^\fV\cap \sU$ is contained in a submanifold of codimension $\rk \Lambda_p^\fV$.    
  \item
    Let $\rho \in \N_0^m$ and let $U \subset M$ be an open subset.
    A linear elliptic differential operator $D\co \Gamma(E) \to \Gamma(F)$ satisfies the \defined{$\fV$--equivariant Petri condition up to rank $\rho$ in $U$}
    if for every non-zero $B = (B_1,\ldots,B_m) \in \bigoplus_{\alpha=1}^m \ker D_p^{\uV_\alpha} \otimes_{\,\bK_\alpha^\op} \ker D_p^{\dagger,\uV_\alpha^*}$ with $\rk B_\alpha \leq \rho_\alpha$ for $\alpha = 1,\ldots,m$ the section $\varpi^\fV(B)$ does not vanish on $U$.
    If $D_p$ satisfies this condition and $(D_p)_{p\in \sP}$ is $\fV$--equivariantly flexible,
    then
    \begin{equation*}
      \rk \Lambda_p \geq \sum_{\alpha=1}^m \min\set{\rho_\alpha,d_\alpha,e_\alpha}\max\set{d_\alpha,e_\alpha}.
    \end{equation*}
  \item
    \label{Rmk_EquivariantBrillNoetherLoci_EstimateCodimension_PetriUpToRankThree}
    Let $\rho \in \N_0^m$ and let $U \subset M$ be an open subset.
    Every first order linear elliptic differential operator $D \co \Gamma(E) \to \Gamma(F)$ satisfies the $\fV$--equivariant Petri condition up to rank $\rho$ on $U$ provided
    \begin{equation}
      \label{Eq_SumRhoCondition}
      \sum_{\alpha=1}^m \rk_\R V_\alpha \cdot \rho_\alpha \leq 3.
    \end{equation}
    \begin{proof}
      Let $x_0 \in M$, set $G \coloneq \pi_1(M,x_0)$, and denote by $\pi\co \tilde M \to M$ the universal cover.
      Every $s \in \ker D^{\uV_\alpha}$ can be regarded as a $G$--invariant section $\tilde s \in \Gamma(\pi^*E\otimes V_\alpha)^G$ with $\pi_1(M,x_0) \to \O(V_\alpha)$ denoting the monodromy representation of $\uV_\alpha$.
      This section can be regarded as $r_\alpha \coloneq \rk_\R V_\alpha$ sections $s_1,\ldots,s_{r_\alpha}$ of $\pi^*E$.
      For every $\alpha=1,\ldots,m$ let $s_1^\alpha,\ldots,s_{q_\alpha}^\alpha \in \ker D^{\uV_\alpha}$ be linearly independent over $\bK_\alpha$.
      The resulting collection of sections $(s_{j,k}^\alpha : \alpha = 1,\ldots,m, j=1,\ldots,q_\alpha, k=1,\ldots,r_\alpha)$ of $\pi^*E$ are linearly independent.
      (Analogous statements hold for $D^\dagger$ instead of $D$.)
      At this point, one can apply the argument in \autoref{Rmk_BrillNoetherLoci_EstimateCodimension}\autoref{Rmk_BrillNoetherLoci_EstimateCodimension_PetriUpToRankThree}.
      \pushQED{\hfill$\blacksquare$}
    \end{proof}
    Unfortunately, this is not as useful as \autoref{Rmk_BrillNoetherLoci_EstimateCodimension}\autoref{Rmk_BrillNoetherLoci_EstimateCodimension_PetriUpToRankThree} because \autoref{Eq_SumRhoCondition} is very restrictive;
    however, it is what lies at the heart of \citeauthor{Eftekhary2016}'s proof of the $4$--rigidity conjecture \cite{Eftekhary2016}.
    \qedhere
  \end{enumerate}
\end{remark}

%%% Local Variables:
%%% mode: latex
%%% TeX-master: "EquivariantBrillNoetherSuperRigidity"
%%% End:

\section{Equivariant Brill--Noether loci, II}
\label{Sec_EquivariantBrillNoetherLoci_Pullbacks}

In this section we formulate a variant of \autoref{Thm_BrillNoetherLoci} which applies to families of linear elliptic differential operators pulled back by a finite normal covering map.
This is not needed in \autoref{Part_Application}.
Throughout this section, assume the following.

\begin{situation}
  \label{Sit_G}
  Let $x_0 \in M$.
  Let $G$ be the quotient of $\pi_1(M,x_0)$ by a finite index normal subgroup $N$.
  Denote by $\pi\co (\tilde M,\tilde x_0) \to (M,x_0)$ the pointed covering map with characteristic subgroup $N$.  
  Let
  \begin{equation*}
    \mu_\alpha\co G \to \O(V_\alpha) \quad (\alpha = 1,\ldots,m = m(G))
  \end{equation*}
  be the irreducible representations of $G$.
  Set
  \begin{equation*}
    \bK_\alpha \coloneq \End_G(V_\alpha) \qandq
    k_\alpha \coloneq \dim_\R \bK_\alpha.
    \qedhere
  \end{equation*}
\end{situation}

If $D\co \Gamma(E) \to \Gamma(F)$ is a linear elliptic differential operator,
then $\ker \pi^*D$ and $\coker \pi^*D$ are representations of $G$.
Every representation $V$ of $G$ can be decomposed into irreducible representations.
Indeed, the evaluation map defines a $G$--equivariant isomorphism
\begin{equation}
  \label{Eq_DecompositionIntoIrreducibleRepresentations}
  \bigoplus_{\alpha=1}^m \Hom_G(V_\alpha,V) \otimes_{\,\bK_\alpha} V_\alpha \iso V.
\end{equation}
Hence,
\begin{equation*}
  V \iso \bigoplus_{\alpha=1}^m V_\alpha^{\oplus d_\alpha}
  \qwithq
  d_\alpha \coloneq \dim_{\bK_\alpha^\op} \Hom_G(V_\alpha,V).
\end{equation*}
In particular, $d = (d_1,\ldots,d_m) \in \N^m$ determines $V$ up to isomorphism.

\begin{definition}
  \label{Def_EquivariantBrillNoetherLoci_Pullback}
  Let $(D_p)_{p \in \sP}$ be a family of linear elliptic differential operators.
  For $d,e \in \N_0^m$
  define the \defined{$G$--equivariant Brill--Noether locus} $\sP_{d,e}^G$ by
  \begin{equation*}
    \sP_{d,e}^G
    \coloneq
    \set*{
      p \in \sP
      :
      \!\!
      \begin{array}{l}
        \dim_{\bK_\alpha^\op} \Hom_G(V_\alpha,\ker \pi^*D_p) = d_\alpha \text{ and } \\
        \dim_{\bK_\alpha^\op} \Hom_G(V_\alpha,\coker \pi^*D_p) = e_\alpha \text{ for every }
        \alpha = 1,\ldots,m
      \end{array}
      \!\!
    }.
    \qedhere
  \end{equation*}  
\end{definition}

\begin{remark}
  Let $D\co \Gamma(E) \to \Gamma(F)$ is a linear elliptic differential operator.
  The $G$--equivariant index of $\pi^*D$ is $\ind_G \pi^*D \coloneq [\ker \pi^*D] - [\coker \pi^*D] \in R(G)$.
  Here $R(G)$ denotes the representation ring of $G$;
  its elements are formal differences of isomorphism classes of representations of $G$.
  It is a consequence of the above discussion that $R(G) \iso \Z^m$ as abelian groups.

  For families of linear elliptic operators with $G$--equivariant index corresponding to $i \in \Z^m$ what was said in \autoref{Rmk_Index_Twist} applies in the present situation as well.
\end{remark}

\autoref{Lem_LyapunovSchmidtReduction} has the following refinement for $G$--equivariant Fredholm operators.

\begin{lemma}
  \label{Lem_EquivariantLyapunovSchmidtReduction}
  Let $X$ and $Y$ be two Banach spaces equipped with $G$--actions.
  Denote by $\sF_G(X,Y)$ the space of $G$--equivariant Fredholm operators.
  For every $L \in \sF_G(X,Y)$ there is an open neighborhood $\sU \subset \sF_G(X,Y)$ and a smooth map $\sS \co \sU \to \Hom_G(\ker L,\coker L)$ such that
  for every $T \in \sU$ there are $G$--equivariant isomorphisms
  \begin{equation*}
    \ker T \iso \ker \sS(T) \qandq
    \coker T \iso \coker \sS(T);
  \end{equation*}
  furthermore, $\rd_L\sS \co T_L\sF_G(X,Y) \to \Hom_G(\ker L,\coker L)$ satisfies
  \begin{equation*}
    \rd_L\sS(\hat L)s
    =
    \hat L s \mod \im L.
  \end{equation*}
\end{lemma}

\begin{proof}
  The proof of \autoref{Lem_LyapunovSchmidtReduction} carries over provided $\coim L$ and the lift of $\coker L$ are chosen $G$--invariant.
\end{proof}

\autoref{Lem_LyapunovSchmidtReduction} immediately implies the following.

\begin{theorem}
  \label{Thm_EquivariantBrillNoetherLoci_Pullback}
  Let $(D_p)_{p \in \sP}$ be a family of linear elliptic differential operators.
  Let $d,e \in \N_0^m$.
  If for every $p \in \sP_{d,e}^G$ the map $\Lambda_p^G \co T_p\sP \to \Hom_G(\ker \pi^*D_p,\coker \pi^*D_p)$ defined by
  \begin{equation*}
    \Lambda_p^G(\hat p)s \coloneq \rd_p(\pi^*D)(\hat p) s \mod \im \pi^*D_p
  \end{equation*}
  is surjective,
  then the following hold:
  \begin{enumerate}
  \item
    $\sP_{d,e}^G$ is a submanifold of codimension
    \begin{equation*}
      \codim \sP_{d,e}^G = \sum_{\alpha=1}^m k_\alpha d_\alpha e_\alpha.
    \end{equation*}
  \item
    If $\sP_{d,e}^G \neq \emptyset$,
    then $\sP_{\tilde d,\tilde e}^G \neq \emptyset$ for every $\tilde d,\tilde e \in \N_0^m$ with $\tilde d \leq d, \tilde e \leq e$, and $\tilde d-\tilde e = d-e$.
    \qed
  \end{enumerate}
\end{theorem}

\autoref{Def_Flexible},
\autoref{Def_PetrisCondition},
and \autoref{Prop_Flexible+Petri=>Surjective}
have the following analogues in the present situation.

\begin{definition}
  \label{Def_GEquivariantlyFlexible}
  A family of linear elliptic differential operators $(D_p)_{p \in \sP}$ is \defined{$G$--equivariantly flexible in $U$} if for every $p \in \sP$ and $A \in \Gamma(\Hom(E,F))$ supported in $U$
  there is a $\hat p \in T_p\sP$ such that
  \begin{equation*}
    \rd_p (\pi^*D)(\hat p)s = (\pi^*A)s \mod \im \pi^*D_p
  \end{equation*}
  for every $s \in \ker \pi^*D_p$.
\end{definition}

\begin{prop}
  \label{Prop_StronglyFlexible=>GEquivariantlyFlexible}
  If $(D_p)_{p \in \sP}$ is strongly flexible in $U$,
  then it is $G$--equivariantly flexible in $U$.
  \qed
\end{prop}

\begin{definition}
  \label{Def_GEquivariantPetriCondition}
  Let $U \subset M$ be an open subset.
  A linear elliptic differential operator $D\co \Gamma(E) \to \Gamma(F)$ satisfies the \defined{$G$--equivariant Petri condition in $U$} if the map
  \begin{equation*}
    \varpi_{D,U}^G\co (\ker \pi^*D \otimes \ker \pi^*D^\dagger)^G \to \Gamma(\pi^{-1}(U),\pi^*E\otimes \pi^*F^\dagger)^G
  \end{equation*}
  induced by the Petri map is injective.
\end{definition}

\begin{prop}
  \label{Prop_GEquivarianFlexible+Petri=>Surjective}
  Let $(D_p)_{p \in \sP}$ be a family of linear elliptic differential operators and let $p \in \sP$.
  Let $U \subset M$ be an open subset.
  If $(D_p)_{p \in \sP}$ is $G$--equivariantly flexible in $U$ and
  $D_p$ satisfies the $G$--equivariant Petri condition in $U$,
  then the map $\Lambda_p^G$ defined in \autoref{Thm_EquivariantBrillNoetherLoci_Pullback} is surjective.
  \qed
\end{prop}

The present discussion is an instance of that in \autoref{Sec_EquivariantBrillNoetherLoci_Twists} for $\fV = (\uV_\alpha)_{\alpha=1}^m$ with
\begin{equation*}
  \uV_\alpha \coloneq \tilde M \times_{\mu_\alpha} V_\alpha.
\end{equation*}
To make this precise,
we require the following result.

\begin{prop}
  \label{Prop_HomGVjKerD=KerDVj}
  Let $D\co \Gamma(E) \to \Gamma(F)$ be a linear differential operator.
  For every $\alpha=1,\ldots,m$ there are isomorphisms
  \begin{equation*}
    \Hom_G(V_\alpha,\ker \pi^*D) \iso \ker D^{\uV_\alpha^*} \qandq
    \Hom_G(V_\alpha,\coker \pi^*D) \iso \coker D^{\uV_\alpha^*}.
  \end{equation*}
\end{prop}

\begin{proof}
    The left and right regular representations of $G$ on $\R[G] \coloneq \Map(G,R)$ are defined by
  \begin{equation*}
    (\lambda_gf)(x) \coloneq f(g^{-1}x)
    \qandq
    (\rho_hf)(x) \coloneq f(xh) 
  \end{equation*}
  respectively.
  Since $\lambda_g$ and $\rho_h$ commute,
  $(g,h) \mapsto \lambda_g\circ \rho_h$ defines a representation of $G\times G$ on $\R[G]$.
  $G\times G$ also acts on $V_\alpha^*\otimes_{\,\bK_\alpha} V_\alpha$ via $(g,h) \mapsto \mu_\alpha(h^{-1})^* \otimes \mu_\alpha(g)$.
  The isomorphisms \autoref{Eq_DecompositionIntoIrreducibleRepresentations} corresponding to $\lambda$ and $\rho$ induces a $G\times G$--equivariant isomorphism
  \begin{equation}
    \label{Eq_RegularRepresentation}
    \R[G]
    \iso
    \bigoplus_{\alpha=1}^m V_\alpha^* \otimes_{\,\bK_\alpha} V_\alpha.
  \end{equation}
  To see this,
  observe that map $\ev_\one\co \Hom_G(V_\alpha,\Map(G,\R)) \to V_\alpha^*$ defined by $\ev_\one(\ell)(v) \coloneq \ell(v)(1)$ is an isomorphism.
  
  Set $\uV \coloneq \pi_*\ubR$.
  By \autoref{Prop_Pullback=Twist}\autoref{Prop_Pullback=Twist_Monodromy},
  the monodromy representation of $\uV$ is $\lambda$.
  Since $\lambda$ and $\rho$ commute,
  $\rho$ defines an action of $G$ on $\uV$.
  This is precisely the action induced by deck transformations of $\pi$.
  Therefore,
  \autoref{Eq_RegularRepresentation}
  induces a $G$--equivariant isomorphism
  \begin{equation}
    \label{Eq_EquivariantDecompositionOfV}
    \uV \iso \bigoplus_{\alpha=1}^m V_\alpha^* \otimes_{\,\bK_\alpha} \uV_\alpha.
  \end{equation}
  The isomorphisms $\pi_*$ from \autoref{Prop_Pullback=Twist}\autoref{Prop_Pullback=Twist_Isomorphism}
  together with the isomorphism induced by \autoref{Eq_EquivariantDecompositionOfV} identify $\pi^*D\co \Gamma(\pi^*E) \to \Gamma(\pi^*F)$ with
  \begin{equation*}
    \bigoplus_{\alpha=1}^m \id_{V_\alpha^*} \otimes_{\,\bK_\alpha} D^{\uV_\alpha}
    \co
    \bigoplus_{\alpha=1}^m V_\alpha^*\otimes_{\,\bK_\alpha} \Gamma(E\otimes \uV_\alpha)
    \to
    \bigoplus_{\alpha=1}^m V_\alpha^*\otimes_{\,\bK_\alpha} \Gamma(F\otimes \uV_\alpha).
  \end{equation*}
  Therefore,
  \begin{equation*}
    \ker \pi^*D \iso \bigoplus_{\alpha=1}^m V_\alpha^*\otimes_{\,\bK_\alpha} \ker D^{\uV_\alpha}
    \qandq
    \coker \pi^*D \iso \bigoplus_{\alpha=1}^m V_\alpha^*\otimes_{\,\bK_\alpha} \coker D^{\uV_\alpha}.
  \end{equation*}
  This implies the assertion.
\end{proof}

If $V$ and $W$ are representations of $G$,
then \autoref{Eq_DecompositionIntoIrreducibleRepresentations} induces isomorphisms
\begin{align*}
  \Hom_G(V,W) &\iso \bigoplus_{\alpha=1}^m \Hom_{\bK_\alpha}(\Hom_G(V_\alpha^*,V),\Hom_G(V_\alpha^*,W)) \qand \\
  (V\otimes W)^G &\iso \bigoplus_{\alpha=1}^m \Hom_G(V_\alpha^*,V) \otimes_{\,\bK_\alpha^\op} \Hom_G(V_\alpha,W).
\end{align*}
Therefore and by \autoref{Prop_HomGVjKerD=KerDVj},
there are isomorphisms
\begin{align*}
  \eta\co
  \Hom_G(\ker \pi^*D,\coker \pi^*D)
  &\to
    \bigoplus_{\alpha=1}^m \Hom_{\bK_\alpha}(\ker D^{\uV_\alpha},\coker D^{\uV_\alpha}) \qand \\
  \tau\co
  (\ker \pi^*D \otimes \ker \pi^*D^*)^G
  &\to
    \bigoplus_{\alpha=1}^m \ker D^{\uV_\alpha}\otimes_{\,\bK_\alpha} \ker D^{\dagger,\uV_\alpha^*}.
\end{align*}
Since the representations $V_\alpha^*$ are irreducible,
there is a permutation $\sigma \in S_m$ such that $V_\alpha^* = V_{\sigma(j)}$.
With this notation in place the discussions here and in \autoref{Sec_EquivariantBrillNoetherLoci_Twists} can be connected as follows:
\begin{enumerate}
\item
  In the situation of \autoref{Def_EquivariantBrillNoetherLoci_Twists} and 
  \autoref{Def_EquivariantBrillNoetherLoci_Pullback},
  \begin{equation*}
    \sP_{d,e}^G = \sP_{\sigma^*d,\sigma^*e}^\fV
  \end{equation*}
  with $(\sigma^*d)_\alpha = d_{\sigma(j)}$ and $(\sigma^*e)_\alpha = e_{\sigma(j)}$.
\item
  In the situation of \autoref{Thm_EquivariantBrillNoetherLoci_Twists} and \autoref{Thm_EquivariantBrillNoetherLoci_Pullback},
  \begin{equation*}
    \Lambda_p^\fV = \eta \circ \Lambda_p^G.
  \end{equation*}
\item
  In the situation of \autoref{Def_GEquivariantlyFlexible},
  the maps
  \begin{align*}
    \ev_p^\fV&\co \Gamma_c(U,\Hom(E,F)) \to \bigoplus_{\alpha=1}^m \Hom_{\bK_\alpha}(\ker D_p^{\uV_\alpha},\coker D_p^{\uV_\alpha}) \qand \\
    \ev_p^G &\co \Gamma_c(U,\Hom(E,F)) \to \Hom_G(\ker \pi^*D_p,\coker \pi^*D_p)
  \end{align*}
  defined by  
  \begin{align*}
    \ev_p^\fV
    &\coloneq \bigoplus_{\alpha=1}^m \ev_p^\alpha
      \qwithq
      \ev_p^\alpha(A)s \coloneq (A \otimes \id_{\uV_\alpha})s \mod \im D_p^{\uV_\alpha} \qand \\
    \ev_p^G(A)s
    &\coloneq (\pi^*A)s \mod \im \pi^*D_p
  \end{align*}
  satisfy
  \begin{equation*}
    \ev_p^\fV = \eta \circ \ev_p^G.
  \end{equation*}
  Therefore, $(D_p)_{p \in \sP}$ is $G$--equivariantly flexible in $U$ if and only if it is $\fV$--equivariantly flexible in $U$.
\item
  In the situation of \autoref{Def_GEquivariantPetriCondition},
  the map $\varpi_{D,U}^G$ satisfies
  \begin{equation}
    \label{Eq_EquivariantPetriCondition_PullbackVsTwist}
    \varpi_{D,U}^G = \pi^* \circ \varpi_{D,U}^\fV \circ \tau.
  \end{equation}
  Therefore, $D$ satisfies the $G$--equivariant Petri condition in $U$ if and only if it satisfies the $\fV$--equivariant Petri condition in $U$.
\end{enumerate}

\begin{remark}
  Suppose that $\pi\co \tilde M \to M$ is a finite covering map with characteristic subgroup $C < \pi_1(M,x_0)$.
  Denote by $N$ the normal core of $C$,
  denote by $\rho\co (\hat M,\hat x_0) \to (M,x_0)$ the pointed covering map with characteristic subgroup $N$, and
  set $G \coloneq \pi_1(M,x_0)/N$.
  The argument in the proof of \autoref{Prop_HomGVjKerD=KerDVj} shows that
  \begin{align*}
    \ker \pi^*D
    &\iso
    \bigoplus_{\alpha=1}^m V_\alpha^C\otimes_{\,\bK_\alpha^\op} \Hom_G(V_\alpha,\ker \rho^*D)
    \qand \\
    \coker \pi^*D
    &\iso
    \bigoplus_{\alpha=1}^m V_\alpha^C\otimes_{\,\bK_\alpha^\op} \Hom_G(V_\alpha,\coker \rho^*D).
  \end{align*}
  The crucial point is that for $S \coloneq \pi_1(M,x_0)/C$ the decomposition \autoref{Eq_DecompositionIntoIrreducibleRepresentations} of $\Map(S,\R)$ is
  \begin{equation*}
    \Map(S,\R) \iso \bigoplus_{\alpha=1}^m (V_\alpha^*)^C\otimes_{\,\bK_\alpha} V_\alpha;
  \end{equation*}
  indeed:
  the map $\ev_{[\one]}\co \Hom_G(V_\alpha,\Map(S,\R)) \to V_\alpha^*$ defined by $\ev_\one(\ell)(v) \coloneq \ell(v)([1])$ is injective and its image is $(V_\alpha^*)^C$.
  With the above in mind \autoref{Thm_EquivariantBrillNoetherLoci_Pullback} can be brought to bear on non-normal covering maps.
\end{remark}

%%% Local Variables:
%%% mode: latex
%%% TeX-master: "EquivariantBrillNoetherSuperRigidity"
%%% End:

\section{Petri's condition revisited}
\label{Sec_PetrisConditionRevisited}

While Petri's condition is typically hard to verify for any particular elliptic operator,
one can sometimes prove that it is satisfied for a generic element of a family of operators. 
\autoref{Thm_PetrisConditionFailsInInfiniteCodimension} provides a useful tool for proving such statements.
This result has been developed by \citet[Section 5.2]{Wendl2016} and was the essential innovation which allowed \citeauthor{Wendl2016} to prove the super-rigidity conjecture.
 
Throughout this section, let $x \in M$ and, furthermore, amend \autoref{Def_FamilyOfEllipticOperators} to
\begin{quote}
  \emph{require in addition that for every $p \in \sP$ the linear elliptic differential operator $D_p$ has smooth coefficients.}
\end{quote}
Let us begin by introducing the following algebraic variant of Petri's condition.

\begin{definition}
  Denote by $\sE$ the sheaf of sections of $E$ and by $\sE_x$ its stalk at $x$;
  that is:
  \begin{equation*}
    \sE_x \coloneq \varinjlim_{x \in U} \Gamma(U,E).
  \end{equation*}
  If $s \in \sE_x$ vanishes at $x$,
  then its derivative at $x$ does not depend on the choice of a local trivialization and defines an element $\rd_xs \in \Hom(T_xM,E_x)$.
  If $\rd_xs = 0$,
  then $s$ has a second derivative $\rd_x^2s \in \Hom(S^2T_xM,E_x)$ at $x$;
  and so on:
  if $s(x)$, $\rd_xs$, \ldots, $\rd_x^{j-1}s$ vanish,
  then $s$ is said to vanish to $(j-1)^{\text{st}}$ order and its \defined{$j^{\text{th}}$ derivative}
  \begin{equation*}
    \rd_x^j s \in \Hom(S^j T_xM,E_x)
  \end{equation*}
  is defined.
  The \defined{vanishing order filtration} $\sV_\bullet\sE_x$ on $\sE_x$ is defined by
  \begin{equation*}
    \sV_j\sE_x
    \coloneq    
    \set{
        s \in \sE_x
        :
        s 
        \text{ vanishes to $(j-1)^{\text{st}}$ order}
      }
  \end{equation*}
  for $j \in \N_0$ and $\sV_{-j}\sE_x \coloneq \sE_x$ for $j \in \N$.
  For $\ell \in \N_0$ the \defined{$\ell$--jet space of $E$ at $x$} is
  \begin{equation*}
    J_x^\ell E \coloneq \sE_x/\sV_{\ell+1}\sE_x.
  \end{equation*}
  The \defined{$\infty$--jet space of $E$ at $x$} is
  \begin{equation*}
    J_x^\infty E \coloneq \varprojlim J_x^\ell E = \sE_x/\sV_\infty\sE_x
    \qwithq
    \sV_\infty\sE_x \coloneq \bigcap_{j\in\Z} \sV_j\sE_x.
  \end{equation*}
  For $\ell \in \N_0\cup\set{\infty}$ the \defined{$\ell$--jet} of a linear differential operator $D\co \Gamma(E) \to \Gamma(F)$ of order $k$ is the linear map
  \begin{equation*}
    J_x^\ell D \co J_x^{k+\ell}E \to J_x^\ell F.
  \end{equation*}
  induced by $D$. 
\end{definition}

\begin{definition}
  \label{Def_JetPetriCondition}
  An $\infty$--jet of a linear elliptic differential operator $J_x^\infty D\co J_x^\infty E \to J_x^\infty F$ satisfies the \defined{$\infty$--jet Petri condition} if the map
  \begin{equation*}
    \varpi_{J_x^\infty D}\co \ker J_x^\infty D \otimes \ker J_x^\infty D^\dagger \to J_x^\infty(E \otimes F^\dagger)
  \end{equation*}
  induced by the Petri map is injective.
\end{definition}

The $\infty$--jet Petri condition and the $\fV$--equivariant Petri condition are related as follows.

\begin{prop}
  \label{Prop_JetPetriImpliesVEquivariantPetri}
  Assume \autoref{Sit_V}.
  Let $U \subset M$ be an open neighborhood of $x$.
  Let $D\co \Gamma(E) \to \Gamma(F)$ be a linear elliptic differential operator.
  Suppose that $D$ and $D^\dagger$ possess the strong unique continuation property at $x$.
  If $J_x^\infty D$ satisfies the $\infty$--jet Petri condition,
  then $D$ satisfies the $\fV$--equivariant Petri condition in $U$.
\end{prop}

\begin{proof}
  Set $G \coloneq \pi_1(M,x)$, and denote by $\pi\co \tilde M \to M$ the universal cover.
  It suffices to prove that the map $\varpi_{D,U}^G\co \ker \pi^*D \otimes \ker \pi^*D \to \Gamma(\tilde U,\pi^*E\otimes \pi^*F^\dagger)$ induced by Petri map is injective;
  cf. \autoref{Rmk_EquivariantBrillNoetherLoci_EstimateCodimension}\autoref{Rmk_EquivariantBrillNoetherLoci_EstimateCodimension_PetriUpToRankThree} and \autoref{Eq_EquivariantPetriCondition_PullbackVsTwist}.
  Let $\tilde x \in \pi^{-1}(x)$.
  Since $J_x^\infty D = J_{\tilde x}^\infty \pi^*D$ and $J_x^\infty D^\dagger = J_{\tilde x}^\infty \pi^*D^\dagger$,
  there is a commutative diagram
  \begin{equation*}
    \begin{tikzcd}
      \ker \pi^*D \otimes \ker \pi^*D \ar[d] \ar[r,"\varpi_{D,U}^G"]
      &
      \Gamma(\tilde U,\pi^*E\otimes \pi^*F^\dagger) \ar[d] \\
      \ker J_x^\infty D \otimes \ker J_x^\infty D^\dagger \ar[r,"\varpi_{J_x^\infty D}"] 
      &
      J_x^\infty(E\otimes F^\dagger).
    \end{tikzcd}
  \end{equation*}
  Since $D$ and $D^\dagger$ have the strong unique continuation property,
  the left vertical map is injective.
  Therefore,
  if $\varpi_{J_x^\infty D}$ is injective,
  then so is $\varpi_{D,U}^G$.
\end{proof}

The failure of the $\infty$--jet Petri condition manifests itself at the level of symbols as follows.

\begin{definition}
  Let $k \in \N_0$.
  A \defined{symbol of order $k$} is an element $\sigma \in S^kT_xM\otimes \Hom(E_x,F_x)$.
  Since every $v \in T_xM$ defines a derivation $\del_v$ on the polynomial algebra $S^\bullet T_x^* M$,
  every symbol $\sigma$ defines a \defined{formal differential operator}
  \begin{equation*}
    \hat\sigma\co S^\bullet T_x^* M\otimes E_x \to S^\bullet T_x^* M\otimes F_x.
  \end{equation*}
  The \defined{adjoint symbol} $\sigma^\dagger \in S^kT_xM\otimes \Hom(F_x^\dagger,E_x^\dagger)$ is $(-1)^k$--times the image of $\sigma$ under the the map induced by the canonical isomorphism $\Hom(E_x,F_x) \iso \Hom(F_x^\dagger,E_x^\dagger)$.
\end{definition}

The symbol of a linear differential operator $D \co \Gamma(E) \to \Gamma(F)$ of order $k$ is a section $\sigma(D) \in \Gamma(S^k TM \otimes \Hom(E,F))$.
Its value $\sigma_x(D)$ at $x \in M$ is a symbol in the above sense and depends only on $J_x^0 D$.
Furthermore, $\sigma_x(D^\dagger) = \sigma_x(D)^\dagger$.

\begin{definition}  
  The \defined{polynomial Petri} map
  $\hat\varpi\co (S^\bullet T_x^* M\otimes E_x) \otimes (S^\bullet T_x^* M\otimes F_x^\dagger) \to   S^\bullet T_x^* M\otimes E_x\otimes F_x^\dagger$ is defined by
  \begin{equation*}
    \hat\varpi\paren*{(p\otimes e) \otimes (q \otimes f)}
    \coloneq
    (p\cdot q) \otimes e \otimes f.
  \end{equation*}
  A symbol $\sigma \in S^kT_xM\otimes \Hom(E_x,F_x)$ satisfies the \defined{polynomial Petri condition} if the map
  \begin{equation*}
    \hat\varpi_\sigma\co \ker \hat\sigma \otimes \ker\hat\sigma^\dagger \to S^\bullet T_x^* M\otimes E_x\otimes F_x^\dagger
  \end{equation*}
  induced by the polynomial Petri map is injective.
\end{definition}

\begin{prop}
  \label{Prop_PolynomialPetri=>JetPetri}
  If $J_x^\infty D$ fails to satisfy the $\infty$--jet Petri condition,
  then $\sigma_x(D)$ fails to satisfy the polynomial Petri condition.
\end{prop}

The proof of this result and the upcoming discussion require the following algebraic definitions, constructions, and facts:
\begin{enumerate}
\item
  Let $V$ be a vector space equipped with a filtration $\sF_\bullet V$.
  The \defined{order of $\sF_\bullet V$} is the map $\ord\co V \to \Z \cup \set{\infty,-\infty}$ defined by
  \begin{equation*}
    \ord(v) \coloneq \sup\set{ j \in \Z : v \in \sF_j V }.
  \end{equation*}
  $\sF_\bullet V$ is called \defined{exhaustive} if $\ord^{-1}(-\infty) = \emptyset$ or, equivalently,
  $\bigcup_{j \in \Z} \sF_j V = V$.
  $\sF_\bullet V$ is \defined{separated} if $\ord^{-1}(\infty) = 0$ or, equivalently,
  $\bigcap_{j \in \Z} \sF_j V = 0$.
\item
  The \defined{associated graded vector space of $\sF_\bullet V$} is
  \begin{equation*}
    \gr V \coloneq \bigoplus_{j \in \Z} \gr_j V
    \qwithq
    \gr_j V \coloneq \sF_j V/\sF_{j+1} V.
  \end{equation*}  
  Define $[\cdot]\co \ord^{-1}(\Z) \to \gr V$ by
  \begin{equation*}
    [v] \coloneq v + \sF_{j+1} V \in \gr_j V \qwithq j \coloneq \ord(v).
  \end{equation*}
  This is map is not linear and not even continuous (except for a few corner cases).
  It is appropriate to regard $[v]$ as the leading order term of $v$.
\item
  Let $W$ be a further vector space equipped with a filtration $\sF_\bullet W$.
  A linear map $f\co V \to W$ is \defined{of order $k \in \Z$} if $f(\sF_jV) \subset \sF_{j+k}W$ for every $j \in \Z$ but the same does not hold for $k+1$ instead of $k$.
  If this is the case,
  then $f$ induces a linear map
  \begin{equation*}
    \gr f\co \gr V \to \gr W
  \end{equation*}
  of degree $k$.
  There is a canonical inclusion $\gr \ker f \into \ker \gr f$ and and a canonical projection $\coker \gr f \onto \gr \coker f$.
  These maps are typically not isomorphisms.
\item
  The tensor product $V\otimes W$ inherits the \defined{tensor product filtration} defined by
  \begin{equation*}
    \sF_j(V\otimes W) \coloneq \sum_{j_1+j_2 = j} \sF_{j_1}V \otimes \sF_{j_2}W.
  \end{equation*}
  There is a canonical graded isomorphism
  \begin{equation}
    \label{Eq_GrVW=GrVGrW}
    \gr (V\otimes W) \iso \gr V\otimes \gr W.
  \end{equation}
  If $\sF_\bullet V$ and $\sF_\bullet W$ are both separated (exhaustive),
  then so is $\sF_\bullet(V\otimes W)$.
\end{enumerate}

Let $\ell \in \N_0\cup\set{\infty}$.
The vanishing order filtration on $\sE_x$ descends to filtration on $J_x^\ell E$.
Taylor expansion defines an isomorphism
\begin{equation}
  \label{Eq_TaylorExpansionIsomorphism}
  T_x^\ell \co \gr J_x^\ell E \to \bigoplus_{j=0}^\ell S^j T_x^* M\otimes E_x;
\end{equation}
in particular:
\begin{equation}
  \label{Eq_DimJetLE}
  \dim J_x^\ell E = r\cdot\binom{n+k+\ell}{n}
\end{equation}
(and similarly for $E^\dagger$, $F$, and $F^\dagger$ instead of $E$).
If $D \co \Gamma(E) \to \Gamma(F)$ is a linear differential operator and $\sigma_x(D)$ denotes its symbol at $x$,
then
\begin{equation*}
  T_x^\infty \circ \gr J_x^\infty D
  = \hat \sigma_x(D) \circ T_x^\infty.
\end{equation*}
Furthermore,
\begin{equation*}
  T_x^\infty \circ \gr J_x^\infty\varpi = \hat\varpi \circ  (T_x^\infty \otimes T_x^\infty).
\end{equation*}
This implies corresponding identities for $\ell \in \N_0$ instead of $\infty$ provided $\hat\sigma_x(D)$ and $\hat\varpi$ are appropriately truncated.

\begin{proof}[Proof of \autoref{Prop_PolynomialPetri=>JetPetri}]
  The vanshing order filtration on $J_x^\infty E$ and $J_x^\infty F$ is exhaustive and separated.
  Therefore,
  if $B \in \ker \varpi_{J_x^\infty \omega}$ is non-zero,
  then
  \begin{equation*}
    [B] \in
    (\ker \gr J_x^\infty D \otimes \ker \gr J_x^\infty D^\dagger)
    \cap
    \ker  \gr J_x^\infty \varpi
  \end{equation*}
  is defined and non-zero.
  By the preceding discussion,
  $T_x^\infty$ induces an isomorphism
  \begin{equation*}
    (\ker\gr J_x^\infty D \otimes \ker\gr J_x^\infty D^\dagger)
    \cap
    \ker  \gr J_x^\infty \varpi
    \iso
    \ker \hat\varpi_\sigma
    \qwithq
    \sigma = \sigma_x(D).
    \qedhere
  \end{equation*}
\end{proof}

\autoref{Prop_PolynomialPetri=>JetPetri} is probably not terribly useful for establishing the $\infty$--jet Petri condition.
The polynomial Petri condition fails for real Cauchy--Riemann operators (see \cite[Example 5.5]{Wendl2016} and \autoref{Prop_BInKerHatVarpi}) and we suspect that it typically fails.
However, this is no reason to despair.
It can be shown that every $\hat B \in \ker \hat\sigma_x(D)\otimes\ker \hat\sigma_x(D^\dagger)$ admits some (but not a unique) lift to an element $B \in \ker J_x^\infty D\otimes \ker J_x^\infty D^\dagger$.
However,
if $\hat B \in \ker \hat\varpi$,
then this does not imply that $B \in \ker J_x^\infty\varpi$.
In fact, it is reasonable to expect that typically the higher order terms will prevent the vanishing of $J_x^\infty\varpi(B)$.
The upcoming theorem shows that this heuristic is valid assuming an algebraic hypothesis on symbol level. 

\begin{definition}
  Let $k,\ell \in \N_0$.
  A family of linear elliptic differential operators $(D_p)_{p\in\sP}$ of order $k$ is \defined{$\ell$--jet strongly flexible at $x$} if for every $p\in\sP$ and $A \in J_x^\ell\Hom(E_x,F_x)$ there is a $\hat p\in T_p\sP$ such that
  \begin{equation*}
    \rd_pJ_x^\ell D(\hat p) s = As
  \end{equation*}
  for every $s \in J_x^{k+\ell} E$.
\end{definition}

\begin{definition}
  \label{Def_WendlsCondition}
  Let $k \in \N_0$ and $\sigma \in S^kT_xM\otimes\Hom(E_x,F_x)$.
  Let $c_0\co \N_0\times\N \to (0,\infty)$ and $\ell_0\co \N_0\times \N \to \N_0$.
  The symbol $\sigma$ satisfies \defined{Wendl's condition for $c_0$ and $\ell_0$} if for every homogeneous $B \in \ker \hat\varpi_\sigma$ there are right-inverses $\hat R$ and $\hat R^\dagger$ of $\hat \sigma$ and $\hat \sigma^\dagger$ such that the linear map
  \begin{equation*}
    \hat\bL_{\sigma,B}\co S^\bullet T_x^*M\otimes \Hom(E_x,F_x) \to S^\bullet T_x^*M\otimes E_x\otimes F_x^\dagger
  \end{equation*}
  defined by
  \begin{equation*}
    \hat\bL_{\sigma,B}(A)
    \coloneq   
    \hat\varpi
    \paren[\big]{
      (\hat RA \otimes \one
      +
      \one\otimes \hat R^\dagger A^\dagger)
      B
    }
  \end{equation*}
  satisfies
  \begin{equation*}
    \rk \hat\bL_{\sigma,B}^{\leq\ell}
    \geq
    c_0(d,\rho)\ell^n
  \end{equation*}
  for every $\ell \geq \ell_0(d,\rho)$
  with $d \coloneq \deg B$ and $\rho \coloneq \rk B$.
  Here
  \begin{equation*}
    \hat\bL_{\sigma,B}^{\leq\ell}
    \co
    \bigoplus_{j=0}^\ell S^j T_x^*M\otimes\Hom(E_x,F_x)
    \to
    \bigoplus_{j=0}^{k+\ell} S^j T_x^*M\otimes E_x\otimes F_x^\dagger
  \end{equation*}
  denotes the truncation of $\bL_{\sigma,B}$.
\end{definition}

\begin{remark}
  The reader is by no means expected to understand the significance of Wendl's condition at this point.
  The following remarks might help clarify the definition:
  \begin{enumerate}
  \item
    \autoref{Prop_JetD} proves that $\hat\sigma$ and $\hat\sigma^\dagger$ have right-inverses provided $\sigma$ is elliptic. 
  \item
    The maps $\bL_{\sigma,B}^{\leq\ell}$ play a crucial role in the proof of \autoref{Thm_PetrisConditionFailsInInfiniteCodimension}.
    Their ranks provide lower bounds for the ranks of certain map between jet spaces tied to the failure of the $\infty$--jet Petri condition.
  \item
    The dimension of the codomain of $\hat\bL_{\sigma,B}^{\leq\ell}$ grows like $\ell^n$;
    therefore, $\rk \hat\bL_{\sigma,B}^{\leq\ell}$ is assumed have maximal growth rate.
  \item
    Unfortunately, it appears not to be easy to verify whether a given symbol $\sigma$ satisfies Wendl's condition or not.
    In fact, even determining $\ker\hat\varpi_\sigma$ is a non-trivial task.
    \autoref{Thm_CauchyRiemannSymbolSatisfiesWendlsCondition} proves that the symbol $\sigma$ of a real Cauchy--Riemann operator satisfies Wendl's condition.
    As far as we know,
    it is possible that every elliptic symbol satisfies Wendl's condition.
    \qedhere
  \end{enumerate}
\end{remark}

\begin{theorem}[{\citet[Section 5.2]{Wendl2016}}]
  \label{Thm_PetrisConditionFailsInInfiniteCodimension}
  Let $(D_p)_{p\in\sP}$ be a family of linear elliptic differential operators. 
  If 
  \begin{enumerate}
  \item
    $(D_p)_{p\in\sP}$ is $\ell$--jet strongly flexible at $x$ for every $\ell \in \N_0$, and
    \item 
    there are $c_0\co \N_0\times\N \to (0,\infty)$ and $\ell_0\co \N_0\times \N \to \N_0$ such that for every $p \in \sP$ the symbol $\sigma_x(D_p)$ satisfies Wendl's condition for $c_0$ and $\ell_0$,
  \end{enumerate}
  then the subset
  \begin{equation*}
    \sR
    \coloneq
    \set{
      p \in \sP
      :
      J_x^\infty D
      \textnormal{ fails to satisfy the $\infty$--jet Petri condition}
    }
  \end{equation*}
  has infinite codimension.%
  \footnote{%
    \autoref{Sec_CodimensionInBanachManifolds} defines the codimension of a subset of a Banach manifold.
  }
\end{theorem}

The remainder of this section is devoted to the proof of \autoref{Thm_PetrisConditionFailsInInfiniteCodimension}.
The following observation decomposes $\sR$ into pieces whose codimensions can be estimated using the hypotheses of the theorem.

\begin{prop}
  \label{Prop_AsymptoticFiniteJetPetriImpliesInfinityJetPetri}
  For $d \in \N_0$ and $\rho \in \N$ set
  \begin{equation*}
    \sR_{d,\rho}^\ell
    \coloneq
    \set*{
      p \in \sP
      :
      \!\!
      \begin{array}{l}
        \textnormal{there is a }
        B \in (\ker J_x^\ell D_p \otimes \ker J_x^\ell D_p^\dagger) \cap \ker J_x^{k+\ell}\varpi \\
        \textnormal{with }
        \ord(B) \leq d \textnormal{ and}
        \rk B = \rho
      \end{array}
      \!\!
    }.
  \end{equation*}
  The set $\sR$ satisfies
  \begin{equation*}
    \sR
    \subset
    \bigcup_{
      \subalign{
        d &\in \N_0\\ \rho &\in \N\\ \ell_0 &\in \N_0
      }
    }
    \bigcap_{\ell \geq \ell_0}
    \sR_{d,\rho}^\ell.
  \end{equation*} 
\end{prop}

The proof relies on the following fact.

\begin{prop}
  \label{Prop_RestrictionToJetsInjectivity}
  Let $V$ be a vector space and equipped with a filtration $\sF_\bullet V$.
  Set 
  \begin{equation*}
    Q_\ell \coloneq V/\sF_\ell V \qandq
    Q \coloneq \varprojlim Q_\ell.
  \end{equation*}
  If $R \subset Q$ is a finite dimensional subspace,
  then there is an $\ell_0 \in \N_0$ such that for every $\ell \geq \ell_0$ the restriction of the composition $R \to Q \to Q_\ell$ is injective.
\end{prop}

\begin{proof}
  $K_\ell \coloneq \ker(R \to Q \to Q_\ell)$ is a decreasing sequence of finite dimensional vector spaces with $\varprojlim K_\ell = 0$.
  Therefore, $K_\ell = 0$ for $\ell \gg 1$.
\end{proof}

\begin{proof}[Proof of \autoref{Prop_AsymptoticFiniteJetPetriImpliesInfinityJetPetri}]
  If $p \in \sR$,
  then there exists a non-zero $B \in \ker \varpi_{J_x^\infty D_p}$.  
  Set $d \coloneq \ord(B)$, $\rho \coloneq \rk B$, and write $B$ as
  \begin{equation*}
    B = \sum_{i=1}^\rho s_i\otimes t_i.
  \end{equation*}
  with $s_1,\ldots,s_\rho$ and $t_1,\ldots,t_\rho$ linearly independent.
  Since
  \begin{equation*}
    J_x^\infty E = \varprojlim J_x^\ell E
  \end{equation*}
  and by \autoref{Prop_RestrictionToJetsInjectivity},
  there is an $\ell_0 \in \N_0$ such that for every $\ell \geq \ell_0$
  the $(k+\ell)$--jets
  $\tilde s_1,\ldots,\tilde s_\rho \in J_x^{k+\ell}E$ and
  $\tilde t_1,\ldots,\tilde t_\rho \in J_x^{k+\ell}F^\dagger$ are linearly independent.  
  By construction,
  \begin{equation*}
    \tilde B \coloneq \sum_{i=1}^\rho \tilde s_i\otimes \tilde t_i
    \in
    \ker \varpi_{J_x^\ell D},
    \quad
    \ord(\tilde B) = d,
    \qandq
    \rk \tilde B = \rho.
  \end{equation*}
  Therefore, $p \in \sR_{d,\rho}^\ell$ for every $\ell \geq \ell_0$.
\end{proof}

To estimate the codimension of $\sR_{d,\rho}^\ell$ we require the following.
Recall that $\dim M = n$ and $\rk E = \rk F = r$.

\begin{prop}
  \label{Prop_JetD}
  Let $J_x^\infty D$ be an $\infty$--jet of an elliptic differential operator of order $k$.
  The following hold:
  \begin{enumerate}
  \item
    \label{Prop_JetD_HatSigmaSurjective}
    The formal differential operator $\hat\sigma_x(D)$ is surjective.
  \item
    \label{Prop_JetD_JetDSurjective}
    For every $\ell \in \N_0\cup\set{\infty}$ the $\ell$--jet $J_x^\ell D$ is surjective.
  \item
    \label{Prop_JetD_DimKerJetD}
    For every $\ell \in \N_0$
    \begin{equation*}
      \dim \ker J_x^\ell D = r\cdot\sqparen*{\binom{n+k+\ell}{n}-\binom{n+\ell}{n}}.
    \end{equation*}
  \end{enumerate}
\end{prop}

\begin{proof}
  Since $D$ is elliptic,
  the restriction $\hat\sigma_x^k(D)\co S^k T_x^*M \otimes E_x \to F_x$ is surjective.
  Choose a basis $(\xi_1,\ldots,\xi_n)$ of $T_x^*M$.
  For a multi-index $\alpha \in \N_0^n$ set $\xi^\alpha \coloneq \prod_{i=1}^n \xi_i^{\alpha_i}$.
  A moment's thought shows that
  \begin{equation*}
    \hat\sigma_x(D)(\xi_1^k\xi^{\alpha}\otimes e)
    =
    \binom{k+\alpha_1}{\alpha_1}
    \xi^\alpha \otimes \hat\sigma_x(D)(\xi_1^k\otimes e)
    +
    R
  \end{equation*}
  with $R$ denoting a sum of tensors of the form $\xi^\beta \otimes w$ with $\beta_1 > \alpha_1$.
  Therefore,
  the image of $\hat\sigma_x(D)$ contains every tensor product of the form $\xi_1^m\otimes f$.
  Descending induction on $\alpha_1$ starting at $m$ proves that the image of $\hat\sigma_x(D)$ contains every tensor product of the form $\xi^\alpha \otimes w$ with $\abs{\alpha} = m$.
  This proves \autoref{Prop_JetD_HatSigmaSurjective}.

  Since $\coker \gr J_x^\ell D \onto \gr \coker J_x^\ell D$,
  \autoref{Prop_JetD_HatSigmaSurjective}
  implies
  \autoref{Prop_JetD_JetDSurjective}.

  Finally,
  \autoref{Prop_JetD_JetDSurjective} implies $\dim \ker J_x^\ell D = \dim J_x^{k+\ell} E - \dim J_x^\ell F$.
  Therefore, \autoref{Prop_JetD_DimKerJetD} follows from \autoref{Eq_DimJetLE}.
\end{proof}

\begin{proof}[Proof of \autoref{Thm_PetrisConditionFailsInInfiniteCodimension}]
  Let $d \in \N_0$, $\rho \in \N$, and $\ell \geq \ell_0(d,\rho)$.
  By \autoref{Prop_JetD},
  \begin{align*}
    \sK^\ell
    &\coloneq
      \set[\big]{
      (p,s) \in \sP\times J_x^{k+\ell} E
      :
      s \in \ker J_x^\ell D_p
      }
      \qand \\
    \sC^\ell
    &\coloneq
      \set[\big]{
      (p,t) \in \sP\times J_x^{k+\ell} F^\dagger
      :
      t \in \ker J_x^\ell D_p^\dagger
      }
  \end{align*}
  are vector bundles over $\sP$ of rank
  \begin{equation*}    
    \rk \sK^\ell
    =
    \rk \sC^\ell
    =
    r\cdot \sqparen*{\binom{n+k+\ell}{n}-\binom{n+\ell}{n}}.
  \end{equation*}
  Therefore,
  \begin{equation*}
    \sT_{d,\rho}^\ell
    \coloneq
    \set*{
      (p,B) \in \sK^\ell\otimes \sC^\ell
      :
      \ord(B) \leq d
      \textnormal{ and }
      \rk B = \rho
    }
  \end{equation*}
  is a fiber bundle over $\sP$ of with fibers of dimension
  \begin{equation*}
    2\rho r
    \cdot
    \sqparen*{
      \binom{n+k+\ell}{n}-\binom{n+\ell}{n}
    }
    -
    \rho^2
    \leq
    c(n,k) \rho \ell^{n-1}.
  \end{equation*}
  Denote by $\pi\co \sT_{d,\rho}^\ell \to \sP$ the projection map.
  By construction,
  \begin{equation*}
    \sR_{d,\rho}^\ell
    =
    \pi\paren{(J_x^{k+\ell}\varpi\circ\pr_2)^{-1}(0)}.
  \end{equation*}
  The upcoming discussion proves that for every $(p,B) \in (J_x^{k+\ell}\varpi\circ\pr_2)^{-1}(0)$
  \begin{equation*}
    \rk \rd_{(p,B)}(J_x^{k+\ell}\varpi\circ\pr_2) \geq c_0(d,\rho)\ell^n.
  \end{equation*} 
  Therefore and by \autoref{Prop_Codimension_Image},
  $\sR_{d,\rho}^\ell$ has codimension at least
  \begin{equation*}
    (c_0(d,\rho)\ell - c(n,k) \rho) \ell^n.
  \end{equation*}
  This immediately implies the theorem.

  Let $(p,B) \in (J_x^{k+\ell}\varpi\circ\pr_2)^{-1}(0)$.
  Set $\sigma \coloneq \sigma_x(D_p)$.
  Denote by $\hat R$ and $\hat R^\dagger$ the right-inverses of $\hat \sigma$ and $\hat \sigma^\dagger$ from \autoref{Def_WendlsCondition}.
  Denote by $R$ and $R^\dagger$ right-inverses of $J_x^{k+\ell} D_p$ and $J_x^{k+\ell} D_p^\dagger$ such that $\gr R$ and $\gr R^\dagger$ correspond to the truncations of $\hat R$ and $\hat R^\dagger$ with respect to \autoref{Eq_TaylorExpansionIsomorphism}.  
  Define $\bL_{p,B}\co J_x^\ell\Hom(E_x,F_x) \to J_x^{k+\ell} (E_x\otimes F_x^\dagger)$ by
  \begin{equation*}
    \bL_{p,B}(A)
    \coloneq
    J_x^{k+\ell}\varpi
    \paren[\big]{
      (RA\otimes \one + \one\otimes R^\dagger A^\dagger) B
    }
  \end{equation*}
  Since $(D_p)_{p\in\sP}$ is $\ell$--jet flexible,
  for every $A \in J_x^\ell\Hom(E_x,F_x)$ there is a $\hat p\in T_p\sP$ such that
  \begin{equation*}
    \rd_pJ_x^\ell D(\hat p) s = As
  \end{equation*}
  for every $s \in J_x^{k+\ell} E$.
  If this identity holds,
  then
  \begin{equation*}
    (\hat p,(RA\otimes \one + \one\otimes R^\dagger A^\dagger) B)
    \in T_{(p,B)}\sT_{d,\rho}^\ell.
  \end{equation*}
  Therefore,  
  \begin{equation*}
    \rk \rd_{(p,B)}(J_x^{k+\ell}\varpi\circ\pr_2) \geq \rk \bL_{p,B}.
  \end{equation*}
  Since $\gr \ker \bL_{p,B} \into \ker \gr \bL_{p,B}$,
  \begin{equation*}
    \rk \bL_{p,B}
    =
    r\cdot\binom{n+\ell}{n} - \dim \gr \ker \bL_{p,B}
    \geq
    r\cdot\binom{n+\ell}{n} - \dim \ker \gr \bL_{p,B}
    =
    \rk \gr \bL_{p,B}.
  \end{equation*}  
  The isomorphism \autoref{Eq_TaylorExpansionIsomorphism} identifies $[B]$ with $\hat B \in \ker\hat\varpi_\sigma$, which is homogeneous of degree $d$.
  Furthermore, it identifies $\gr \bL_{p,B}$ with $\hat\bL_{\sigma,\hat B}^{\leq \ell}$.
  Therefore,
  \begin{equation*}
    \rk \rd_{(p,B)}(J_x^{k+\ell}\varpi\circ\pr_2)
    \geq
    \rk \hat\bL_{\sigma,\hat B}^{\leq \ell}
    \geq
    c_0(d,\rho)\ell^n.
  \end{equation*}
  This finishes the proof.
\end{proof}

%%% Local Variables:
%%% mode: latex
%%% TeX-master: "EquivariantBrillNoetherSuperRigidity"
%%% End:

\begin{appendices}
  \counterwithin{section}{part}
  \renewcommand{\thesection}{\arabic{part}.\Alph{section}}
  \section{Self-adjoint operators}
\label{Sec_SelfAdjointOperators}

The material contained in this section is not needed in \autoref{Part_Application}.

\begin{definition}
  \label{Def_FamilyOfSelfAdjointEllipticOperators}
  Let $k \in \N_0$.
  A \defined{family of self-adjoint linear elliptic differential operators} of order $k$ consists of
  a Banach manifold $\sP$ and
  a smooth map
  \begin{equation*}
    D \co \sP \to \sL(W^{k,2}\Gamma(E),L^2\Gamma(E))
  \end{equation*}
  such that for every $p \in \sP$ the operator $D_p \coloneq D(p)$ is the extension of a self-adjoint linear elliptic differential operator of order $k$.
\end{definition}

Throughout this section,
assume \autoref{Sit_V}.
Thealgebras $\bK_\alpha$ carry an anti-involution $\lambda \mapsto \lambda^*$ and an inner product $\Inner{\lambda,\mu} \coloneq \tr(\mu^*\lambda)$.
(These correspond to the standard conjugation and inner products on $\R$\, $\C$, and $\H$.)
The Euclidean metric on $\uV_\alpha$ is $\bK_\alpha$--sesquilinear.
Let $\bK$ be such a commuting algebra and let $V$ be a left $\bK$--module equipped with a $\bK$--sesquilinear inner product.
Denote by $\Sym_\bK(V)$ the space of self-adjoint $\bK$--linear map.
$V$ is a right $\bK$--module with $v\cdot \lambda \coloneq \lambda^* \cdot v$.
Therefore, one can form the tensor product $V\otimes_{\,\bK} V$ and the symmetric tensor product $S_\bK^2V$.

Here is the analogue of the theory developed in \autoref{Sec_EquivariantBrillNoetherLoci_Twists}.

\begin{definition}
  Let $(D_p)_{p \in \sP}$ be a family of self-adjoint linear elliptic differential operators.
  For
  $d \in \N_0^m$
  define the \defined{$\fV$--equivariant self-adjoint Brill--Noether locus} $\sP_d^\fV$ by
  \begin{equation*}
    \sP_d^\fV
    \coloneq \set*{ p \in \sP : \dim_{\bK_i}\ker D_p^{\uV_i} = d_i }.
    \qedhere
  \end{equation*}  
\end{definition}

\begin{theorem}
  \label{Thm_EquivariantSelfAdjointBrillNoetherLoci}
  Let $(D_p)_{p \in \sP}$ be a family of linear elliptic differential operators.
  Let $d,e \in \N_0^m$.
  If for every $p \in \sP$ the map $\Lambda_p^\fV\co T_p\sP \to \bigoplus_{\alpha=1}^m \Sym_{\bK_\alpha}(\ker D_p^{\uV_\alpha})$ defined in \autoref{Thm_EquivariantBrillNoetherLoci_Twists} is surjective, then the following hold:
  \begin{enumerate}
  \item
    $\sP_d^\fV$ is a submanifold of codimension
    \begin{equation*}
      \codim \sP_d^\fV =\sum_{\alpha=1}^m d_\alpha + k_\alpha\binom{d_\alpha}{2}.
    \end{equation*}    
  \item
    If $\sP_d^\fV \neq \emptyset$,
    then $\sP_{\tilde d}^\fV \neq \emptyset$ for every $\tilde d \in \N_0^m$ with $\tilde d \leq d$.
  \end{enumerate}
\end{theorem}

\begin{proof}
  There is a straight-forward variation of \autoref{Lem_LyapunovSchmidtReduction} to self-adjoint Fredholm operators.
  This reduces the proof to the finite-dimensional situation.
  The latter is straightforward. 
  The codimension formula follows from
  \begin{equation*}
    \dim \Sym_\bK(\bK^d) = d + k \binom{d}{2}
  \end{equation*}
  with $k \coloneq \dim_\R \bK$.
\end{proof}

\begin{definition}
  \label{Def_EquivariantlySelfAdjointFlexible}
  A family of linear elliptic differential operators $(D_p)_{p \in \sP}$ is \defined{$\fV$--equivariantly symmetrically flexible in $U$} if for every $p \in \sP$ and $A \in \Gamma(\Sym(E))$ supported in $U$
  there is a $\hat p \in T_p\sP$ such that
  \begin{equation*}
    \rd_p D^{\uV_\alpha}(\hat p)s = (A\otimes \id_{\uV_\alpha})s \mod \im D_p^{\uV_\alpha}
  \end{equation*}
  for every $\alpha = 1,\ldots,m$ and $s \in \ker D_p^{\uV_\alpha}$.
\end{definition}

\begin{definition}
  \label{Def_EquivariantlySymmetricallyFlexible}
  A family of linear elliptic differential operators $(D_p)_{p \in \sP}$ is \defined{strongly symmetrically flexible in $U$} if for every $p \in \sP$ there is a $C^0$--dense subset $\sH_p \subset \Gamma_c(U,\Sym(E))$ such that for every $A \in \sH_p$ there is a $\hat p \in T_p\sP$ such that
  \begin{equation*}
    \rd_p D(\hat p)s = As.
  \end{equation*}
  for every $s \in \Gamma(E)$
\end{definition}

\begin{prop}
  \label{Prop_StronglySymmetricallyFlexible=>EquivariantlySymmetricallyFlexible}
  If $(D_p)_{p \in \sP}$ is strongly flexible in $U$,
  then it is $\fV$--equivariantly symmetrically flexible in $U$.
  \qed
\end{prop}

\begin{definition}
  \label{Def_EquivariantSymmetricPetriCondition}
  The \defined{$\fV$--equivariant symmetric Petri map}
  \begin{equation*}
    \varsigma^\fV \co \bigoplus_{\alpha=1}^m S_{\bK_\alpha}^2 \Gamma(E\otimes \uV_\alpha) \to \Gamma(S^2E)
  \end{equation*}
  is defined by $\varpi^\fV \coloneq \sum_{\alpha=1}^m \varpi_\alpha$ with $\varpi_\alpha$ denoting the composition of the Petri map
  \begin{equation*}
    \varsigma_\alpha\co S_{\bK_\alpha} ^2\Gamma(E\otimes \uV_\alpha) \to \Gamma(S^2 E \otimes S_{\bK_\alpha}^2 \uV_\alpha)
  \end{equation*}
  and the map induced by the inner product $\inner{\cdot}{\cdot}\co S_{\bK_\alpha}^2 \uV_\alpha \to \ubR$.
  Let $U \subset M$ be an open subset.
  A self-adjoint linear elliptic differential operator $D\co \Gamma(E) \to \Gamma(E)$ satisfies the \defined{$\fV$--equivariant symmetric Petri condition in $U$} if the map
  \begin{equation*}
    \varsigma_{D,U}^\fV
    \co
    \bigoplus_{\alpha=1}^m S_{\bK_\alpha}^2 \ker D_p^{\uV_\alpha}
    \to
    \Gamma(U,S^2E)
  \end{equation*}
  induced by the $\fV$--equivariant symmetric Petri map is injective.
\end{definition}

\begin{prop}
  \label{Prop_EquivarianSelfAdjointFlexible+Petri=>Surjective}
  Let $(D_p)_{p \in \sP}$ be a family of self-adjoint linear elliptic differential operators and let $p \in \sP$.
  Let $U \subset M$ be an open subset.
  If $(D_p)_{p \in \sP}$ is $\fV$--equivariantly symmetrically flexible in $U$ and
  $D_p$ satisfies the $\fV$--equivariant symmetric Petri condition in $U$,
  then the map $\Lambda_p^\fV$ defined in \autoref{Thm_EquivariantSelfAdjointBrillNoetherLoci} is surjective.
  \qed
\end{prop}

\autoref{Rmk_EquivariantBrillNoetherLoci_EstimateCodimension} carries over mutatis mutandis;
in particular, \autoref{Eq_SumRhoCondition} it is still sharp for self-adjoint operators.
Finally,
these are the analogues of the results from \autoref{Sec_PetrisConditionRevisited}.

\begin{definition}  
  The \defined{polynomial symmetric Petri map}
  $\hat\varsigma\co S^2(S^\bullet T_x^* M\otimes E_x)\to S^\bullet T_x^* M\otimes S^2E_x$ is defined as the restriction of the polynomial Petri map.
  A symmetric symbol $\sigma \in S^kT_xM\otimes \Sym(E_x)$ satisfies the \defined{polynomial symmetric Petri condition} if the map
  \begin{equation*}
    \hat\varsigma_\sigma\co S^2\ker \hat\sigma \to S^\bullet T_x^* M\otimes S^2E_x
  \end{equation*}
  induced by the polynomial symmetric Petri map is injective.
\end{definition}

\begin{prop}
  \label{Prop_PolynomialSymmetricPetri=>JetSymmetricPetri}
  If $J_x^\infty D$ fails to satisfy the $\infty$--jet symmetric Petri condition,
  then $\sigma_x(D)$ fails to satisfy the polynomial symmetric Petri condition.
  \qed
\end{prop}

\begin{definition}
  Let $k,\ell \in \N_0$
  A family of self-adjoint linear elliptic differential operators $(D_p)_{p\in\sP}$ of order $k$ is \defined{$\ell$--jet strongly symmetrically flexible at $x$} if for every $p\in\sP$ and $A \in J_x^\ell\Sym(E_x,F_x)$ there is a $\hat p\in T_p\sP$ such that
  \begin{equation*}
    \rd_pJ_x^\ell D(\hat p) s = As
  \end{equation*}
  for every $s \in J_x^{k+\ell} E$.
\end{definition}

\begin{definition}
  \label{Def_SymmetricWendlCondition}
  Let $k \in \N_0$ and let $\sigma \in S^kT_xM\otimes\Sym(E_x)$ be a symbol.
  Let $c_0\co \N_0\times\N \to (0,\infty)$ and $\ell_0\co \N_0\times \N \to \N_0$.
  The symbol $\sigma$ satisfies the \defined{the symmetric Wendl condition for $c_0$ and $\ell_0$} if for every homogeneous $B \in \ker \hat\varsigma_\sigma$ there is a right-inverse $\hat R$ of $\hat \sigma$ such that the linear map
  $\hat\bL_{\sigma,B}\co S^\bullet T_x^*M\otimes \Sym(E_x) \to S^\bullet T_x^*M\otimes S^2 E_x$
  defined by
  \begin{equation*}
    \hat\bL_{\sigma,B}(A)
    \coloneq   
    \hat\varsigma
    \paren[\big]{
      (\hat RA \otimes \id
      +
      \id\otimes \hat R A)
      B
    }
  \end{equation*}
  satisfies
  \begin{equation*}
    \rk \hat\bL_{\sigma,B}^{\leq\ell}
    \geq
    c_0(d,\rho)\ell^n
  \end{equation*}
  for every $\ell \geq \ell_0(d,\rho)$
  with $d \coloneq \deg(B)$ and $\rho \coloneq \rk B$.
  Here
  \begin{equation*}
    \hat\bL_{\sigma,B}^{\leq\ell}
    \co
    \bigoplus_{j=0}^\ell S^j T_x^*M\otimes\Sym(E_x)
    \to
    \bigoplus_{j=0}^{k+\ell} S^j T_x^*M\otimes S^2E_x
  \end{equation*}
  denotes the truncation of $\bL_{\sigma,B}$.
\end{definition}

\begin{theorem}
  \label{Thm_SymmetricPetrisConditionFailsInInfiniteCodimension}
  Let $(D_p)_{p\in\sP}$ be a family of self-adjoint linear elliptic differential operators. 
  If 
  \begin{enumerate}
  \item
    $(D_p)_{p\in\sP}$ is $\ell$--jet strongly symmetrically flexible at $x$ for every $\ell \in \N_0$, and
    \item 
    there are $c_0\co \N_0\times\N \to (0,\infty)$ and $\ell_0\co \N_0\times \N \to \N_0$ such that for every $p \in \sP$ the symbol $\sigma_x(D_p)$ satisfies symmetric Wendl condition for $c_0$ and $\ell_0$,
  \end{enumerate}
  then the subset
  \begin{equation*}
    \sR
    \coloneq
    \set{
      p \in \sP
      :
      J_x^\infty D
      \textnormal{ fails to satisfy the $\infty$--jet symmetric Petri condition}
    }
  \end{equation*}
  has infinite codimension.
\end{theorem}

\begin{proof}
  The proof of \autoref{Thm_PetrisConditionFailsInInfiniteCodimension} carries over with minor changes.
  The salient point is that
  \begin{equation*}
    \sT_{d,\rho}^\ell
    \coloneq
    \set*{
      (p,B) \in S^2\sK^\ell
      :
      \ord(B) \leq d
      \text{ and }
      \rk B = \rho
    }.
  \end{equation*}
  is a fiber bundle over $\sP$ of with fibers of dimension at most $c(n,k) \rho \ell^{n-1}$.
\end{proof}

%%% Local Variables:
%%% mode: latex
%%% TeX-master: "EquivariantBrillNoetherSuperRigidity"
%%% End:

  %!TEX root = EquivariantBrillNoetherSuperRigidity.tex

\section{Codimension in Banach manifolds}
\label{Sec_CodimensionInBanachManifolds}

There are numerous possible definitions of the concept of codimension of a subset of a Banach manifold.
The following is a minor variation of the definition from \cite[Section 2.3]{Bernard2015} and particularly well-suited for the purposes of this article. 

\begin{definition}
  \label{Def_Codimension}
  Let $X$ be a Banach manifold and $c \in \N_0$.
  A subset $S \subset X$ has \defined{codimension at least $c$} if there is a $C^1$ Banach manifold $Z$ and a $C^1$ Fredholm map $\zeta \co Z \to X$ such that
  \begin{equation*}
    \sup_{z \in Z} \ind \rd_z \zeta \leq -c \qandq
    S \subset \im\zeta.
  \end{equation*}
  The \defined{codimension of $S$} is defined by
  \begin{equation*}
    \codim S
    \coloneq
    \sup\set{ c \in \N_0 : S \text{ is of codimension at least } c }
    \in \N_0\cup\set{\infty}.
    \qedhere
  \end{equation*}
\end{definition}

The additivity of Fredholm indices implies the following.

\begin{prop}
  \label{Prop_Codimension_Image}
  Let $X,Y$ be Banach manifolds.
  If $S \subset X$ and $f\co X\to Y$ is a Fredholm map,
  then
  \pushQED{\qed}
  \begin{equation*}
    \codim f(S) \geq \codim S - \inf_{x\in X}\ind \rd_xf.
    \qedhere
  \end{equation*}
  \popQED
\end{prop}

The codimension of a subset can be regarded as a measure of the non-genericity of its elements.
In topology, one considers the following concepts.

\begin{definition}
  Let $X$ be a topological space and $S \subset X$.
  $S$ is \defined{meager} if it is contained in a countable union of closed subsets with empty interior.
  $S$ is \defined{comeager} if $X\setminus S$ is meager.
\end{definition}

Recall from \autoref{Footnote_BanachManifolds} that Banach manifolds are assumed to be Hausdorff, paracompact, and separable.
The Baire category theorem asserts that a meager subset of a completely metrizable space (e.g., a Banach manifold) has empty interior or, equivalently, that every comeager subset of such a space is dense.
In light of this,
one often regards a meager subset as consisting of non-generic points and
a comeager subset as consisting of generic points.

\begin{prop}
  \label{Prop_Codimension_Meager}
  Let $X$ be a Banach manifold and $S \subset X$.
  If $\codim S > 0$, then $S$ is meager.%
  \footnote{%
    The following stronger statement, which will not be used in this article, follows from Sard's theory of cotypes \cite{Sard1969}:
    $S$  is contained in the countable union of closed subsets, none of which contains a submanifold of codimension $\codim S-1$.
    In particular: since such closed subsets have empty interior, this condition implies that $S$ is meager.
  }
\end{prop}

\begin{proof}
  Let $Z$ and $\zeta$ be as in \autoref{Def_Codimension}.
  Since $\ind \rd_z\zeta < 0$,
  by the Sard--Smale Theorem \cite[Theorem 1.3]{Smale1965} $\im \zeta$ is meager;
  hence, so is $S$.
\end{proof}

In practice, one often proves that a subset is meager by proving that it has positive codimension.
The latter, however, yields more precise information.

\begin{prop}
  \label{Prop_Codimension_Map}
  Let $M$ be a finite-dimensional manifold and let
$X$ be a Banach manifold.
  For every $S \subset X$ and $k \in \N$ the following hold: 
  \begin{enumerate}
  \item
    \label{Prop_Codimension_Map_CodimensionLowerBoundInCk}
    The subset $\tilde S \subset C^k(M,X)$ consisting of those $f$ such that $f^{-1}(S) \neq \emptyset$ satisfies
    \begin{equation*}
      \codim \tilde S \geq \codim S - \dim M.
    \end{equation*}
  \item
    \label{CkProp_Codimension_Map_CodimensionLowerBoundForGenericMaps}
    The subset consisting of those $f \in C^k(M,X)$ for which $ \codim f^{-1}(S) \geq \codim S$  is comeager.
    \end{enumerate}
\end{prop}

\begin{proof}
  Suppose that $S$ has codimension at least $c$ and let $Z$ and $\zeta$ be as in \autoref{Def_Codimension}.
  Set $F \coloneq M \times C^k(M,X)$.
  The evaluation map $\ev\co F \to X$ is a $C^k$ submersion.
  Therefore,
  \begin{equation*}
    \ev^*Z
    \coloneq
    \set[\big]{ (x,f;z) \in F \times Z : \ev(x,f) = \zeta(z) }
  \end{equation*}
  is a $C^k$ Banach manifold and the map $\pr_1 \co \ev^*Z \to F$ is a Fredholm map of index at most $-c$.
  The projection map $\pr_2\co F \to C^k(M,X)$ is a Fredholm map of index $\dim M$.
   
  To prove  \autoref{Prop_Codimension_Map_CodimensionLowerBoundInCk},
  observe that $\ev^{-1}(S) \subset \im \pr_1$.
  Therefore, $\codim \ev^{-1}(S) \geq c$;
  hence, by \autoref{Prop_Codimension_Image},
  $\tilde S = \pr_2(\ev^{-1}(S))$ has codimension at least $c-\dim M$. 

  If $f \in C^k(M,X)$ is a regular value of $\pr_2 \circ \pr_1 \co \ev^*Z \to C^k(M,X)$,
  then $(\pr_2 \circ \pr_1)^{-1}(f)$ is a $C^k$ submanifold of $\ev^*Z$ of dimension at most $\dim M - c$.  
  Therefore, its projection to $M$ has codimension at least $c$.
  A moment's thought shows that this projection is $f^{-1}(\im\zeta)$; hence, it contains $f^{-1}(S)$.
  Therefore, $\codim f^{-1}(S) \geq \codim f$.
  By the Sard--Smale Theorem, the set of regular values of $\pr_2\circ\pr_1$ is comeager.
  This implies \autoref{CkProp_Codimension_Map_CodimensionLowerBoundForGenericMaps}.
\end{proof}

%%% Local Variables:
%%% mode: latex
%%% TeX-master: "EquivariantBrillNoetherSuperRigidity"
%%% End:

\end{appendices}

\part{Application to super-rigidity}
\label{Part_Application}

%!TEX root = EquivariantBrillNoetherSuperRigidity.tex

\section{Bryan and Pandharipande's super-rigidity conjecture}
\label{Sec_SuperRigidity}

The notion of super-rigidity for holomorphic maps was first introduced in algebraic geometry by \citet[Section 1.2]{Bryan2001}.
The purpose of this section is to recall the corresponding notion in symplectic geometry as defined by \citet[Section 1]{Eftekhary2016} and \citet[Section 2.1]{Wendl2016}.

\begin{definition}
  \label{Def_JHolomorphicMap}
  Let $(M,J)$ be an almost complex manifold.
  A \defined{$J$--holomorphic map} $u \co (\Sigma,j) \to (M,J)$ is a pair consisting of
  a closed, connected Riemann surface $(\Sigma,j)$ and
  a smooth map $u\co \Sigma \to M$ satisfying the non-linear Cauchy--Riemann equation
  \begin{equation}
    \label{Eq_JHolomorphic}
    \delbar_J(u,j) \coloneq \frac12(\rd u + J(u)\circ \rd u\circ  j) = 0.
  \end{equation}
  
  Let $u \co (\Sigma,j) \to (M,J)$ be a $J$--holomorphic map.
  Let $\phi \in \Diff(\Sigma)$ be a diffeomorphism.
  The \defined{reparametrization} of $u$ by $\phi$ is the $J$--holomorphic map $u\circ\phi^{-1}\co (\Sigma,\phi_*j) \to (M,J)$.
  
  If $\pi\co (\tilde \Sigma,\tilde j) \to (\Sigma,j)$ is a holomorphic map of degree $\deg(\pi) \geq 2$,
  then the composition $u\circ \pi \co (\tilde\Sigma,\tilde j) \to (M,J)$ is said to be a \defined{multiple cover of $u$}.
  A $J$--holomorphic map is \defined{simple} if it is not constant and not a multiple cover.
\end{definition}

Super-rigidity is a condition on the infinitesimal deformation theory of the images of $J$--holomorphic maps (up to reparametrization).
To give the precise definition,
let us recall the salient parts of this theory.
This material is standard and details can be found, for example, in \cite[Chapter 3]{McDuff2012} and \cite{Wendl2016a}.

Let $(M,J)$ be an almost complex manifold and let $u \co (\Sigma,j) \to (M,J)$ be a non-constant $J$--holomorphic map.
Set
\begin{equation*}
  \Aut(\Sigma,j)
  \coloneq
    \set*{ \phi \in \Diff(\Sigma) : \phi_*j = j }
  \qandq
  \aut(\Sigma,j)
  \coloneq
    \set*{ v \in \Vect(\Sigma) : \sL_v j = 0 }.
\end{equation*}
Let $\sS$ be an $\Aut(\Sigma,j)$--invariant slice of the Teichmüller space $\sT(\Sigma)$ through $j$.
Denote by $\rd_{u,j}\delbar_J \co \Gamma(u^*TM)\oplus T_j\sS \to \Omega^{0,1}(u^*TM)$ the linearization of $\delbar_J$ at $(u,j)$ restricted to $C^\infty(\Sigma,M)\times \sS$.
The action of $\Aut(\Sigma,j)$ on $C^\infty(M)\times \sS$ preserves $\delbar_J^{-1}(0)$.
Therefore,
there is an inclusion $\aut(\Sigma,j) \into \ker \rd_{u,j}\delbar_J$.
The moduli space of $J$--holomorphic maps up to reparametrization containing $[u,j]$  has virtual dimension
\begin{equation*}
  \ind \rd_{u,j}\delbar_J - \dim\aut(\Sigma,j)
  =
  (n-3)\chi(\Sigma) + 2\inner{[\Sigma]}{u^*c_1(M,J)}.
\end{equation*}
 
\begin{definition}
  Let $(M,J)$ be an almost complex manifold of dimension $2n$.
  The \defined{index} of a $J$--holomorphic map $u \co (\Sigma,j) \to (M,J)$ is
  \begin{equation}
    \label{Eq_Index}
    \ind(u) \coloneq (n-3)\chi(\Sigma) + 2\inner{[\Sigma]}{u^*c_1(M,J)}.
    \qedhere
  \end{equation}
\end{definition}

Infinitesimal deformations of $j$ do not affect $\im u$.
Therefore,
we restrict our attention to $\fd_{u,J}\co \Gamma(u^*TM) \to \Omega^{0,1}(\Sigma,u^*TM)$,
the restriction of $\rd_{u,j}\delbar_J$ to $\Gamma(u^*TM)$.
A brief computation shows that
\begin{equation}
  \label{Eq_DU}
  \fd_{u,J}\xi
  =
  \frac12\paren{
    \nabla \xi + J\circ (\nabla \xi) \circ j + (\nabla_\xi J)\circ \rd u\circ j
  }.
\end{equation}
Here $\nabla$ denotes any torsion-free connection on $TM$ and also the induced connection on $u^*TM$. 
If $(u,j)$ is a $J$--holomorphic map,
then the right-hand side of \autoref{Eq_DU} does not depend on the choice of $\nabla$;
see \cite[Proposition 3.1.1]{McDuff2012}.
The operator $\fd_{u,J}$ has the property that if $\xi \in \Gamma(T\Sigma)$, then $\fd_{u,J}(\rd u(\xi))$ is a $(0,1)$--form taking values in $\rd u(T\Sigma) \subset u^*TM$.
If $u$ is non-constant,
then there is a unique complex subbundle
\begin{equation*}
  Tu \subset u^*TM
\end{equation*}
of rank one containing $\rd u(T\Sigma)$ \cite[Section 1.3]{Ivashkovich1999}; see also \cite[Section 3.3]{Wendl2010} and \autoref{Sec_NormalCauchyRiemannOperator} for a detailed discussion.
Since $Tu$ agrees with $\rd u(T\Sigma)$ outside finitely many points,  $\fd_{u,J}$ maps $\Gamma(Tu)$ to $\Omega^{0,1}(\Sigma,Tu)$.
Infinitesimal deformations along $\Gamma(Tu)$ also do not affect $\im u$.
This leads us to the following.

\begin{definition}
  \label{Def_NormalCauchyRiemannOperator}
  Let $(M,J)$ be an almost complex manifold.
  Let $u\co (\Sigma,j) \to (M,J)$ be a non-constant $J$--holomorphic map.
  Set
  \begin{equation*}
    Nu \coloneq u^*TM/Tu.
  \end{equation*}
  The \defined{normal Cauchy--Riemann operator} associated with $u$ is the linear map
  \begin{equation*}
    \fd_{u,J}^N \co \Gamma(Nu) \to \Omega^{0,1}(\Sigma, Nu)
  \end{equation*}
  induced by $\fd_{u,J}$.
\end{definition}

The following illuminates the role of the normal Cauchy--Riemann operator in the infinitesimal deformation theory of $J$--holomorphic maps.

\begin{prop}[{\cites[Lemma 1.5.1]{Ivashkovich1999}[Theorem 3]{Wendl2010}}; see also \autoref{Sec_NormalCauchyRiemannOperator}]
  \label{Prop_DDelbarVSDN}  
  Let $(M,J)$ be an almost complex manifold.
  Let $u\co (\Sigma,j) \to (M,J)$ be a non-constant $J$--holomorphic map.
  Denote by $Z(\rd u)$ the number of critical points of $u$ counted with multiplicity.
  The following hold:
  \begin{enumerate}
  \item
    There is a surjection
    \begin{equation*}
      \ker \rd_{u,j}\delbar_J \onto \ker \fd_{u,J}^N
    \end{equation*}
    whose kernel contains $\aut(\Sigma,j)$ and has dimension $\dim \aut(\Sigma,j) + 2Z(\rd u)$.%
    \footnote{%
      The summand $2Z(\rd u)$ corresponds to infinitesimally deforming the location of the critical points of $u$ without deforming $\im u$.
    }
  \item
    There is an isomorphism
    \begin{equation*}
      \coker \rd_{u,j}\delbar_J \iso \coker \fd_{u,J}^N.
    \end{equation*}
  \item
    The index of $\fd_{u,J}^N$ satisfies
    \begin{equation*}
      \ind \fd_{u,J}^N = \ind(u) - 2Z(\rd u) \leq \ind(u).
    \end{equation*}
  \end{enumerate}
\end{prop}

Finally,
everything is in place to define super-rigidity.

\begin{definition}
  \label{Def_RigidMap}
  Let $(M,J)$ be an almost complex manifold.
  A non-constant $J$--holomorphic map $u$ is \defined{rigid} if $\ker \fd_{u,J}^N = 0$.
\end{definition}

A multiple cover $\tilde u$ of $u$ may fail to be rigid,
even if $u$ itself is rigid.

\begin{definition}
  \label{Def_SuperRigidMap}
  Let $(M,J)$ be an almost complex manifold.
  A simple $J$--holomorphic map $u\co (\Sigma,j) \to (M,J)$ is called \defined{super-rigid} if it is rigid and all of its multiple covers are rigid.
\end{definition}

For a generic $J$ every simple $J$--holomorphic map satisfies $\ind(u) \geq 0$ and is an immersion.
Therefore, by \autoref{Prop_DDelbarVSDN}, simple $J$--holomorphic map, notion of super-rigidity is interesting only for simple $J$--holomorphic map $u$ with $\ind(u) = 0$.

\begin{definition}
  \label{Def_SuperRigidAlmostComplexStructure}
  Let $M$ be a manifold of dimension at least six.
  An almost complex structure $J$ on $M$ is called \defined{super-rigid} if the following hold:
  \begin{enumerate}
  \item
    \label{Def_SuperRigidAlmostComplexStructure_NonNegativeIndex}
    Every simple $J$--holomorphic map has non-negative index.
  \item
    \label{Def_SuperRigidAlmostComplexStructure_Embedding}
    Every simple $J$--holomorphic map of index zero is an embedding, and
    every two simple $J$--holomorphic maps of index zero either have disjoint images or are related by a reparametrization.
  \item
    \label{Def_SuperRigidAlmostComplexStructure_SuperRigid}
    Every simple $J$--holomorphic map of index zero is super-rigid.
    \qedhere
  \end{enumerate}
\end{definition}

\begin{remark}
  If $\dim M = 4$, one should weaken condition \autoref{Def_SuperRigidAlmostComplexStructure_Embedding} and require only that every simple $J$--holomorphic map of index zero is an immersion with transverse self-intersections, and that two such maps are either transverse to one another or are related by a reparametrization.
  However, we will only be concerned with dimension at least six.
\end{remark}

Let $(M,\omega)$ be a symplectic manifold.
\citet[Section 1.2]{Bryan2001} conjectured that a generic almost complex structure compatible with $\omega$ is super-rigid.
This conjecture has recently been proved by \citet{Wendl2016a}.
This part of the present article is an exposition of \citeauthor{Wendl2016}'s proof using the theory developed in \autoref{Part_Theory}.
The precise statement of \citeauthor{Wendl2016}'s requires the following Banach manifold of almost complex structures introduced by \citet[Section 5]{Floer1988}; see also \cite[Remark 3.2.7]{McDuff2012} and \cite[Appendix B]{Wendl2016a}.%
\footnote{%
  The theory developed in \autoref{Sec_PetrisConditionRevisited} requires linear elliptic differential operators with smooth coefficients.
  Therefore,
  it is not possible to simply work with $C^k$ almost complex structures.
}

\begin{definition}
  \label{Def_FloerNorm}
  Let $(X,g)$ be a Riemannian manifold and let $E$ be an Euclidean vector bundle over $X$ equipped with an orthogonal connection.
  Let $\epsilon = (\epsilon_\ell)_{\ell \in \N_0}$ be a sequence in $(0,1)$.
  For $s \in \Gamma(E)$ set
  \begin{equation*}
    \Abs{s}_{C_\epsilon^\infty} \coloneq \sum_{\ell=0}^\infty \epsilon_\ell \Abs{\nabla^\ell s}_{C^0}.
  \end{equation*}
  The vector space $C_\epsilon^\infty\Gamma(E) \coloneq \set{ s \in \Gamma(E) : \Abs{s}_{C_\epsilon^\infty} < \infty}$ equipped with the norm $\Abs{\cdot}_{C_\epsilon^\infty}$ is a separable Banach space.
\end{definition}

\begin{definition}
  \label{Def_JSuperrigid}
  Let $(M,\omega)$ be a symplectic manifold,
  let $J_0$ be an almost complex structure on $M$ compatible with $\omega$, and
  let $\epsilon = (\epsilon_\ell)_{\ell \in \N_0}$ be a sequence in $(0,1)$.
  For $\delta > 0$
  \begin{equation*}
    \sU(J_0,\epsilon,\delta)
    \coloneq
    \set*{
      \hat J \in \Gamma(\End(TM))
      :
      J_0 \hat J + \hat J J_ 0 = 0,
      \omega(\hat J\cdot, \cdot) + \omega(\cdot,\hat J \cdot) = 0,
      \textnormal{ and }
      \Abs{\hat J}_{C_\epsilon^\infty} < \delta
    }
  \end{equation*}
  is a Banach manifold.
  If $\delta$ is sufficiently small,
  then
  $%
  \hat J
  \mapsto  
  \paren{\one+\frac12 J_0\hat J}
  J_0
  \paren{\one+\frac12 J_0\hat J}^{-1}%
  $
  defines a continuous inclusion of $\sU(J_0,\epsilon,\delta)$ into the Fréchet space of all almost complex structures compatible with $\omega$.
  Denote by
  \begin{equation*}
    \sJ(M,\omega;J_0,\epsilon,\delta)
  \end{equation*}
  the image of $\sU(J_0,\epsilon,\delta)$ under this embedding.
\end{definition}

Throughout the remainder of this article,
choices of $J_0$, $\epsilon$, and $\delta > 0$ are fixed and it is assumed that $\epsilon$ decays sufficiently fast (see \autoref{Prop_FloerNormStrongFlexibility}) and $\delta$ is sufficiently small.
With those choices being made,
set $\sJ(M,\omega) \coloneq \sJ(M,\omega;J_0,\epsilon,\delta)$ and denote by
\begin{equation*}
  \sJ_\superrigid(M,\omega)
\end{equation*}
the subset of those $J \in \sJ(M,\omega)$ which are super-rigid.

\begin{theorem}[\citet{Wendl2016}]
  \label{Thm_SuperRigidity}  
  Let $(M,\omega)$ be a symplectic manifold with $\dim M \geq 6$.
  The complement of $\sJ_\superrigid(M,\omega)$ in $\sJ(M,\omega)$ has codimension at least one.
\end{theorem}

The notion of codimension for a subsets of Banach manifolds is explained in \autoref{Sec_CodimensionInBanachManifolds}.
By \autoref{Prop_Codimension_Meager}, in particular, $\sJ_\superrigid(M,\omega)$ is dense in $\sJ(M,\omega)$.
However, it is not always possible to slightly perturb a path in $\sJ(M,\omega)$ to one contained in $\sJ_\superrigid(M,\omega)$.
This is discussed in detail in \autoref{Sec_SuperRigidityOneParameterFamilies}.

\begin{remark}
  Since $J_0$ is arbitrary,
  the subset of super-rigid almost complex structures is dense in the Fréchet space of all almost complex structures compatible with $\omega$.
  In fact,
  with a little more work one can show that this subset is comeager;
  cf. \cite[Theorem A]{Wendl2016}.
\end{remark}

The proof of \autoref{Thm_SuperRigidity} occupies the bulk of the remainder of \autoref{Part_Application}.
Throughout,
$(M,\omega)$ is a symplectic manifold of dimension $2n \geq 6$.
Let us immediately take care of \autoref{Def_SuperRigidAlmostComplexStructure_NonNegativeIndex} and \autoref{Def_SuperRigidAlmostComplexStructure_Embedding} in \autoref{Def_SuperRigidAlmostComplexStructure}.

\begin{definition}
  \label{Def_UniversalModuliSpace}
  Let $k\in \Z$.
  The \defined{universal moduli space of simple $J$--holomorphic maps of index $k$} is the space $\sM_k(M,\omega)$  of pairs $(J;[u,j])$ consisting of
  an almost complex structure $J \in \sJ(M,\omega)$, and
  an equivalence class of simple $J$--holomorphic maps $u\co (\Sigma,j) \to (M,J)$ of index $k$ up to reparametrization by $\Diff(\Sigma)$.
\end{definition}

The following is a standard transversality result for simple $J$--holomorphic maps; see, for example, \cite[Proposition 3.2.1]{McDuff2012}. 

\begin{theorem}
  \label{Thm_UniversalModuliSpace}
  Let $k \in \Z$.
   $\sM_k(M,\omega)$ is a Banach manifold and the projection map $\Pi\co \sM_k(M,\omega) \to \sJ(M,\omega)$ is a Fredholm map of index $k$.
\end{theorem}

By the index formula \autoref{Eq_Index},
$\sM_k(M,\omega) = \emptyset$ for odd $k$.
Therefore,
the subset
\begin{equation*}
  \sW_{\geq 0}(M,\omega)
  \coloneq
  \set{
    J \in \sJ(M,\omega) :
    \textnormal{%
      \autoref{Def_SuperRigidAlmostComplexStructure_NonNegativeIndex} in  \autoref{Def_SuperRigidAlmostComplexStructure} fails%
    }
  }
\end{equation*}
has codimension at least two.

\begin{theorem}[{\cites[Theorem 1.1]{Oh2009}[Proposition A.4]{Ionel2018}}]
  The subset
  \begin{equation*}
    \sW_{\incl}(M,\omega)
    \coloneq
    \set{
      J \in \sJ(M,\omega) :
      \textnormal{%
        \autoref{Def_SuperRigidAlmostComplexStructure_Embedding}
        in \autoref{Def_SuperRigidAlmostComplexStructure} fails%
      }    
    }
  \end{equation*}
  has codimension at least $2(n-2)$.
\end{theorem}

%%% Local Variables:
%%% mode: latex
%%% TeX-master: "EquivariantBrillNoetherSuperRigidity"
%%% End:

% !TEX root = EquivariantBrillNoetherSuperRigidity.tex

\section{Flexibility and Petri's condition}
\label{Sec_FlexibilityPetriCondition}

The objective of the next five sections is to prove that
\begin{equation*}
  \sW_{\superrigid}(M,\omega)
  \coloneq
  \set{
    J \in \sJ(M,\omega) :
    \textnormal{%
      \autoref{Def_SuperRigidAlmostComplexStructure_SuperRigid}
      in \autoref{Def_SuperRigidAlmostComplexStructure} fails%
    }    
  }
\end{equation*}
has codimension at least one.
This will be achieved using the theory developed in \autoref{Part_Theory} applied to certain families of elliptic operators which will be introduced in \autoref{Sec_RigidityOfUnbranchedCovers} and \autoref{Sec_CodimensionEstimateForBranchedCovers}. 
The present section ensures that $\fV$--equivariant flexibility and the $\fV$--equivariant Petri condition hold for these families and,
therefore, \autoref{Thm_EquivariantBrillNoetherLoci_Twists} can be applied.

The following observation implies strong flexibility in the application.
The reader should keep in mind that $\Omega^{0,1}(\Sigma,Nu) = \Gamma(F)$ with $F = \overline\Hom_\C(T\Sigma,Nu)$ denoting the bundle of complex anti-linear maps from $T\Sigma$ to $Nu$.

\begin{prop}
  \label{Prop_NormalCauchyRiemannFlexibility}
  Let $J \in \sJ(M,\omega)$.
  Let $u \co (\Sigma,j) \to (M,J)$ be a simple $J$--holomorphic map.  
  Consider the set of embedded points
  \begin{equation*}
    U \coloneq \set{ x \in \Sigma : u^{-1}(u(x)) = \set{x} \textnormal{ and } \rd_xu \neq 0 }.
  \end{equation*}
  For every
  \begin{equation*}
    A \in C_\epsilon^\infty\Gamma(\Hom(Nu,\overline{\Hom}_\C(T\Sigma,Nu))
  \end{equation*}
  with support in $U$
  there exists a $1$--parameter family $(J_t)_{t \in \R} \subset \sJ(M,\omega)$ such that:
  \begin{enumerate}
  \item
    $u$ is $J$--holomorphic with respect to all $J_t$, and
  \item
    $\left.\frac{\rd}{\rd t}\right|_{t=0} \fd_{u,J_t}^N = A$.
  \end{enumerate}
\end{prop}

\begin{proof}
  The tangent space to $\sJ(M,\omega)$ at $J$ is given by
  \begin{equation*}
    T_J\sJ(M,\omega)=
    \set*{
      \hat J \in C_\epsilon^\infty\Gamma(\End(TM))
      :
      \hat JJ + J\hat J = 0 \textnormal{ and } \omega(\hat J\cdot,\cdot) + \omega(\cdot,\hat J\cdot) = 0
    }.
  \end{equation*}
  This means that $T_J\sJ$ consists of anti-linear endomorphisms which are skew-adjoint with respect to $\omega$.
  For $x \in U$,
  $T_xM$ decomposes as $T_xM = T_x\Sigma\oplus N_x\Sigma$.
  Given $\hat j \in C_\epsilon^\infty\overline\Hom_\C(T\Sigma,Nu)$,
  denote by $\hat j^\dagger$ its adjoint with respect to $\omega$ and set
  \begin{equation*}
    \hat J
    \coloneq
    \begin{pmatrix}
      0 & -\hat j^\dagger \\
      \hat j & 0
    \end{pmatrix}.
  \end{equation*}
  By construction $\hat JJ + J\hat J = 0$ and $\omega(\hat J\cdot,\cdot) + \omega(\cdot,\hat J\cdot) = 0$;
  that is: $\hat J \in T_J\sJ(M,\omega)$.

  Given $A \in C_\epsilon^\infty\Gamma(\Hom(Nu,\overline{\Hom}_\C(T\Sigma,Nu))$ with support in $U$,
  pick $(J_t)_{t\in\R} \subset \sJ(M,\omega)$ such that $J_t|_{u(\Sigma)} = J$ for every $t$ and
  such that for every $\xi \in \Gamma(Nu)$
  \begin{equation*}    
    \frac12 \nabla_\xi \left.\frac{\rd}{\rd t}\right|_{t=0} J_t
    =    
    \begin{pmatrix}
      0 & \(A(\xi)j\)^\dagger \\
      -A(\xi)j & 0
    \end{pmatrix}.
  \end{equation*}
  By construction $u$ is $J$--holomorphic with respect to all $J_t$.
  It follows from \autoref{Eq_DU}
  that
  \begin{equation*}
    \left.\frac{\rd}{\rd t}\right|_{t=0} \fd_{u,J_t}^N = A.
    \qedhere
  \end{equation*}
\end{proof}

\begin{prop}
  \label{Prop_FloerNormStrongFlexibility}
  Assume the situation of \autoref{Def_FloerNorm}.
  Let $U \subset X$ be an open subset and let $x \in U$.
  If $\epsilon \in (0,1)^{\N_0}$ decays sufficiently fast,
  then the following hold:
  \begin{enumerate}
  \item
    \label{Prop_FloerNormStrongFlexibility_C0}
    $C_\epsilon^\infty\Gamma(E) \cap \Gamma_c(U,E)$ is $C^0$--dense in $\Gamma_c(U,E)$.
  \item
    \label{Prop_FloerNormStrongFlexibility_Jet}
    For every $\ell \in \N_0$ the map $J_x^\ell\co C_\epsilon^\infty\Gamma(U) \cap \Gamma_c(U,E) \to J_x^\ell E$ is surjective.   
  \end{enumerate}
\end{prop}

\begin{proof}
  For the proof of  \autoref{Prop_FloerNormStrongFlexibility_C0} see \cite[Theorem B.6(1)]{Wendl2016};
  it is essentially identical to the proof of \cite[Lemma 5.1]{Floer1988}.

  It suffices to prove \autoref{Prop_FloerNormStrongFlexibility_Jet} for $\R^n$, $U = B_r(0)$, and $x = 0$ with arbitrary $r \in (0,1]$.
  Choose a cut-off function $\chi \in C^\infty(\R^n,[0,1])$ with support in $B_1(0)$ and with $\chi = 1$ on $B_{1/2}(0)$.
  Set $\chi_r \coloneq \chi(\cdot/r)$.
  Let
  \begin{equation*}
    p = \sum_{\abs{\alpha} \leq d} c_\alpha x^\alpha
  \end{equation*}
  be a polynomial on $\R^n$.
  Evidently,
  $J_x^\ell(\chi_r p) = p$.
  Therefore,
  it remains to produce $\epsilon$ such that $\Abs{\chi_r p}_{C_\epsilon^\infty} < \infty$ for every $r \in (0,1]$ and $p$.
  For every $p$ of degree $d$
  \begin{equation*}
    \Abs{\chi_r p}_{C_\epsilon^\infty}
    \leq
    c_p
    \sum_{\ell=0}^\infty
    \epsilon_\ell
    c_\ell
    \Abs{\chi_r}_{C^\ell}
    \leq
    c_p
    \sum_{\ell=0}^\infty
    \epsilon_\ell
    c_\ell
    r^{-\ell}
    \Abs{\chi}_{C^\ell}
    \qwithq
    c_p \coloneq \Abs{p}_{C^d}.
  \end{equation*}
  Since $(r\ell)^{-\ell}$ is summable for every $r \in (0,1]$,  
  it suffices to impose the condition that
  \begin{equation*}
    \epsilon_\ell
    \leq    
    \frac{1}{c_\ell \ell^\ell \Abs{\chi}_{C^\ell}}.
    \qedhere
  \end{equation*}  
\end{proof}

The following and \autoref{Thm_PetrisConditionFailsInInfiniteCodimension} imply that the $\fV$--equivariant Petri condition holds away from a subset of infinite codimension in the application.

\begin{definition}
  \label{Def_RealCROperator}
  Let $(\Sigma,j)$ be a Riemann surface and let $E$ be a complex vector bundle over $\Sigma$.
  A \defined{real Cauchy--Riemann operator} on $E$ is a real linear first order elliptic differential operator $D \co \Gamma(E) \to \Omega^{0,1}(\Sigma, E)$ satisfying 
  \begin{equation}
    \label{Eq_CRLeibnizRule}
    \fd(fs) = \delbar f \otimes s + f \fd s
  \end{equation}
  for every $f \in C^\infty(\Sigma,\R)$ and $s \in \Gamma(E)$.
\end{definition}

By \autoref{Eq_DU},
the normal Cauchy--Riemann operator associated with a $J$--holomorphic map is real Cauchy--Riemann operator.

\begin{theorem}[{\citet[Section 5.3]{Wendl2016}}]
  \label{Thm_CauchyRiemannSymbolSatisfiesWendlsCondition}
  There are $c_0 \co \N_0\times \N \to (0,\infty)$ and $\ell_0 \co \N_0\times \N \to \N_0$ such that every symbol of a real Cauchy--Riemann operator satisfies Wendl's condition for $c_0$ and $\ell_0$.
\end{theorem}

Before embarking on the proof of this result,
let us remind the reader of the following fact.
Let $V$ and $W$ be complex vector spaces.
Denote by $\overline W$ the complex vector space $W$ with scalar multiplication $(\lambda,w)\mapsto \bar \lambda w$.
The tensor product $V\otimes W$ admits two commuting complex structures: $I_1 \coloneq i\otimes \one$ and $I_2 \coloneq \one\otimes i$.
$V\otimes W$ decomposes into the subspace on which $I_1=I_2$ and the subspace on which $I_1=-I_2$.
These can be identified with $V\otimes_\C W$ and $V\otimes_\C \overline W$;
hence:
\begin{equation*}
  V\otimes W = \paren{V\otimes_\C W} \oplus \paren{V\otimes_\C \overline W}.
\end{equation*}
The space of real linear maps $\Hom(V,W)$ admits two commuting complex structures given by pre- and post-composition with $i$.
This decomposes $\Hom(V,W)$ into the space of complex linear map $\Hom_\C(V,W)$ and the space of complex anti-linear maps $\overline\Hom_\C(V,W)$;
that is:
\begin{equation*}
  \Hom(V,W) = \Hom_\C(V,W) \oplus \overline{\Hom}_\C(V, W).
\end{equation*}

\begin{proof}[Proof of \autoref{Thm_CauchyRiemannSymbolSatisfiesWendlsCondition}]
  The symbol
  \begin{equation*}
    \sigma \coloneq \sigma_x(\fd)
  \end{equation*}
  at $x \in \Sigma$ of a real Cauchy--Riemann operator $\fd\co \Gamma(E) \to \Omega^{0,1}(\Sigma,E) = \Gamma(F)$ depends only on $E_x$.
  Denote by $z = s+it$ a local holomorphic coordinate around $x$ and identify $E_x = \C^r$ with $r \coloneq \rk_\C E$.
  Identifying
  \begin{equation*}
    \R[s,t]\otimes E_x
    =
    \R[s,t]\otimes F_x
    =
    \R[s,t]\otimes E_x^\dagger
    =
    \R[s,t]\otimes F_x^\dagger
    =
    \C[z,\bar z]\otimes_\C \C^r
  \end{equation*}
  the formal differential operators $\hat\sigma$ and $-\hat\sigma^\dagger$ both become
  \begin{equation*}
    \delbar\otimes_\C \id_{\C^r}
    \co \C[z,\bar z]\otimes_\C \C^r \to \C[z,\bar z]\otimes_\C \C^r.
  \end{equation*}
  Furthermore, identifying
  \begin{equation*}
    \C^r \otimes \C^r = \paren{\C^r \otimes_\C \C^r}\oplus \paren{\C^r \otimes_\C \bar \C^r}
  \end{equation*}
  via $v\otimes w \mapsto (v\otimes_\C w, v\otimes_\C \bar w)$ the polynomial Petri map $\hat\varpi$ becomes
  \begin{equation*}
    (\hat\varpi_1,\hat\varpi_2)
    \co
    \paren*{\C[z,\bar z]\otimes_\C \C^r}^{\otimes 2}
    \to
    \C[z,\bar z]\otimes_\C \paren{\C^r\otimes_\C \C^r}
    \oplus
    \C[z,\bar z]\otimes_\C \paren{\C^r\otimes_\C \bar\C^r}
  \end{equation*}
  defined by
  \begin{equation*}
    \hat\varpi_1(p,q) \coloneq pq \qandq
    \hat\varpi_2(p,q) \coloneq p\bar q.
  \end{equation*}
  From this it is evident that it suffices to consider the case $r = 1$ to prove \autoref{Thm_CauchyRiemannSymbolSatisfiesWendlsCondition}.

  \begin{prop}
    \label{Prop_BInKerHatVarpi}     
    If $B \in \ker \hat\varpi_\sigma$ is homogeneous of degree $d$,
    then it is of the form
    \begin{equation}
      \label{Eq_BInKerHatVarpi}
      B
      =
      \sum_{j=0}^d
      b_j \paren[\big]{z^j\, \otimes z^{d-j} - iz^j \otimes iz^{d-j}}
      + b_j' \paren[\big]{iz^j \otimes z^{d-j} + z^j \otimes iz^{d-j}}.
    \end{equation}
    with $b,b' \in \R^d$ satisfying
    \begin{equation}
      \label{Eq_BInKerHatVarpi_Cofficients}
      \sum_{j=0}^d b_j = 0 \qandq \sum_{j=0}^d b_j' = 0.
    \end{equation}
  \end{prop}

  \begin{proof}
    Every homogeneous $B \in \ker \delbar \otimes \ker \delbar$ of degree $d$ is of the form
    \begin{equation*}
      B
      =
      \sum_{j=0}^d
      b_j z^j\, \otimes z^{d-j}
      + b_j' iz^j \otimes z^{d-j}
      + b_j'' z^j \otimes iz^{d-j}
      + b_j''' iz^j \otimes iz^{d-j}
    \end{equation*}
    with $b,b',b'',b''' \in \R^d$.
    $B$ satisfies $\hat\varpi_2(B) = 0$ if and only if for every $j=0,\ldots,d$
    \begin{equation*}
      b_j + b_j''' = 0 \qandq b_j' - b_j'' = 0;
    \end{equation*}
    that is: $B$ is of the form \autoref{Eq_BInKerHatVarpi}.
    If $B$ is of this form, then $\hat\varpi_1(B) = 0$ is equivalent to   \autoref{Eq_BInKerHatVarpi_Cofficients}.
  \end{proof}

  Henceforth, let $B \in \ker\varpi_\sigma$ be homogeneous of degree $d$.
  The right-inverse of $\hat\sigma (=-\hat\sigma^\dagger)$ can be chosen as
  \begin{equation}
    \label{Eq_RightInverseOfCRSymbol}
    \hat R (z^\alpha \bar z^\beta) =  \frac{1}{\beta+1}z^\alpha\bar z^{\beta+1}.
  \end{equation}
  Define the map $\hat\bL_{\sigma,B} \co \C[z,\bar z] \otimes_\C \Hom(\C,\C) \to \C[z,\bar z]\otimes_\C (\C \otimes \C)$ by
  \begin{equation*}
    \hat\bL_{\sigma, B}(A)
    \coloneq
    \hat\varpi
    \paren[\big]{
      (\hat RA \otimes \one
      +
      \one\otimes \hat R^\dagger A^\dagger)
      B
    }
  \end{equation*}
  as in \autoref{Def_WendlsCondition}.
  $\Hom(\C,\C)$ and $\C\otimes\C$ decompose as
  \begin{equation*}
    \Hom(\C,\C) = \Hom_\C(\C,\C) \oplus \overline \Hom_\C(\C,\C) \qandq \C\otimes\C = (\C\otimes_\C \C) \oplus (\C \otimes_\C \overline\C).
  \end{equation*}
  This induces decompositions of the domain and codomain of $\hat\bL_{\sigma,B}$. 
  Each of the summands is isomorphic to $\C[z,\bar z]$. 
  With respect to these decompositions, $\hat\bL_{\sigma,B}$ is a matrix of four operators $\C[z,\bar z] \to \C[z, \bar z]$.
  Denote by
  \begin{equation*}
    \bQ_B\co \C[z,\bar z] \to \C[z,\bar z]
  \end{equation*}
  the bottom right component of $\hat\bL_{\sigma,B}$ that is:
  the restriction of $\hat\bL_{\sigma,B}$ to $\C[z,\bar z]\otimes_\C  \overline\Hom_\C(\C,\C)$ composed with the projection to $\C[z,\bar z]\otimes_\C \paren{\C\otimes_\C \overline\C}$.%
  \footnote{%    
    This operator is only apparently different from the one in \cite[Section 5.3]{Wendl2016}.
    The origin of this difference is that \citeauthor{Wendl2016} defines the formal adjoint $\fd^*$ using the Hermitian metric in contrast to our definition of $\fd^\dagger$ in \autoref{Def_FormalAdjoint}.
    Both operators are related by
    $\fd^\dagger = \bar\ast \circ \fd^* \circ \bar\ast$ with $\bar\ast \co \Lambda^{p,q}T^*\Sigma \otimes_\C E \to \Lambda^{1-p,1-q}T^*\Sigma \otimes_\C E^*$ denoting the anti-linear Hodge star operator.
  }
  (The other components of $\hat\bL_{\sigma,B}$ can be seen to vanish.)
  For $\ell \in \N_0$ denote $\C[z,\bar z]^{\leq \ell}$ the ring of polynomials in $z$ and $\bar z$ of degree at most $\ell$ and denote by $\bQ^{\leq \ell} \co \C[z,\bar z]^{\leq \ell} \to \C[z,\bar z]^{\leq \ell+1}$ the truncation of $\bQ_B$.
  The map $\bQ_B$ is complex linear.
  Since
  \begin{equation*}
    \rk \hat\bL_{\sigma,B}^{\leq\ell} \geq \rk_\C \bQ_B^{\leq \ell},
  \end{equation*}
  it suffices to estimate the latter.

  \begin{prop}
    The map $\bQ_B$ satisfies
    \begin{equation*}
      \rk_\C \bQ_B^{\leq \ell} \geq \frac{1}{16}\ell^2
    \end{equation*}
    for $\ell \geq 8d$.
  \end{prop}
  
  \begin{proof}    
    The map $\bQ_B$ can be computed explicitly.
    To do so,
    observe that
    $\overline\Hom_\C(\C,\C) = \C$ acts on $\C$ via $\lambda\cdot \mu \coloneq \lambda\bar\mu$ and its adjoint is $\lambda^\dagger \cdot \mu \coloneq \bar \lambda \bar\mu$;
    furthermore,
    recall that $\hat R^\dagger = - \hat R$ and that $\hat R$ is given by \autoref{Eq_RightInverseOfCRSymbol}.
    With this in mind it is easy to verify that for
    \begin{equation*}
      A = z^\alpha \bar z^\beta
      \qandq
      B
      =
      \sum_{j=0}^d
      b_j \paren[\big]{z^j\, \otimes z^{d-j} - iz^j \otimes iz^{d-j}}
      + b_j' \paren[\big]{iz^j \otimes z^{d-j} + z^j \otimes iz^{d-j}}
    \end{equation*}
    the map $\bQ_B$ satisfies
    \begin{align*}
      \bQ_B(A)
      &=
        \hat\varpi_2
        \paren[\big]{
        (\hat R A\otimes \one
        +
        \one\otimes \hat R^\dagger A^\dagger)
        B
        } \\
      &=
        \sum_{j=0}^d
        2(b_j-ib_j') \hat R(z^{\alpha}\bar z^{\beta+j}) \bar z^{d-j}
        -
        2(b_j+ib_j') z^j \overline{\hat R(z^\beta \bar z^{\alpha+d-j})} \\
      &=
        \sum_{j=0}^d
        \frac{2(b_j-ib_j')}{\beta+j+1} z^{\alpha}\bar z^{\beta+d+1}
        -
        \frac{2(b_j+ib_j')}{\alpha+d-j+1} z^{\alpha+d+1}\bar z^\beta \\
      &=
        \paren*{p_\beta^B \bar z^{d+1} + q_\alpha^B z^{d+1}}z^\alpha\bar z^\beta
    \end{align*}
    with
    \begin{equation*}
      p_\beta^B
      \coloneq
      \sum_{j=0}^d
      \frac{2(b_j-ib_j')}{\beta+j+1} \bar z^{d+1}
      \qandq
      q_\alpha^B
      \coloneq
      -
      \sum_{j=0}^d
      \frac{2(b_j+ib_j')}{\alpha+d-j+1} z^{d+1}.
    \end{equation*}
    The same formula holds for $\bQ_B^{\leq\ell}$ provided $\alpha+\beta+d \leq \ell$.

    Set
    \begin{equation*}
      S \coloneq
      \set*{
        (\alpha,\beta) \in \N_0^2
        :
        \alpha+\beta+d \leq \ell \textnormal{ and }
        (p_\beta^B,q_\alpha^B) \neq (0,0)
      }.
    \end{equation*}
    Choose a subset $S^\star \subset S$ such that $\#S^\star \geq \frac12\#S$ and such that if $(\alpha,\beta) \in S$,
    then $(\alpha-d-1,\beta+d+1) \notin S$.    
    The restriction of $\bQ_B^{\leq\ell}$ to $\C\Span{z^\alpha \bar z^\beta : (\alpha,\beta) \in S^\star }$ is injective.
    The latter is evident from the construction of $S^\star$ and
    \begin{equation*}
      \bQ_B\paren*{\sum_{\alpha,\beta \in S^\star} \lambda_{\alpha,\beta} z^\alpha\bar z^\beta}
      =
      \sum_{\alpha,\beta \in S^\star}
      \lambda_{\alpha,\beta}
      \paren*{
        p_\beta^B z^\alpha\bar z^{\beta+d+1}
        + q_\alpha^B z^{\alpha+d+1}\bar z^\beta
      }.
    \end{equation*}    
    Therefore,
    \begin{equation*}
      \rk_\C \bQ_B \geq \frac12 \#S.
    \end{equation*}

    It remains to find a lower bound on $\#S$. 
    At most $d$ of the numbers $p_0^B,p_1^B,\ldots$ are non-zero.
    This is a consequence of the following.
    If $\beta_0,\ldots,\beta_{d+1}$ are $d+1$ distinct positive numbers,
    then the matrix
    \begin{equation*}
      \begin{pmatrix}
        \frac{1}{\beta_0+1} & \frac{1}{\beta_0+2} & \ldots & \frac{1}{\beta_0+d+1} \\
        \frac{1}{\beta_0+1} & \frac{1}{\beta_1+2} & \ldots & \frac{1}{\beta_1+d+1} \\
        \ldots & \ldots & \ldots & \ldots \\
        \frac{1}{\beta_d+1} & \frac{1}{\beta_d+2} & \ldots & \frac{1}{\beta_d+d+1}
      \end{pmatrix}
    \end{equation*}
    is a Cauchy matrix and, therefore, invertible.
    If $p_{\beta_0}^B = \cdots = p_{\beta_d}^B = 0$,
    then the product of the above matrix with $(b_0-ib_0,\ldots,b_d-ib_d)$ would vanish:
    a contradiction.
    A variation of this argument shows that at most $d$ of the numbers $q_0^B,q_1^B,\ldots$ are non-zero.
    Therefore,
    \begin{equation*}
      \#S
      \geq
      \frac{(\ell - d)^2}{4} - (\ell-d)d.
    \end{equation*}
    This implies the assertion.    
  \end{proof}
  
  This finishes the proof of \autoref{Thm_CauchyRiemannSymbolSatisfiesWendlsCondition} with $c_0(\rho,d) = \frac{1}{16}$ and $\ell_0(\rho,d) = 8d$.
\end{proof}

%%% Local Variables:
%%% mode: latex
%%% TeX-master: "EquivariantBrillNoetherSuperRigidity"
%%% End:

%!TEX root = EquivariantBrillNoetherSuperRigidity.tex

\section{Rigidity of unbranched covers}
\label{Sec_RigidityOfUnbranchedCovers}

The purpose of this section is to prove the following.

\begin{prop}[{cf. \cite[Theorem 1.3]{Gerig2017}}]
  \label{Prop_UnbranchedCoversFailToBeRigidInCodimensionOne}
  Denote by $\sW_\lozenge(M,\omega)$ the subset of those $J \in \sJ(M,\omega)$ for which there is a simple $J$--holomorphic map $u$ of index zero such that an \emph{unbranched} cover of $u$ fails to be rigid.
  $\sW_\lozenge(M,\omega)$ has codimension at least one.
\end{prop}

This is a only warm-up because it does not account for branched covers.
The following discussion puts us in a position to prove \autoref{Prop_UnbranchedCoversFailToBeRigidInCodimensionOne} using \autoref{Thm_EquivariantBrillNoetherLoci_Twists}.

\begin{definition}
  \label{Def_PullbackCauchyRiemannOperator}
  Let $(\Sigma,j)$ and $(\tilde\Sigma,\tilde j)$ be Riemann surfaces,
  let $E$ be a complex vector bundle over $\Sigma$, and
  let $\fd\co \Gamma(E) \to \Omega^{0,1}(\Sigma,E)$ be a real Cauchy--Riemann operator.
  Let $\pi\co (\tilde \Sigma,j) \to (\Sigma,j)$ be a non-constant holomorphic map.
  The \defined{pullback} of $\fd$ by $\pi$ is the real Cauchy--Riemann operator
  \begin{equation*}
    \pi^\star\fd \co \Gamma(\pi^*E) \to \Omega^{0,1}(\tilde \Sigma,\pi^*E)
  \end{equation*}
  characterized by
  \begin{equation*}
    (\pi^\star \fd)(\pi^*s) = \pi^*(\fd s).
    \qedhere
  \end{equation*}  
\end{definition}

The following is proved as \autoref{Prop_PullbackGeneralizedNormalBundle} in \autoref{Sec_NormalCauchyRiemannOperator}.

\begin{prop}
  \label{Prop_PullbackGeneralizedNormalBundle_Pre}
  Let $u\co (\Sigma,j) \to (M,J)$ be a non-constant $J$--holomorphic map.
  If $\pi\co (\tilde \Sigma,\tilde j) \to (\Sigma,j)$ is a non-constant holomorphic map and $\tilde u \coloneq u\circ \pi$,
  then there is an isomorphism $N(u\circ\pi) \iso \pi^*Nu$ with respect to which $\fd_{u\circ\pi,J}^N = \pi^\star\fd_{u,J}^N$.
\end{prop}

In the situation of \autoref{Def_PullbackCauchyRiemannOperator},
if $\pi$ is covering map,
then $\pi^\star\fd$ and $\pi^*\fd$ defined in \autoref{Def_DPullback} are related by the commutative diagram
\begin{equation*}
  \begin{tikzcd}
    \Gamma(\pi^*E) \ar[equal]{d} \ar{r}{\pi^*\fd} & \Gamma(\pi^*T^*\Sigma^{0,1} \otimes_{\C} \pi^*E) \ar{d}{\pi^*} \\
    \Gamma(\pi^*E) \ar{r}{\pi^\star\fd} & \Omega^{0,1}(\tilde \Sigma,\pi^*E)
  \end{tikzcd}
\end{equation*}
with $\pi^*$ being an isomorphism.
(This explains the intentionally confusing choice of notation.)
Therefore and by \autoref{Prop_Pullback=Twist},
if $J \in \sW_\lozenge$,
then there exists a simple $J$--holomorphic map $u \co (\Sigma,j) \to (M,J)$ of index zero and an irreducible Euclidean local system $\uV$ whose monodromy representation factors through a finite quotient of $\pi_1(\Sigma,x_0)$ such that
\begin{equation*}
  \ker \fd_{u,J}^{\uV} \neq 0.
\end{equation*}

\begin{proof}[Proof of \autoref{Prop_UnbranchedCoversFailToBeRigidInCodimensionOne}]
  An open subset $\sU \subset \sM_0(M,\omega)$ is liftable if there is a Riemann surface $(\Sigma,j_0)$ and an $\Aut(\Sigma,j_0)$--invariant slice $\sS$ of Teichmüller space through $j_0$ such that for every $(J,[u,j]) \in \sU$ there is a unique lift $u\co (\Sigma,j) \to (M,J)$ with $j \in \sS$.
  The universal moduli space $\sM_0(M,\omega)$ is covered by countably many such open subsets.

  Let $\sU \subset \sM\sO_{0,s}(M,\omega)$ be as above and such that for every $(J,[u,j;\nu]) \in \sU$ the map $u$ is an embedding.
  Consider the family of normal Cauchy--Riemann operators $\fd_{u,J}^N\co \Gamma(Nu) \to \Omega^{0,1}(Nu)$ parametrized by $\sU$.
  Strictly speaking, this is not a family of linear elliptic differential operators as in \autoref{Def_FamilyOfEllipticOperators}.
  Indeed,
  $Nu$ and $\overline\Hom_\C(T\Sigma,Nu)$ depend on $u$ and $j$ and thus define vector bundles $\bE$ and $\bF$ over $\sU\times \Sigma$.
  However, after shrinking $\sU$, one can construct isomorphisms $\bE \iso \pr_\Sigma^* E$ and $\bF \iso \pr_\Sigma^* F$ with $E \coloneq Nu_0$ and $F \coloneq \Hom_\C(T\Sigma,Nu_0)$ for $(J_0,[u_0,j_0]) \in \sU$.
  Employing these isomorphisms,
  the normal Cauchy--Riemann operators form a family of linear elliptic differential operators 
  \begin{equation*}
    \fd \co \sU \to \sF(W^{1,2}\Gamma(E),L^2\Gamma(F))
  \end{equation*}
  as in \autoref{Def_FamilyOfEllipticOperators}.
  A moment's thought shows that the map $\Lambda_p^\fV$ defined in \autoref{Thm_EquivariantBrillNoetherLoci_Twists},
  $\fV$--equivariant flexiblity defined in \autoref{Def_VEquivariantlyFlexible}, and
  the $\fV$--equivariant Petri condition defined in \autoref{Def_VEquivariantPetriCondition} are independent of the choice of isomorphisms $\bE \iso \pr_\Sigma^* E$ and $\bF \iso \pr_\Sigma^* F$.
  Therefore, the results from \autoref{Sec_EquivariantBrillNoetherLoci_Twists} apply without reservation.
    
  By \autoref{Prop_NormalCauchyRiemannFlexibility},
  \autoref{Prop_FloerNormStrongFlexibility},
  and \autoref{Prop_StronglyFlexible=>VEquivariantlyFlexible},
  $\fd$ is $\fV$--equivariantly flexible in $U$.
  Furthermore,
  by \autoref{Thm_PetrisConditionFailsInInfiniteCodimension} and \autoref{Thm_CauchyRiemannSymbolSatisfiesWendlsCondition},
  the subset of $(J,[u,j]) \in \sU$ for which $\fd_{u,J}^N$ fails to satisfy the $\fV$--equivariant Petri condtion in $U$ has infinite codimension.

  Let $\uV$ be an irreducible Euclidean local system whose monodromy representation factors through a finite quotient of $\pi_1(\Sigma,x_0)$ and consider $\fV$ as in \autoref{Sit_V} consisting only of $\uV$.
  Denote by
  \begin{equation*}
    \sW_{\Lambda;\sU,\uV}
  \end{equation*}
  the subset of those $p \coloneq (J,[u,j]) \in \sU$ for which the map $\Lambda_p^\fV$ defined in \autoref{Thm_EquivariantBrillNoetherLoci_Twists} fails to be surjective.
  Evidently, $\sW_{\Lambda,\sU,\uV}$ is closed; in particular, $\sU\setminus \sW_{\Lambda;\sU,\uV}$ is a Banach manifold.
  By the preceding paragraph, $\sW_{\Lambda,\sU,\uV}$ has infinite codimension.
  Set
  \begin{equation*}
    \sW_{\lozenge;\sU,\uV}
    \coloneq
    \set*{
      (J,[u,j]) \in \sU \setminus \sW_{\Lambda;\sU,\uV}
      :
      \ker \fd_{u,J}^{N,\uV} \neq 0
    }.
  \end{equation*}
  Since
  \begin{equation*}
    \ind \fd_{u,J}^{N,\uV}
    = \rk\uV \cdot \ind \fd_{u,J}^N
    \leq \rk\uV \cdot \ind(u)
    = 0,
  \end{equation*}
  by \autoref{Thm_EquivariantBrillNoetherLoci_Twists},
  $\sW_{\lozenge;\sU,\uV}$ has codimension at least one.  

  By the above discussion,
  $\sW_\lozenge(M,\omega)\setminus \sW_{\incl}(M,\omega)$ is contained in the union of countably many subsets of the form $\Pi(\sW_{\Lambda;\sU,\uV})$ and $\Pi(\sW_{\lozenge;\sU,\uV})$ each of which has codimension at least one because $\Pi$ has index zero.
  Therefore, $\sW_\lozenge(M,\omega)$ has codimension at least one.
\end{proof}

The next two sections develop tools with which the above argument can be carried over to branched covering maps.

%%% Local Variables:
%%% mode: latex
%%% TeX-master: "EquivariantBrillNoetherSuperRigidity"
%%% End:

%!TEX root = EquivariantBrillNoetherSuperRigidity.tex

\section{Branched covering maps as orbifold covering maps}
\label{Sec_BranchedCoveringMapsAsOrbifoldCoveringMaps}

An orbifold Riemann surface can be constructed from a smooth Riemann surface and a collection of points equipped with multiplicities.%
\footnote{%
  For an introduction to complex orbifolds we refer the reader to \cites{Kawasaki1979}[Section 1]{Furuta1992}[Section 8(ii)]{Kronheimer1995}. 
}
The starting point of this construction is the following observation.
Denote by $\bD$ the unit disk in $\C$.
For $k \in \N$ denote by
\begin{equation*}
  \mu_k \coloneq \set{ \zeta \in \C : \zeta^k = 1 }
\end{equation*}
the group of $k^{\text{th}}$ roots of unity.
The map $\pi\co \bD \to \bD$ defined by $\pi(z) \coloneq z^k$ induces a homeomorphism $\bD/\mu_k \iso \bD$.
Denote by $[\bD/\mu_k]$ the orbifold with $\bD$ as the underlying topological space and $\pi$ as chart.
The map $\pi$ also induces an orbifold map $\beta\co [\bD/\mu_k] \to \bD$ which induces the identity map on the underlying topological spaces.
The identity map $\bD \to \bD$ defines an orbifold map $\hat\pi\co \bD \to [\bD/\mu_k]$.
This map is a covering map because $\bD \iso [(\bD\times\mu_k)/\mu_k]$;
cf. \autoref{Footnote_CoveringMap} on page~\pageref{Footnote_CoveringMap}.
By construction, $\pi = \beta\circ \hat\pi$.
This can be globalized as follows.

\begin{definition}
  \label{Def_OrbifoldRiemannSurface}
  Let $(\Sigma,j)$ be a Riemann surface.
  A \defined{multiplicity function} is  a function $\nu \co \Sigma \to \N$ such that the set 
  \begin{equation*}
    Z_\nu \coloneq \set{ x \in \Sigma : \nu(x) > 1 }
  \end{equation*}
  is discrete. 
  Given a multiplicity function $\nu$,
  denote by $(\Sigma_{\nu},j_\nu)$ the orbifold Riemann surface whose underlying topological space is $\Sigma$ and such that for every $x \in \Sigma$ and every holomorphic chart $\phi\co \bD \to \Sigma$ with $\phi(0) = x$ the map $\phi_{\nu(x)} \co \bD \to \Sigma$ defined by
  \begin{equation*}
    \phi_{\nu(x)}(z) \coloneq \phi(z^{\nu(x)})
  \end{equation*}
  is a holomorphic orbifold chart.
  Denote by $\beta_\nu\co (\Sigma_\nu,j_\nu) \to (\Sigma,j)$ the holomorphic orbifold map given by $z \mapsto z^{\nu(x)}$ with respect to these charts.  
\end{definition}

\begin{remark}
  Every effective orbifold Riemann surface is isomorphic to one constructed as in \autoref{Def_OrbifoldRiemannSurface}.
\end{remark}

This construction allows us to canonically associate an orbifold cover with every branched cover of Riemann surfaces.

\begin{prop}
  \label{Prop_BranchedCoversOfRiemannSurfacesAsOrbifoldCovers}
  Let $(\Sigma,j)$ and $(\tilde\Sigma,\tilde j)$ be Riemann surfaces and let $\pi \co (\tilde \Sigma,\tilde j) \to (\Sigma,j)$ be a non-constant holomorphic map.  
  For every $\tilde x \in \tilde \Sigma$ denote by $r(\tilde x) \in \N$ the ramification index of $\pi$ at $\tilde x$.
  Define $\nu\co \Sigma \to \N$ and $\tilde \nu\co \tilde \Sigma \to \N$ by
  \begin{equation*}
    \nu(x) \coloneq \lcm \set{ r(\tilde x) : \tilde x \in \pi^{-1}(x) }
    \qandq
    \tilde \nu(\tilde x) \coloneq \nu(\pi(\tilde x))/r(\tilde x).
  \end{equation*}
  Let $(\tilde\Sigma_{\tilde\nu},\tilde j)$ and $(\Sigma_\nu,j_\nu)$ be the corresponding orbifold Riemann surfaces constructed in \autoref{Def_OrbifoldRiemannSurface}. 
  There is a unique holomorphic covering map $\hat\pi \co (\tilde \Sigma_{\tilde \nu},\tilde j_{\tilde \nu}) \to (\Sigma_\nu,j_\nu)$ such that the diagram
  \begin{equation}
    \label{Eq_BranchedCoversOfRiemannSurfacesAsOrbifoldCovers}
    \begin{tikzcd}
      (\tilde \Sigma_{\tilde \nu},\tilde j_{\tilde \nu}) \ar[r,"\hat\pi"] \ar[d,"\beta_{\tilde \nu}"]  & (\Sigma_\nu,j_\nu) \ar[d,"\beta_\nu"] \\
      (\tilde \Sigma,\tilde j) \ar[r,"\pi"] & (\Sigma,j)
    \end{tikzcd}
  \end{equation}
  commutes.
\end{prop}

\begin{proof}
  For every $x \in \Sigma$ there is a holomorphic chart $\phi \co \bD \to \Sigma$ with $\phi(0) = x$ and for every $\tilde x \in \pi^{-1}(x)$ there is a holomorphic chart $\tilde \phi\co \bD \to \tilde \Sigma$ such that $\tilde\phi(0) = \tilde x$ and $\pi\circ \phi(z) = \phi(z^{r(\tilde x)})$.
  There is a unique orbifold map $\hat\pi\co \tilde \Sigma_{\tilde \nu} \to \Sigma_\nu$ which is given by the identity map with respect to the charts $\phi_{\tilde \nu(\tilde x)}$ and $\phi_{\nu(x)}$.
  Evidently,
  this map is holomorphic.
  It is a covering map because
  \begin{equation*}
     [\bD/\mu_{\tilde\nu(\tilde x)}]
    \iso
    \sqparen{
      \paren{
        \bD \times_{\mu_{\tilde \nu(\tilde x)}} \mu_{\nu(x)}
      }
      /\mu_{\nu(x)}
    }
  \end{equation*}
  and the canonical map
  \begin{equation*}
    \bD \times_{\mu_{\tilde \nu(\tilde x)}} \mu_{\nu(x)} 
    \to
    \bD
  \end{equation*}
  is a $\mu_{\nu(x)}$--equivariant covering map.  
\end{proof}

\begin{remark}
  Every covering map of effective orbifold Riemann surfaces arises from a branched cover of the underlying smooth Riemann surfaces by the above construction.
\end{remark}

%%% Local Variables:
%%% mode: latex
%%% TeX-master: "EquivariantBrillNoetherSuperRigidity"
%%% End:

% !TEX root = EquivariantBrillNoetherSuperRigidity.tex

\section{A criterion for the failure of super-rigidity}

The orbifoldization process from \autoref{Def_OrbifoldRiemannSurface} does not affect the kernel and cokernel of real Cauchy--Riemann operators.

\begin{definition}
  \label{Def_AssociatedOrbifoldCROperator}
  Let $(\Sigma,j)$ be a Riemann surface with a multiplicity function $\nu$.
  Given a complex vector bundle $E$ over $\Sigma$ and a real Cauchy--Riemann operator $\fd \co \Gamma(E) \to \Omega^{0,1}(\Sigma,E)$,
  set
  \begin{equation*}
    E_\nu \coloneq \beta_\nu^*E
    \qandq
    \fd_\nu \coloneq \beta_\nu^\star\fd
  \end{equation*}
  with $\beta_\nu^\star\fd \co \Gamma(E_\nu) \to \Omega^{0,1}(\Sigma_\nu,E_\nu)$ as in \autoref{Def_PullbackCauchyRiemannOperator}. 
\end{definition}

\begin{prop}
  \label{Prop_KernelOfOrbifoldCROperator}
  If $(\Sigma,j)$ is a closed Riemann surface with a multiplicity function $\nu$,
  $E$ is a complex vector bundle over $\Sigma$, and
  $\fd \co \Gamma(E) \to \Omega^{0,1}(\Sigma,E)$ is a real Cauchy--Riemann operator,
  then
  \begin{equation*}
    \ker \fd_\nu \iso \ker \fd \qandq
    \coker \fd_\nu \iso \coker \fd.
  \end{equation*}
\end{prop}

\begin{proof}
  The pullback map $\beta_\nu^* \co \Gamma(E) \to \Gamma(E_\nu)$ induces an injection $\ker \fd \incl \ker \fd_\nu$.
  In fact, this map is an isomorphism.
  To see this, the following local consideration suffices.
  Let $x \in \Sigma$ and set $k \coloneq \nu(x)$. 
  Define $\beta\co \bD \to \bD$ by $\beta(z) \coloneq z^k$.
  Choose a holomorphic chart $\phi\co \bD \to \Sigma$ with $\phi(0) = x$ and a trivialization of $E$ over $\phi(\bD)$.
  With respect to these $\fd$ and $\fd_\nu$ can be written as
  \begin{equation*}
    \fd = \delbar + \fn \qandq \fd_\nu = \delbar + \beta^*\fn
  \end{equation*}
  for some $\fn \in \Omega^{0,1}(\bD,\End_\R(\C^r))$.
  If $\tilde s \in C^\infty(\bD,\C^r)$ is $\mu_k$--invariant,
  then there is a bounded map $s \in C^\infty(\bD\setminus\set{0},\C^r)$ such that $\tilde s = s\circ \beta$.
  If $(\delbar + \beta^*\fn)\tilde s = 0$,
  then $(\delbar + \fn)s = 0$;
  hence, $s$ extends to $\bD$ by elliptic regularity.

  Since $\coker \fd \iso (\ker \fd^\dagger)^*$ and similarly for $\fd_\nu$,
  it suffices to produce an isomorphism $\ker \fd^\dagger \iso \ker \fd_\nu^\dagger$.
  The formal adjoints $\fd^\dagger \co \Omega^{1,0}(\Sigma,E^*) \to \Omega^{1,1}(\Sigma,E^*)$ and $\fd_\nu^\dagger \co \Omega^{1,0}(\Sigma_\nu,E_\nu^*) \to \Omega^{1,1}(\Sigma_\nu,E_\nu^*)$ are real Cauchy--Riemann operator acting on $(1,0)$--forms and locally of the form $\fd^\dagger = \delbar + \fn$ and $\fd_\nu^\dagger = \delbar + \beta^*\fn$.
  The pullback map $\beta_\nu^* \co \Omega^{1,0}(\Sigma,E) \to \Omega^{1,0}(\Sigma_\nu,E_\nu)$ induces an injection $\ker \fd^\dagger \incl \ker \fd_\nu^\dagger$.  
  This map is an isomorphism by the following local consideration.
  If $\tilde s \in C^\infty(\bD,\C^r)$ is such that $\tilde s \,\rd z$ is $\mu_k$--invariant,
  then there is a map $s \in C^\infty(\bD\setminus\set{0},\C^r)$ such that $\tilde s = kz^{k-1}s\circ \beta$.
  If $(\delbar + \beta^*\fn)\tilde s = 0$,
  then $(\delbar + \fn)s = 0$ and a consideration of the Taylor expansion of $\tilde s$ shows that $s$ is bounded.
  Therefore, $s$ extends to $\bD$ and $\tilde s \ \rd z = \beta^*(s \,\rd z)$.
\end{proof}

This together with the discussion in \autoref{Sec_BranchedCoveringMapsAsOrbifoldCoveringMaps} leads to the following criterion for the failure of super-rigidity.

\begin{definition}
  \label{Def_MonodromyAroundX}
  Let $(\Sigma,j)$ be a Riemann surface with a multiplicity function $\nu$ and
  let $x_0 \in \Sigma\setminus Z_\nu$.
  For every $x \in Z_\nu$ there is a conjugacy class of a subgroup $\mu_{\nu(x)} \leq \pi_1(\Sigma_\nu,x_0)$,
  generated by the homotopy class of  a loop in $\Sigma\setminus Z_\nu$ based at $x_0$ which is contractible in $(\Sigma\setminus Z_\nu) \cup \set{x}$ and goes around $x$ once.
  If $\uV$ is a Euclidean local system on $\Sigma_\nu$,
  then its \defined{monodromy around $x$} is the representation $\mu_{\nu(x)} \to \O(V)$ induced by the monodromy representation.
\end{definition}

\begin{prop}
  \label{Prop_FailureOfSuperRigidity}
  Let $u \co (\Sigma,j) \to (M,J)$ be a simple $J$--holomorphic map. 
  If $u$ is not super-rigid,
  then there are a multiplicity function $\nu \co \Sigma \to \N$ and
  an irreducible Euclidean local system $\uV$ on $\Sigma_\nu$ such that:
  \begin{enumerate}
  \item
    the monodromy representation of $\uV$ factors through a finite quotient of $\pi_1(\Sigma_\nu,x_0)$,
  \item
    $\uV$ has non-trivial monodromy around every point of $Z_\nu$, and
  \item
    the twist
    \begin{equation*}
      \fd_\nu^\uV \co \Gamma((Nu)_\nu\otimes\uV) \to \Omega^{0,1}(\Sigma_\nu, (Nu)_\nu \otimes\uV)
      \qwithq
      \fd \coloneq \fd_{u,J}^N
    \end{equation*}
  has non-trivial kernel.
  \end{enumerate}
  
\end{prop}

\begin{proof}
  Let $(\tilde\Sigma,\tilde j)$ be a closed Riemann surface and $\pi \co (\tilde\Sigma,\tilde j) \to (\Sigma,j)$ a non-constant holomorphic map such that $\tilde u \coloneq u\circ\pi$ is not rigid, that is: $\ker \fd_{\tilde u,J}^N$ is non-trivial.

  Let $\hat\pi \co (\tilde \Sigma_{\tilde \nu},\tilde j_{\tilde \nu}) \to (\Sigma_\nu,j_\nu)$ be the corresponding holomorphic covering map between orbifold Riemann surfaces constructed in \autoref{Prop_BranchedCoversOfRiemannSurfacesAsOrbifoldCovers}.
  Set $\fd \coloneq \fd_{u,J}^N$ and
  $\tilde\fd \coloneq \fd_{u\circ\pi,J}^N$.
  By
  \autoref{Prop_KernelOfOrbifoldCROperator},
  \autoref{Prop_PullbackGeneralizedNormalBundle_Pre}, and
  \autoref{Prop_Pullback=Twist},  
  \begin{equation*}
    \ker \tilde\fd
    \iso
    \ker \tilde\fd_{\tilde \nu}
    \iso
    \ker \hat\pi^*\fd_\nu
    \iso
    \ker \fd_\nu^{\hat\pi_*\ubR}.
  \end{equation*}
  Therefore, $\ker \fd_\nu^{\hat\pi_*\ubR}$ is non-trivial.
  
  Since $\hat\pi_*\ubR$ decomposes into irreducible local systems,
  there is an irreducible local system $\uV$ such that $\ker \fd_\nu^\uV$ is non-trivial.
  Define the multiplicity function $\nu'\co \Sigma \to \N$ by
  \begin{equation*}
    \nu'(x) \coloneq 
    \begin{cases}
      \nu(x) & \text{if $\uV$ has non-trivial monodromy around $x$} \\
      1 & \text{otherwise}.
    \end{cases}
  \end{equation*}
  $\uV$ descends to an irreducible local system $\uV'$ on $\Sigma_{\nu'}$ with non-trivial monodromy around every $x \in Z_{\nu'}$.
  By \autoref{Prop_KernelOfOrbifoldCROperator},
  $\ker\fd_{\nu'}^{\uV'} \iso \ker\fd_\nu^\uV$.
\end{proof}

The following index formula is the final preparation required for the proof of \autoref{Thm_SuperRigidity}.
Its proof is presented in \autoref{Sec_OrbifoldRiemannRoch}.

\begin{prop}
  \label{Prop_IndexOfNormalCROperatorOfBranchedCover}
  Let $(\Sigma,j)$ be a closed Riemann surface with a multiplicity function $\nu$,
  let $E$ be a complex vector bundle over $\Sigma$, and
  let $\fd\co \Gamma(E) \to \Omega^{0,1}(\Sigma,E)$ a real Cauchy--Riemann operator on $E$.
  If $\uV$ on $\Sigma_\nu$ is a Euclidean local system,
  then
  \begin{equation*}
    \ind \fd_\nu^\uV = \dim V \ind \fd - \rk_\C E \sum_{x \in Z_\nu}\dim(V / V^{\rho_x}).
  \end{equation*}
  Here $\rho_x$ denotes the monodromy of $\uV$ around $x$.
\end{prop}

%%% Local Variables:
%%% mode: latex
%%% TeX-master: "EquivariantBrillNoetherSuperRigidity"
%%% End:

% !TEX root = EquivariantBrillNoetherSuperRigidity.tex

\section{The loci of failure of super-rigidity}
\label{Sec_CodimensionEstimateForBranchedCovers}

To prove \autoref{Thm_SuperRigidity} it remains to prove that $\sW_\superrigid(M,\omega)$ has codimension at least one;
that is: $\sW_\superrigid(M,\omega)$ contained in the image of Fredholm map of index at most $-1$.
To construct this map,
we consider the following decorated variant of the universal moduli space $\sM_k(M,\omega)$.

\begin{definition}
  \label{Def_UniversalSpaceOfOrbifoldHolomorphicMaps}
  Let $k \in \Z$ and $s \in \N_0$. 
  Denote by $\sM\sO_{k,s}(M,\omega)$ be space of pairs $(J,[u,j;\nu])$ consisting of an almost complex structure $J \in \sJ(M,\omega)$,
  and an equivalence class $[u,j;\nu]$ of
  \begin{enumerate}
  \item
    a simple $J$--holomorphic map $u\co (\Sigma,j) \to (M,J)$ of index $k$, and \
  \item
    a multiplicity function $\nu\co \Sigma \to \N$ with $\#Z_\nu = s$
  \end{enumerate}
  up to reparametrization by $\Diff(\Sigma)$.
\end{definition}

The proof of \autoref{Thm_UniversalModuliSpace} shows that $\sM\sO_{k,s}$ is a Banach manifold and the map $\Pi_{k,s}\co \sM\sO_{k,s} \to \sJ(M,\omega)$ is a Fredholm map of index $k+2s$.
By \autoref{Prop_FailureOfSuperRigidity},
the failure of super-rigidity is detected by the following subsets of $\sM\sO_{0,s}$.

\begin{definition}
  Let $s \in \N_0$.
  Denote by $\sW_{\superrigid,s}(M,\omega)$ the subset of those $(J,[u,j;\nu]) \in \sM\sO_{0,s}(M,\omega)$ which are embedded and for which there exists an irreducible Euclidean local system $\uV$ on $\Sigma_\nu$ such that:
  \begin{enumerate}
  \item
    the monodromy representation of $\uV$ factors through a finite quotient of $\pi_1(\Sigma,x_0)$,
  \item
    $\uV$ has non-trivial monodromy around every $x \in Z_\nu$,
    and
  \item
    $\ker \fd_{u,J;\nu}^{N,\uV} \neq 0$.
  \end{enumerate}  
  Denote by $\sW_{\superrigid,s}^{\textnormal{top}}(M,\omega)$ the subset for those $(J,[u,j;\nu]) \in \sM\sO_{0,s}(M,\omega)$ for which $\uV$ can be chosen with the following additional properties:
  \begin{enumerate}[resume]    
  \item
    $\dim \ker \fd_{u,J;\nu}^{N,\uV} = 1$,
  \item
    \label{Def_TopStratum_OtherLocalSystems}
    if $\uV'$ is any other irreducible Euclidean local system on $\Sigma_\nu$ whose monodromy representation factors through a finite quotient of $\pi_1(\Sigma,x_0)$, then
    \begin{equation*}
      \dim \ker \fd_{u,J;\nu}^{N,\uV'} = 0.
    \end{equation*}
  \item
    if $n = 3$, then $\dim (V/V^{\rho_x}) = 1$ for every $x\in Z_\nu$;
    otherwise, $s = 0$.
    \qedhere
  \end{enumerate}
\end{definition}

\begin{prop}
  \label{Prop_LociOfFailureOfSuperrigidity}
  For every $s \in \N_0$ there is a closed subset $\sW_{\superrigid,s}^\Lambda(M,\omega) \subset
  \sM\sO_{0,s}(M,\omega)$ of infinite codimension such that the following hold:
  \begin{enumerate}
  \item
    \label{Prop_LociOfFailureOfSuperrigidity_TopStratumSubmanifold}
    $\sW_{\superrigid,s}^{\textnormal{top}}(M,\omega)\setminus \sW_{\superrigid,s}^\Lambda(M,\omega)$ is contained in a submanifold of codimension $2s + 1$.
  \item
    \label{Prop_LociOfFailureOfSuperrigidity_CodimensionEstimate}
    $\sW_{\superrigid,s}(M,\omega)\setminus(\sW_{\superrigid,s}^{\textnormal{top}}(M,\omega)\cup\sW_{\superrigid,s}^\Lambda(M,\omega))$ has codimension at least $2s + 2$.   
  \end{enumerate}
\end{prop}

\begin{proof}
  An open subset $\sU \subset \sM\sO_{0,s}(M,\omega)$ is liftable if there is a Riemann surface $(\Sigma,j_0)$ and an $\Aut(\Sigma,j_0)$--invariant slice $\sS$ of Teichmüller space through $j_0$ such that for every $(J,[u,j;\nu]) \in \sU$ there is a unique lift $u\co (\Sigma,j) \to (M,J)$ and $\nu\co \Sigma \to \N$ with $j \in \sS$.
  The universal moduli space $\sM\sO_{0,s}(M,\omega)$ is covered by countably many such open subsets.

  Let $\sU \subset \sM\sO_{0,s}(M,\omega)$ be as above and such that for every $(J,[u,j;\nu]) \in \sU$ the map $u$ is an embedding.
  As in the proof of \autoref{Prop_UnbranchedCoversFailToBeRigidInCodimensionOne},
  the operators $\fd_{u,J;\nu}^N\co \Gamma((Nu)_\nu) \to \Omega^{0,1}(\Sigma_\nu,(Nu)_\nu)$ for $(J,[u,j,\nu]) \in \sU$ can be regarded as a family of linear elliptic operators
  \begin{equation*}
    \fd \co \sU \to \sF(W^{1,2}\Gamma(E),L^2\Gamma(F))
  \end{equation*}
  as in \autoref{Def_FamilyOfEllipticOperators}.

  Let $\uV_1,\uV_2$ be a pair of non-isomorphic irreducible Euclidean local system whose monodromy representation factors through a finite quotient of $\pi_1(\Sigma,x_0)$ and consider $\fV$ as in \autoref{Sit_V} consisting of $\uV_1$ and $\uV_2$.
  Denote by $\sW_{\superrigid,s;\sU,\fV}^\Lambda$ the subset of those $p \coloneq (J,[u,j]) \in \sU$ for which the map $\Lambda_p^\fV$ defined in \autoref{Thm_EquivariantBrillNoetherLoci_Twists} fails to be surjective.
  The argument from the proof of \autoref{Prop_UnbranchedCoversFailToBeRigidInCodimensionOne} shows that $\sW_{\superrigid,s;\sU,\fV}^\Lambda$ is a closed subset of infinite codimension.
  The union of these subsets is $\sW_{\superrigid,s}^\Lambda$.
  For $d \in \N_0^2$ set
  \begin{equation*}
    \sW_{\superrigid,s;\sU,\fV}^d
    \coloneq
    \set*{
      (J,[u,j]) \in \sU \setminus \sW_{\superrigid,s;\sU,\fV}^\Lambda
      :
      \dim_{\bK_\alpha} \ker \fd_{u,J;\nu}^{N,\uV_\alpha} = d_\alpha \textnormal{ for } \alpha=1,2
    }.
  \end{equation*}
  By \autoref{Thm_EquivariantBrillNoetherLoci_Twists},  
  $\sW_{\superrigid,s;\sU,\fV}^d$ is a submanifold of codimension
  \begin{equation*}
    \codim \sW_{\superrigid,s;\sU,\fV}^d
    =
    \sum_{\alpha=1}^2 k_\alpha d_\alpha(d_\alpha-i_\alpha)
    \qwithq
    i_\alpha
    \coloneq
    \ind_{\bK_\alpha} \fd_{u,J;\nu}^{N,\uV_\alpha}.
  \end{equation*}  
  By \autoref{Prop_IndexOfNormalCROperatorOfBranchedCover},
  \begin{equation*}
    i_\alpha
    \coloneq
    \ind_{\bK_\alpha} \fd_{u,J;\nu}^{N,\uV_\alpha}
    \leq
    - (n-1) \sum_{x \in Z_\nu}\dim_{\bK_\alpha}(V_\alpha/V_\alpha^{\rho_x}).
  \end{equation*}  
  If $\uV_\alpha$ has non-trivial monodromy around every $x \in Z_\nu$,
  then $i_\alpha \leq -(n-1)r$.
  Therefore, if $d_\alpha \geq 1$,
  then
  \begin{equation*}
    \codim \sW_{\superrigid,s;\sU,\fV}^d
    \geq
    (n-1)s+1
    \geq
    2s+1.
  \end{equation*}
  $\sW_{\superrigid,s}\setminus\sW_{\superrigid,s}^\Lambda$ is the union of countably many subsets of the form $\sW_{\superrigid,s;\sU,\fV}^d$ with at least one $\alpha = 1,2$ as above.
  Therefore, it has codimension at least $2s+1$.
    
  Analysing the chain of inequalities shows that $\codim \sW_{\superrigid,s;\sU,\fV}^d = 2s+1$ if and only if there is an $\alpha = 1,2$ such that:
  \begin{enumerate}
  \item
    $d_\alpha = 1$, $\bK_\alpha = \R$ and $d_\beta = 0$ for $\beta \neq \alpha$,
    and
  \item
    if $n=3$, then $\dim(V_\alpha/V_\alpha^{\rho_x}) = 1$ for every $x \in Z_\nu$;
    otherwise, $Z_\nu = \emptyset$.
  \end{enumerate}
  The union of these subsets is $\sW_{\superrigid,s}^{\textnormal{top}}(M,\omega)\setminus\sW_{\superrigid,s}^\Lambda(M,\omega)$;
  hence: \autoref{Prop_LociOfFailureOfSuperrigidity_TopStratumSubmanifold} holds.
  Furthermore,
  only $\sW_{\superrigid,s;\sU,\fV}^d$ of codimension at least $2s+2$ are required to cover 
  $\sW_{\superrigid,s}(M,\omega) \setminus (\sW_{\superrigid,s}^{\textnormal{top}}(M,\omega) \cup\sW_{\superrigid,s}^\Lambda(M,\omega))$.
  This implies \autoref{Prop_LociOfFailureOfSuperrigidity_CodimensionEstimate}.
\end{proof}

\begin{proof}[Proof of \autoref{Thm_SuperRigidity}]
  The subset of $J \in \sJ(M,\omega)$ which fail to be superrigid is
  \begin{equation*}
    \sW_{\geq 0}(M,\omega)
    \cup
    \sW_{\into}(M,\omega)
    \cup
    \sW_{\superrigid}(M,\omega).
  \end{equation*}
  The first two subsets have already been shown to have codimension at least two.
  The third subset is contained in
  \begin{equation*}
    \bigcup_{s \in \N_0} \Pi_{0,s}(\sW_{\superrigid,s}(M,\omega)).
  \end{equation*}
  The theorem follows since $\Pi_{s,0}$ has index $2s$ and $\sW_{\superrigid,s}(M,\omega)$ has codimension at least $2s+1$.  
\end{proof}

%%% Local Variables:
%%% mode: latex
%%% TeX-master: "EquivariantBrillNoetherSuperRigidity"
%%% End:

% !TEX root = EquivariantBrillNoetherSuperRigidity.tex

\section{Super-rigidity along paths of almost complex structures}
\label{Sec_SuperRigidityOneParameterFamilies}

The following describes in detail how super-rigidity may fail along a generic path of almost complex structures.

\begin{definition}
  \label{Def_SuperrigidPaths}
  Set
  \begin{equation*}
    \bsJ(M,\omega) \coloneq C^1([0,1], \sJ(M,\omega)).
  \end{equation*}
  For $\bJ \in \bsJ(M,\omega)$ set $J_t \coloneq \bJ(t)$.
  Denote by $\bsJ\!_{\superrigid}(M,\omega)$ the subset of all $\bJ \in \bsJ(M,\omega)$ for which the following conditions hold:
  \begin{enumerate}
  \item
    \label{Def_SuperrigidPaths_Cobordism}
    The moduli space
    \begin{equation*}
      \bsM_0(M,\bJ)
      \coloneq [0,1] \times_{\sJ(M,\omega)} \sM_0(M,\omega),
    \end{equation*}
    is a $1$--dimensional manifold with boundary. 
  \item
    \label{Def_SuperrigidPaths_Embedded}
    For every $t \in [0,1]$ the following hold:
    \begin{enumerate}
    \item
      Every simple $J_t$--holomorphic map has non-negative index.
    \item
      Every simple $J_t$--holomorphic map of index zero is an embedding, and
      every two simple $J_t$--holomorphic maps of index zero either have disjoint images or are related by a reparametrization.
    \end{enumerate}
  \item 
    \label{Def_SuperrigidPaths_SuperRigidity}
    The set $I_\superrigid$ of those $t \in [0,1]$ for which $J_t$ fails to be super-rigid is countable;
    moreover, $t \in I_\superrigid$ if and only if
    \begin{equation*}
      J_t \in \bigcup_{s \in \N_0} \Pi_{0,s}(\sW_{\superrigid,s}\setminus \sW_{\superrigid,s}^\Lambda).
      \qedhere
    \end{equation*}
  \end{enumerate} 
\end{definition}

\begin{theorem}[{cf.~\cite[Section 2.4]{Wendl2016}}]
  \label{Thm_SuperRigidityInOneParameterFamilies}
  $\bsJ\!_{\superrigid}(M,\omega)$ is a comeager subset of $\bsJ(M,\omega)$.
\end{theorem}

\begin{proof}[Proof of \autoref{Thm_SuperRigidityInOneParameterFamilies}]
  First, we show that the set of $\bJ$ for which condition \autoref{Def_SuperrigidPaths_Cobordism} from  \autoref{Def_SuperrigidPaths} holds is comeager. 
  The evaluation map $\ev \co [0,1] \times \bsJ(M,\omega) \to \sJ(M,\omega)$ defined by
  \begin{gather*}
    \ev (t,\bJ) \coloneq \bJ(t)
  \end{gather*}  
  is a submersion.
  Therefore, the fibered product $\ev^*\sM_0(M,\omega)$ is a Banach manifold with boundary. 
  The projection 
  \begin{equation*}
    \pr \co \ev^*\sM_0(M,\omega) \to \bsJ(M,\omega)
  \end{equation*}
  is a Fredholm map of index $1$ and $\bsM_0(M,\bJ) = \pr^{-1}(\bJ)$.
  For every regular value $\bJ$ of $\pr$ the condition \autoref{Def_SuperrigidPaths_Cobordism} is satisfied.
  The set of regular values of $\pr$ is comeager by the Sard--Smale theorem.

  As discussed in the last part of \autoref{Sec_SuperRigidity},
  $\sW_{\geq 0}(M,\omega)$ has codimension at least $2$ and  $\sW_{\incl}(M,\omega)$ has codimension at least $2(n-2) \geq 2$.
  Therefore and by \autoref{Prop_Codimension_Map},
  the subset of those $\bJ$ for which \autoref{Def_SuperrigidPaths_Embedded} holds is comeager.

  The subset of those $\bJ$ for which \autoref{Def_SuperrigidPaths_SuperRigidity} hold is comeager by \autoref{Prop_Codimension_Map}, \autoref{Prop_LociOfFailureOfSuperrigidity}, and the fact that a codimension one subset of $[0,1]$ is countable.
\end{proof}

\begin{remark}
  If $J_0, J_1 \in \sJ_\superrigid(M,\omega)$,
  then the proof of \autoref{Thm_SuperRigidityInOneParameterFamilies} can be adapted in a straight-forward way to the space of paths in $\sJ(M,\omega)$ connecting $J_0$ and $J_1$.
  It can also be adapted to the space of paths of almost complex structures $C_\epsilon^\infty$--close to a fixed path $(J_t)_{t\in[0,1]}$.
\end{remark}

%%% Local Variables:
%%% mode: latex
%%% TeX-master: "EquivariantBrillNoetherSuperRigidity"
%%% End:

\begin{appendices}
  \setcounter{section}{0}
  \counterwithin{section}{part}
  \renewcommand{\thesection}{\arabic{part}.\Alph{section}}
  %!TEX root = EquivariantBrillNoetherSuperRigidity.tex

\section{The normal Cauchy--Riemann operator}
\label{Sec_NormalCauchyRiemannOperator}

The normal Cauchy--Riemann operator for embedded $J$--holomorphic maps can be traced back to the work of \citet[2.1.B]{Gromov1985}.
It was observed by \citet[Section 1.3]{Ivashkovich1999} that the normal Cauchy--Riemann operator can be defined even for non-embedded $J$--holomorphic maps,
and that it plays an important role in understanding the deformation theory of $J$--holomorphic curves; see also \cite[Section 3]{Wendl2010}.
In this section we will briefly explain the construction of $Tu$ and $Nu$, and discuss the proof of \autoref{Prop_DDelbarVSDN}.

\begin{definition}
  Let $u\co (\Sigma,j) \to (M,J)$ be a non-constant $J$--holomorphic map.
  Denote by $\fd_{u,J}$ the linearization of the $J$--holomorphic map equation introduced in \autoref{Eq_DU}. 
  Denote by $\delbar_{u,J}$ the complex linear part of $\fd_{u,J}$.
  This is a complex Cauchy--Riemann operator and gives $u^*TM$ the structure of a holomorphic vector bundle
  \begin{equation*}
    \sE \coloneq (u^*TM,\delbar_{u,J}).
  \end{equation*}
  Denote by $\sT\Sigma$ the tangent bundle of $\Sigma$ equipped with its natural holomorphic structure.
  The derivative of $u$ induces a holomorphic map $\rd u \co \sT\Sigma \to \sE$.
  The quotient of this map, thought of as a morphism of sheaves, 
  \begin{equation*}
    \sQ \coloneq \sE/\sT\Sigma
  \end{equation*}
  is a coherent sheaf on $\Sigma$.
  Denote by $\Tor(\sQ)$ the torsion subsheaf of $\sQ$.
  The quotient
  \begin{equation*}
    \sN\!u \coloneq \sQ/\Tor(\sQ)
  \end{equation*}
  is torsion-free; hence: locally free.
  The corresponding holomorphic vector bundle $(Nu,\delbar_{Nu})$ is called the \defined{generalized normal bundle of $u$}.
  The kernel
    \begin{equation*}
    \sT\!u \coloneq \ker(\sE \to \sN\!u).
  \end{equation*}
  also is locally free.
  The corresponding holomorphic vector bundle $(Tu,\delbar_{Tu})$ is called the \defined{generalized tangent bundle of $u$}.
\end{definition}

\begin{prop}
  \label{Prop_GeneralizedTangentBundle}
  
  Denote by $D$ the divisor of critical points of $\rd u$ counted with multiplicty.
  There is a short exact sequence
  \begin{equation*}
    0 \to \sT\Sigma \to \sT\!u \to \sO_D \to 0;
  \end{equation*}
  in particular:
  \begin{equation*}
    \sT\!u \iso \sT\Sigma(D).
  \end{equation*}
\end{prop}

\begin{proof}
  The following commutative diagram summarizes the construction of $\sT\!u$ and $\sN\!u$:
  \begin{equation*}
    \begin{tikzcd}
      &  & \Tor(\sQ) \ar[d,hook] \\
      \sT\Sigma \ar[r,hook] \ar[d,hook] & \sE \ar[d,equals] \ar[r,two heads] & \sQ \ar[d,two heads] \\
      \sT\!u \ar[r,hook] \ar[d,two heads] & \sE \ar[r,two heads] & \sN\!u \\
      \sT\!u/\sT\Sigma.
    \end{tikzcd}
  \end{equation*}
  Since the columns and rows are exact sequences,
  it follows from the Snake Lemma that
  \begin{equation*}
    \Tor(\sQ) \iso \sT\!u/\sT\Sigma.
  \end{equation*}
  Thus it remains to prove that $\Tor(\sQ) \iso \sO_D$.
  This is a consequence of the fact that near a critical point $z_0$ of order $k$ we can write $\rd u$ as $(z-z_0)^kf(z)$ with $f(z_0) \neq 0$.
\end{proof}

\begin{prop}
  \label{Prop_PullbackGeneralizedNormalBundle}
  Let $u\co (\Sigma,j) \to (M,J)$ be a non-constant $J$--holomorphic map.
  If $\pi\co (\tilde \Sigma,\tilde j) \to (\Sigma,j)$ is a non-constant holomorphic map and $\tilde u \coloneq u\circ \pi$,
  then
  \begin{equation*}
    \sT\!\tilde u \iso \pi^*\sT\!u \qandq
    \sN\!\tilde u \iso \pi^*\sN\!u.
  \end{equation*}
  The corresponding isomorphism of vector bundles $N\tilde u \iso \pi^*Nu$ induces a commutative diagram
  \begin{equation*}
    \begin{tikzcd}
      \Gamma(\tilde\Sigma,N\tilde u) \ar{d}{\iso} \ar{r}{\fd_{\tilde u,J}^N} & \Omega^{0,1}(\tilde\Sigma,N\tilde u) \ar{d}{\iso} \\ 
      \Gamma(\tilde\Sigma,\pi^*Nu) \ar{r}{\pi^\star\fd_{u,J}^N} & \Omega^{0,1}(\tilde\Sigma,\pi^*Nu).
    \end{tikzcd}
  \end{equation*}
\end{prop}

\begin{proof}
  $\sT\!u \subset \sE$ is the minimal locally free subsheaf which contains the image of $\sT\Sigma \incl \sE$.
  Set $\tilde \sE \coloneq (\tilde u^*TM,\delbar_{\tilde u,j})$.
  There is a canonical isomorphism
  $\tilde \sE \iso \pi^*\sE$.
  Through this identification,
  $\pi^*\sT\!u$ can be regarded as a subsheaf of $\tilde\sE$.
  It is locally free and contains the image of $\sT\tilde \Sigma \incl \tilde\sE$.
  Therefore, $\sT\!\tilde u \iso \pi^*\sT\!u$.
  This also implies that $\sN\!\tilde u \iso \tilde \sE/\sT\!\tilde u \iso \pi^*(\sE/\sT\!u) \iso \pi^*\sN\!u$.

  That the isomorphism $N\tilde u \iso \pi^*Nu$ identifies $\pi^*\fd_{u,J}^N$ and $\fd_{\tilde u, J}^N$ is evident away from the set of critical points of $\pi$.
  Since the latter is nowhere dense, the operators are identified everywhere.
\end{proof}

\begin{proof}[Proof of \autoref{Prop_DDelbarVSDN}]
  Let $\sS$ be an $\Aut(\Sigma,j)$--invariant local slice of the Teichmüller space $\sT(\Sigma)$ through $j$.
  Recall that $\rd_{u,j}\delbar_J \co \Gamma(u^*TM)\oplus T_j\sS \to \Omega^{0,1}(\Sigma,u^*TM)$ is the linearization of $\delbar_J$,
  defined in \eqref{Eq_JHolomorphic},
  restricted to $C^\infty(\Sigma,M)\times \sS$.
  Denote by $Tu$ the complex vector bundle underlying $\sT\!u$ and by $Nu$ the complex vector bundle underlying $\sN\!u$.
  As was mentioned before \autoref{Def_NormalCauchyRiemannOperator},
  $Tu \subset u^*TM$ is the unique complex subbundle of rank one containing $\rd u(T\Sigma)$.
  Using a Hermitian metric on $u^*TM$ we obtain an isomorphism
  \begin{equation*}
    u^*TM \iso Tu \oplus Nu.
  \end{equation*}
  With respect to this splitting $\fd_{J,u}$,
  the restriction of $\rd_{u,j}\delbar_J$ to $\Gamma(u^*TM)$,
  can be written as
  \begin{equation*}
    \fd_{J,u}
    =
    \begin{pmatrix}
      \fd_{u,J}^T & * \\
      \dagger & \fd_{u,J}^N
    \end{pmatrix}
  \end{equation*}
  with $\fd_{u,J}^N$ denoting the normal Cauchy--Riemann operator introduced in \autoref{Def_NormalCauchyRiemannOperator}.
  Since
  \begin{equation*}
    \delbar_{u,J}\circ \rd u = \rd u \circ \delbar_{T\Sigma}
    \qandq
    \sT\!u \iso \sT\Sigma(D),
  \end{equation*}
  it follows that
  \begin{equation*}
    \fd_{u,J}^T = \delbar_{Tu} \qandq \dagger = 0.
  \end{equation*}
  Denote by $\iota \co T_j\sS \to \Omega^{0,1}(\Sigma,u^*TM)$ the restriction of $\rd_{u,j}\delbar_J$ to $T_j\sS$.
  The tangent space to the Teichmüller space $\sT(\Sigma)$ at $[j]$ can be identified with $\coker \delbar_{T\Sigma} \iso \ker \delbar_{T\Sigma}^*$.
  With respect to this identification,
  $\iota$ is the restriction of $\rd u\co T\Sigma \to u^*TM$ to $\ker \delbar_{T\Sigma}^*$.
  Consequently,
  we can write $\rd_{u,j}\delbar_J \co \Gamma(Tu)\oplus T_j\sS \oplus \Gamma(Nu) \to \Gamma(Tu)\oplus\Gamma(Nu)$ as
  \begin{equation*}
    \rd_{u,j}\delbar_J
    =
    \begin{pmatrix}
      \delbar_{Tu} & \iota & *\\
      0 & 0 & \fd_{u,J}^N
    \end{pmatrix}.
  \end{equation*}

  The short exact sequence
  \begin{equation*}
    0 \to \sT\Sigma \to \sT\!u \to \sO_D \to 0
  \end{equation*}
  induces the following long exact sequence in cohomology
  \begin{equation*}
    0 \to H^0(\sT\Sigma) \to H^0(\sT\!u) \to H^0(\sO_D) \to H^1(\sT\Sigma) \to H^1(\sT\!u) \to 0.
  \end{equation*}
  It follows that
  \begin{equation*}
    \ind \delbar_{Tu}
    = 2\chi(\sT\!u)
    = 2\chi(\sT\Sigma) + 2h^0(\sO_D)
    = \ind \delbar_{T\Sigma} + 2Z(\rd u),
  \end{equation*}
  and, moreover,
  that $\ker \delbar_{T\Sigma} \to \ker \delbar_{Tu}$ is injective,
  and $\coker \delbar_{T\Sigma} \to \coker \delbar_{Tu}$ is surjective.
  The latter implies that $\delbar_{Tu} \oplus \iota$ is surjective.
  Therefore,
  there are an exact sequence
  \begin{equation*}
    0 \to \ker \delbar_{Tu} \oplus \iota \to \ker \rd_{u,j}\delbar_J  \to \ker \fd_{u,J}^N \to 0,
  \end{equation*}
  and an isomorphism
  \begin{equation*}
    \coker \rd_{u,j}\delbar_J \iso \coker \fd_{u,J}^N.
  \end{equation*}
  The kernel of $\delbar_{Tu} \oplus \iota$ contains $\aut(\Sigma,j) = \ker \delbar_{T\Sigma}$ and
  \begin{align*}
    \dim\ker \delbar_{Tu}\oplus\iota
    &= \ind \delbar_{Tu}\oplus\iota \\
    &= \ind \delbar_{Tu} + \dim T_j\sS \\
    &= \ind \delbar_{T\Sigma} + \dim T_j\sS + 2Z(\rd u) \\
    &= \dim \aut(\Sigma,j) + 2Z(\rd u).
  \end{align*}
  This completes the proof of \autoref{Prop_DDelbarVSDN}.
\end{proof}

%%% Local Variables:
%%% mode: latex
%%% TeX-master: "EquivariantBrillNoetherSuperRigidity"
%%% End:

  % !TEX root = EquivariantBrillNoetherSuperRigidity.tex

\section{Orbifold Riemann--Roch formula}
\label{Sec_OrbifoldRiemannRoch}

The purpose of this section is to prove \autoref{Prop_IndexOfNormalCROperatorOfBranchedCover}.
The proof relies on Kawasaki's orbifold Riemann--Roch theorem \cite{Kawasaki1979} and a result due to \citet{Ohtsuki1982}.
The Riemann--Roch theorem for complex orbifolds is not easy to digest;
however, for orbifold Riemann surfaces it simplifies significantly and can be proved by an elementary argument based on the discussion in \cites[Section 1]{Furuta1992}[Section 1B]{Nasatyr1995}[Sections 8(ii)--(iii)]{Kronheimer1995}.

This argument relies on the following local considerations.
Let $\rho\co \mu_k \to \GL(V)$ be a representation and let $\mu_k$ act on $\bD\times V$ via $\zeta\cdot(z,v) \coloneq (\zeta z,\rho(\zeta)v)$.
\begin{equation*}
  V_\rho \coloneq [(\bD\times V)/\mu_k]
\end{equation*}
is a vector bundle over $[\bD/\mu_k]$.
In fact, up to isomorphism, every vector bundle over $[\bD/\mu_k]$ is of this form.
$V_\rho$ and $V_\sigma$ are isomorphic if and only if the representations $\rho$ and $\sigma$ are.
However,
if $\rho$ is a complex representation,
then the restriction $V_\rho$ to $\dot\bD \coloneq [(\bD\setminus\set{0})/\mu_k]$ is trivial.
This is a consequence of the fact that $\GL_r(\C)$ is connected;
more concretely, it can be seen as follows.
Choose an isomorphism $V\iso \C^r$ with respect to which $\rho$ is diagonal;
that is:
\begin{equation}
  \label{Eq_RhoAsMatrix}
  \rho(\zeta)
  =
  \begin{pmatrix}
    \zeta^{w_1} & \\
    & \ddots & \\
    & & \zeta^{w_r}
  \end{pmatrix}
\end{equation}
for $w \in (\Z/k\Z)^r$.
A choice of lift of $w$ to $\tilde w \in \Z^r$ extends $\rho$ to a representation $\hat\rho \co \C^* \to \GL(V)$.
The (inverse of the) map $\eta\co V_\rho|_{\dot\bD} \to \dot\bD\times V$ defined by
\begin{equation}
  \label{Eq_LocalHeckeModification}
  \eta([z,v]) \coloneq [z,\hat\rho(z^{-1})v]
\end{equation}
trivializes $V_\rho$ over $\dot\bD$.
The trivial bundle over $[\bD/\mu_k]$ with fiber $V$ and the bundle $V_\rho \to [\bD / \mu_k]$ have canonical holomorphic structures.
Denote by $\sV$ and $\sV_\rho$ the corresponding sheaves of holomorphic orbifold sections.
The map $\eta$ is holomorphic with respect to the canonical holomorphic structures.
If $\tilde w$ is chosen in $(-k,0]^r$,
then $\eta$ induces a sheaf morphism $\eta\co \sV_\rho \to \sV$.
To see this,
observe that if $s$ is a germ of a section of $\sV_\rho$ at $[0]$,
then $\eta(s)$ is bounded and thus defines a germ of a section of $\sV$ at $[0]$.
Evidently, $\eta$ is injective.
Furthermore, it fits into the exact sequence
\begin{equation}
  \label{Eq_VRhoExactSequence}
  \sV_\rho \incl \sV \onto V/V^\rho\otimes \sO_0.
\end{equation}
Here $\sO_0$ denotes the structure sheaf of the point $[0]$. 
To see this it suffices to consider the case $r = 1$.
A germ of a section of $\sV^\rho$ at $[0]$ is nothing but a germ of a holomorphic map $s\co \bD \to \C$ which is $\mu_k$--equivariant;
that is: $s(\zeta z) = \zeta^w s(z)$ for every $\zeta \in \mu_k$.
The map $\eta$ is given by $(\eta s)(z) \coloneq z^{-\tilde w}s(z)$.
If $w = 0$, then $\eta$ is the identity and the final map in \eqref{Eq_VRhoExactSequence} is trivial.
If $w \neq 0$, then the final map in \eqref{Eq_VRhoExactSequence} is the evaluation map at $0$.
The Taylor expansion of a germ of a $\mu_k$--invariant holomorphic map $t\co \bD \to \C$ involves powers of $z^k$.
Therefore, if $t$ vanishes at $0$,
then $s \coloneq z^{\tilde w}t$ is a germ of a $\mu_k$--equivariant holomorphic map such that $\eta s = t$.
(Here it is crucial that $\tilde w \geq -k$.)

\begin{definition} 
  \label{Def_HeckeModification}
  Let $(\Sigma,j)$ be a Riemann surface with a multiplicity function $\nu$ and let $\sE = (E,\delbar)$ be a holomorphic vector bundle over $\Sigma$.
  Define $(\Sigma_\nu,j_\nu)$ and $\beta_\nu \co (\Sigma_\nu,j_\nu) \to (\Sigma,j)$ as in \autoref{Def_OrbifoldRiemannSurface} and set $\sE_\nu \coloneq \beta_\nu^*\sE$.
  Let $\rho = \paren{\rho_x \co \mu_{\nu(x)} \to \GL(E_x)}_{x \in Z_\nu}$ be a collection of representations.
  A \defined{Hecke modification of $\sE_\nu$ of type $\rho$} consists of a holomorphic vector bundle $\sE_{\nu,\rho}$ over $\Sigma_\nu$ together with a holomorphic map
  \begin{equation*}
    \eta\co \sE_{\nu,\rho}|_{\Sigma_\nu\setminus Z_\nu} \to \sE_\nu|_{\Sigma_\nu\setminus Z_\nu}
  \end{equation*}
  such that for every $x \in Z_\nu$ with respect to suitable holomorphic trivializations of $\sE_{\nu,\rho}$ and $\sE_\nu$ around $x$ the map $\eta$ is of the form \autoref{Eq_LocalHeckeModification} with $\rho = \rho_x$ and $\tilde w \in (-k,0]^r$.
\end{definition}

\begin{remark}
  It is evident from the preceding discussion that every holomorphic vector bundle on $(\Sigma_\nu,j_\nu)$ can be obtained by a Hecke modification.
\end{remark}

\begin{theorem}[{Orbifold Riemann--Roch Formula \cite{Kawasaki1979}}]
  \label{Thm_KawasakiRiemannRoch}
  In the situation of \autoref{Def_HeckeModification},
  \begin{equation*}
    \chi(\sE_{\nu,\rho})
    =
    \chi(\sE_\nu) - \sum_{x\in Z_\nu} \dim_\C (E_x/E_x^{\rho_x}).
  \end{equation*}
\end{theorem}

\begin{proof}
  The exact sequence \autoref{Eq_VRhoExactSequence} induces the exact sequence
  \begin{equation*}
    \sE_{\nu,\rho}
    \incl
    \sE_\nu
    \onto
    \bigoplus_{x \in Z_\nu} (E_x/E_x^{\rho_x}) \otimes \sO_x.
  \end{equation*}
  This immediately implies the assertion.
\end{proof}

The proof of \autoref{Prop_IndexOfNormalCROperatorOfBranchedCover} requires one more piece of preparation.
In the situation of \autoref{Def_HeckeModification},
if $\sE_{\nu,\rho}$ carries a holomorphic flat connection $\nabla_{\nu,\rho}$,
then it induces a meromorphic flat connection $\nabla_\nu$ on $\sE$ with simple poles.
With respect to suitable local holomorphic coordinates and trivializations around $x$
\begin{equation*}
  \nabla_\nu = \rd + \Res_x(\nabla_\nu)\frac{\rd z}{z}
  \qwithq
  \Res_x(\nabla_\nu) \coloneq \frac{1}{\nu(x)}
  \begin{pmatrix} 
    \tilde w_1(x) & & \\
    & \ddots & \\
    & & \tilde w_r(x)
  \end{pmatrix}.
\end{equation*}
Here $\tilde w_i(x)$ are as in the discussion preceding \autoref{Def_HeckeModification}.
By (a very special case of) \cite[Theorem 3]{Ohtsuki1982},
the degree of $E$ and the residues $\Res_x(\nabla_\nu)$ are related by
\begin{equation}
  \label{Eq_DegreeResidueFormula}
  \deg E = -\sum_{x\in Z_\nu} \tr \mathrm{Res}_x(\nabla) = -\sum_{x\in Z_\nu}  \sum_{i=1}^r \frac{\tilde w_i(x)}{\nu(x)}.
\end{equation}

\begin{proof}[Proof of \autoref{Prop_IndexOfNormalCROperatorOfBranchedCover}]  
  Set $V^\C \coloneq V \otimes \C$ and $\uV^\C \coloneq \uV \otimes \underline\C$.
  For every $x \in Z_\nu$ denote by $\rho_x^\C \co \mu_{\nu(x)} \to \GL_\C(V^\C)$ the complexification of the monodromy representation of $\uV$ around $x$.
  There is a holomorphic vector bundle $\sV$ over $\Sigma$ such that $\uV^\C \iso \sV_{\nu,\rho^\C}$.
  Equip $E$ with the holomorphic structure $\delbar$ satisfying $\fd = \delbar + \fn$ with $\fn \in \Omega^{0,1}(\Sigma,\overline\End_\C(E))$.
  
  By \autoref{Thm_KawasakiRiemannRoch} and the classical Riemann--Roch formula,
  \begin{align*}
    \ind \fd_\nu^\uV
    &=
      2\chi(\sE_\nu \otimes_{\,\C} \uV^\C) \\
    & =
      2\chi(\sE_\nu \otimes_{\,\C} \sV)
      - 2\rk_{\C} E \sum_{z \in Z_\nu} \dim(V/V^{\rho_x}) \\
    &=
      \dim V \ind \fd
      + 2\rk_{\C} E \,\paren[\Big]{\deg \sV - \sum_{z \in Z_\nu} \dim(V/V^{\rho_x})}.
  \end{align*}
  Therefore, it remains prove that
  \begin{equation*}
    \deg \sV = \frac12\sum_{z \in Z_\nu} \dim(V/V^{\rho_x}).
  \end{equation*}

  Let $k \in \N$.
  The complexification of the trivial representation $\rho_0\co \mu_k \to \GL(\R)$ is the trivial representation $\rho_0^\C\co \mu_k\to \GL_\C(\C)$.
  Therefore, the corresponding weight in $(-k,0]$ is $0$.
  For $w \in \Z/k\Z$ the complexification of the representation $\rho_w\co \mu_k \to \GL(\C)$ defined by $\rho_w(\zeta) \mapsto \zeta^w$ is the representation $\rho_w^\C\co \mu_k \to \GL_\C(\C^2)$ defined by
  \begin{equation*}
    \rho^\C(\zeta)
    \coloneq
    \begin{pmatrix}
      \zeta^w & \\
      & \zeta^{-w}
    \end{pmatrix}.
  \end{equation*}
  Therefore, the corresponding weights in $(-k,0]$ are of the form $\tilde w$ and $-(\tilde w+k)$.
  It follows from this discussion that for every representation $\mu_k \to \GL(V)$ the sum of the weights of the complexification is $-\frac{k}{2}\dim (V/V^\rho)$.
  This combined with \autoref{Eq_DegreeResidueFormula} proves the desired identity for $\deg \sV$.
\end{proof}
%%% Local Variables:
%%% mode: latex
%%% TeX-master: "EquivariantBrillNoetherSuperRigidity"
%%% End:

\end{appendices}

\printreferences

\end{document}

%%% Local Variables:
%%% mode: latex
%%% TeX-master: t
%%% Tex-engine: xetex
%%% End: